\documentclass[11pt]{article}

\usepackage[a4paper,margin=1.1in]{geometry}
\usepackage[T1]{fontenc}
\usepackage{lmodern}
\usepackage{amsmath,amssymb,amsthm,mathtools}
\usepackage{booktabs,array,longtable}
\usepackage{microtype}
\usepackage[hidelinks]{hyperref}
\usepackage{enumitem}
\hypersetup{pdftitle={Higher Reciprocity, Cassels Pairings, and Selmer Towers for the 3/5 Congruent Number Problem},
 pdfauthor={Kaisheng Lei and Shisong Xu},
 pdfkeywords={theta-congruent numbers, Cassels pairings, higher reciprocity, ternary quadratic forms, Selmer groups, higher descent, governing fields},
 bookmarksnumbered=true}
\numberwithin{equation}{section}
\allowdisplaybreaks[2]
\newtheorem{theorem}{Theorem}[section]
\newtheorem{proposition}[theorem]{Proposition}
\newtheorem{lemma}[theorem]{Lemma}
\newtheorem{corollary}[theorem]{Corollary}
\newtheorem{conjecture}[theorem]{Conjecture}
\newtheorem{problem}[theorem]{Problem}
\theoremstyle{remark}
\newtheorem{remark}[theorem]{Remark}
\newtheorem{example}[theorem]{Example}

\DeclareFontFamily{U}{wncy}{}
\DeclareFontShape{U}{wncy}{m}{n}{<->wncyr10}{}
\DeclareFontShape{U}{wncy}{b}{n}{<->wncyb10}{}
\DeclareSymbolFont{cyrletters}{U}{wncy}{m}{n}
\SetSymbolFont{cyrletters}{bold}{U}{wncy}{b}{n}
\DeclareMathSymbol{\Sha}{\mathord}{cyrletters}{"58}

\newcommand{\Q}{\mathbb Q}
\newcommand{\Z}{\mathbb Z}
\newcommand{\F}{\mathbb F}
\newcommand{\Sh}{\Sha}
\newcommand{\Sel}{\operatorname{Sel}}
\newcommand{\Pf}{\operatorname{Pf}}
\newcommand{\Gen}{\operatorname{Gen}}
\newcommand{\Spn}{\operatorname{Spn}}
\newcommand{\Cl}{\operatorname{Cl}}
\newcommand{\rad}{\operatorname{rad}}
\newcommand{\Norm}{\operatorname N}
\newcommand{\Tr}{\operatorname{Tr}}
\newcommand{\disc}{\operatorname{disc}}
\newcommand{\rank}{\operatorname{rank}}
\newcommand{\ord}{\operatorname{ord}}
\newcommand{\leg}[2]{\left(\frac{#1}{#2}\right)}
\newcommand{\D}{\mathcal D}
\newcommand{\eps}{\varepsilon}

\title{Higher Reciprocity, Cassels Pairings, and Selmer Towers\\
for the $3/5$ Congruent Number Problem}
\author{Kaisheng Lei\quad Shisong Xu\thanks{Corresponding author: Shisong Xu
(\href{mailto:shsxu@smail.nju.edu.cn}{shsxu@smail.nju.edu.cn}).}\\[0.35em]
\small Department of Mathematics, Nanjing University\\
\small 22 Hankou Road, Nanjing, Jiangsu 210093, People's Republic of China\\
\small Email: \href{mailto:kaishenglei@smail.nju.edu.cn}{kaishenglei@smail.nju.edu.cn} (Kaisheng Lei)\\
\small Email: \href{mailto:shsxu@smail.nju.edu.cn}{shsxu@smail.nju.edu.cn} (Shisong Xu)}
\date{}

\begin{document}
\maketitle

\begin{abstract}
We study the arithmetic of the elliptic curves
\[
 A_m:y^2=x(x-m)(x+4m)
\]
attached to the $3/5$ congruent number problem. A difference of ternary
representation numbers controls the relevant central $L$-values. For
$p\equiv11\pmod{40}$ the ordinary Cassels pairing degenerates; we construct
an explicit $4$-cover and show that the next Cassels--Tate pairing is governed
by the normalized representation defect, equivalently by a factorial
character, a Pell symbol, and a class number congruence. For composite
parameters we compute the ordinary Cassels matrices of four twists and, in
the two prime case, a degree $1024$ governing field for their joint
distribution. The same higher descent extends to larger radicals: a second
rational pushout determines the full $\Lambda'$ row of the next pairing. We
also determine two explicit Selmer towers with four dimensional ordinary
radical, and all finite $2$-power Selmer groups when the ordinary radical is
one dimensional.
\end{abstract}



\setcounter{tocdepth}{2}
\clearpage
\tableofcontents
\clearpage

\section{Introduction}
\label{sec:introduction}

Let $\cos\theta=3/5$. A positive integer $N$ is $\theta$-congruent if a
triangle with rational side lengths has an angle $\theta$ and area $4N$.
The associated elliptic curves can be written
\[
 E_{5,3}^{(N)}:y^2=x(x-2N)(x+8N),\qquad
 A_m:y^2=x(x-m)(x+4m).
\]
The change of variables $(x,y)\mapsto(4x,8y)$ identifies $A_m$ with
$E_{5,3}^{(2m)}$; thus $A_{5n}$ corresponds to the geometric integer $10n$.
These curves have $j$-invariant $148176/25$ and are non-CM.

Tunnell's theorem \cite{Tunnell} placed ternary quadratic forms at the
heart of the classical congruent number problem. A decisive further step was
taken by Qin \cite{Qin}, who developed a systematic arithmetic approach based
on the forms themselves. In Qin's work, genus identities and congruences of
representation numbers are not merely devices for reading Fourier
coefficients; they become arithmetic tools for extracting information about
the underlying elliptic curves. This pioneering point of view is a principal
starting point for the present paper.

For the variants with cosine $\pm3/5$ and $\pm4/5$, Im--Shin \cite{ImShin}
constructed the relevant half-integral weight forms and established
Tunnell type central value criteria. Their unweighted $3/5$ row provides the
modular input for this paper: it singles out the two ternary forms from which
our arithmetic begins. We determine the genus and spinor genus of these forms,
replace them by diagonal representatives, and study the difference of their
representation numbers. The resulting defect admits congruences and class
number formulas of its own. It also has a second, quite different
interpretation on the Selmer side, through ordinary Cassels pairings and a
higher Cassels--Tate pairing.

The quadratic form analysis occupies the first part of the paper. For
squarefree parameters the defect controls the relevant central $L$-value; for
prime parameters it satisfies explicit class number congruences. The main
reciprocity theorem enters at the first point where ordinary $2$-descent loses
information. When $p\equiv11\pmod{40}$, the ordinary Cassels pairing on the
pure $2$-Selmer group vanishes. We construct a $4$-cover of one Selmer class,
identify the pushout divisor class, and evaluate the resulting Cassels--Tate
pairing. Its Pfaffian is the normalized representation defect, and the same
sign is expressed successively as a factorial character, a Pell symbol, and a
class number congruence. Thus the representation defect continues to detect
$2$-primary arithmetic one step beyond the ordinary Cassels pairing. For the
present non-CM family, this is the higher descent analogue of the ordinary
comparisons for classical congruent number curves in \cite{XuDefect}.

The remaining sections place these calculations in families. Fixing
$p\equiv11,19\pmod{40}$ and varying an auxiliary prime $q$, we find a
single Galois extension of degree $1024$ that governs the ordinary Cassels
matrices of four related twists and hence their joint distribution. With
several auxiliary primes, the same matrices are described uniformly by norm
conics. On a natural composite locus the prime case $4$-cover remains valid
even when the auxiliary block is singular. A second pushout accounts for the
other rational $2$-torsion direction and completes the $\Lambda'$ row of the
next pairing. Two four dimensional examples can then be carried through to
the full finite Selmer tower; when the ordinary radical is one dimensional,
the entire tower is determined abstractly.

Several natural questions are left open. Most notably, the proposed identity
between $\D(pq)/32$ and the ordinary Pfaffian remains
Conjecture~\ref{conj:pqbridge}. We also do not know a positive definite
quaternary model of Qin type for this comparison, nor a distribution theorem
for the higher character as $p$ varies. These questions are logically
separate from the higher descent and governing field results proved here.

For a positive definite integral form $Q$, let $r_Q(n)$ count all
integral representations, including signs and zero coordinates.
Put
\begin{equation}\label{eq:introFG}
 F=x^2+2y^2+40z^2,\qquad G=x^2+8y^2+10z^2,\qquad
 \D(n)=r_F(n)-r_G(n).
\end{equation}
For a fundamental discriminant $D$, write
$\Cl(D)=\Cl(\Q(\sqrt D))$ for the ordinary ideal class group;
when $D<0$, put $h(D)=\#\Cl(D)$. Our first main theorem describes the arithmetic of this representation defect.

\begin{theorem}
\label{thm:introA}\label{thm:introC}
The genus and spinor genus of $F$ both consist of $F$ and $G$.
For every positive squarefree $t\equiv3\pmod8$ with $5\nmid t$,
\begin{equation}\label{eq:introL}
 r_F(t)+r_G(t)=2h(-5t),\qquad
 L(A_{5t},1)=\frac{\sqrt3\,L(A_{15},1)}{16\sqrt t}\D(t)^2,
\end{equation}
where $L(A_{15},1)>0$. For primes $p\equiv11,19\pmod{40}$,
\[
 \D(p)\equiv8h(-p)-2h(-5p)\pmod{32}.
\]
In particular, if $p\equiv11\pmod{40}$, then $16\mid\D(p)$ and
\[
 \frac{\D(p)}{16}\equiv\Xi(p)\pmod2,\qquad
 \Xi(p):=\left(\frac{4h(-p)-h(-5p)}8\bmod2\right)\in\F_2.
\]
\end{theorem}

Section~\ref{sec:forms} proves the quadratic form assertions. The only
modular input is the Im--Shin central value formula; the genus computation and
the representation congruences are established directly for the forms above.
In particular, $\D(t)\ne0$ forces rank zero. For
$p\equiv11\pmod{40}$ this already yields the criterion
$h(-5p)\not\equiv4h(-p)\pmod{16}$ for $10p$ to be non-$\theta$-congruent.

Write $\Sel_2^0(E)$ for the $2$-Selmer group modulo the Kummer
image of rational $2$-torsion. We use the additive Cassels pairing
with values in $\F_2$; the conventions and its relation to
$4$-descent are recalled in Section~\ref{sec:compositedescent}.

\begin{theorem}
\label{thm:introB}\label{thm:introD}
For a prime $p\equiv11,19\pmod{40}$, the space
$\Sel_2^0(A_{5p})$ has basis
$\Lambda=(5,1,5)$, $\Lambda'=(1,5,5)$. Its Cassels matrix is
\[
 C_p=\begin{pmatrix}0&\eps_p\\\eps_p&0\end{pmatrix},\qquad
 \eps_p=\begin{cases}0,&p\equiv11\pmod{40},\\
             1,&p\equiv19\pmod{40},\end{cases}
 \qquad \D(p)/8\equiv\eps_p\pmod2.
\]
Suppose $p\equiv11\pmod{40}$. Choose coprime positive integers
$a,b$ with $p=a^2-5b^2$, $4\mid a$, and $b$ odd, and put
$k=(-1)^{(b-1)/2}$. The conic
\begin{equation}\label{eq:introConic}
 c^2=k\bigl((a+5b)u^2+10(a+b)uv+5(a+5b)v^2\bigr)
\end{equation}
has a rational point. For any primitive integral point
$(u_0,v_0,c_0)$ on it, the sign
$\varrho(p)=k(c_0/5)$ is defined and independent of all choices.
The next pairing, defined through a $4$-Selmer lift, has matrix
\begin{equation}\label{eq:introHigher}
 C_p^{\mathrm{next}}=\begin{pmatrix}0&b_p\\b_p&0\end{pmatrix},
 \qquad (-1)^{b_p}=-\varrho(p).
\end{equation}
Every positive integral solution of $5d^2-py^2=1$ gives
\[
 \varrho(p)=\leg{((p-1)/5)!}{p}=(-1)^{(d-1)/2},\qquad
 \frac{\D(p)}{16}\equiv b_p\equiv\frac{d+1}{2}\pmod2.
\]
Such a solution always exists, and the residue $d\pmod4$ is independent
of its choice. Moreover,
$h(-5p)-4h(-p)\equiv4(d+1)\pmod{16}$.
If $\eps_p=1$, then $\rank A_{5p}(\Q)=0$ and
$\Sh(A_{5p})[2^\infty]\simeq(\Z/2\Z)^2$.
If $b_p=1$, then $\rank A_{5p}(\Q)=0$ and
$\Sh(A_{5p})[2^\infty]\simeq(\Z/4\Z)^2$.
\end{theorem}

The ordinary pairing is computed in Section~\ref{sec:ordinary} from tangent
forms and Hilbert symbols. Section~\ref{sec:higher} then constructs the
$4$-cover, identifies its pushout class, and evaluates the next pairing
locally. The class number identity is obtained from the two factorizations
$20p=(-4)(-5p)=(-20)(-p)$ and Zagier's normalization of ordinary class groups
\cite{Zagier}.

We then let an auxiliary prime vary.
For a fixed prime $p\equiv11,19\pmod{40}$, let
\[
 \mathcal P_p=\{q\text{ prime}:q\equiv1\pmod{40},\ (p/q)=1\}.
\]
All densities below are natural densities relative to $\mathcal P_p$,
unless stated otherwise. This prime set itself has density $1/32$
among all primes.

\begin{theorem}
\label{thm:introF}\label{thm:introG}\label{thm:introH}
Let $q\in\mathcal P_p$ and $E_\delta=A_{5\delta pq}$ for
$\delta\in\{1,-1,2,-2\}$. Their pure $2$-Selmer dimensions are
$(s_1,s_{-1},s_2,s_{-2})=(4,4,2,3)$, and their ordinary Cassels
matrices are explicit functions of five additive residue characters.
Choose $p=A^2-5B^2=A_2^2+2B_2^2$ with $4\mid A$ and $B$ odd.
The Galois field
\[
 \widehat L_p=\Q\bigl(\zeta_{40},\sqrt[4]p,
 \sqrt{A+B\sqrt5},\sqrt{2+\sqrt5},\sqrt[4]2,
 \sqrt{A_2+B_2\sqrt{-2}}\bigr)
\]
has degree $1024$ and governs all four matrices. Their $32$
distinct ordered joint values each have relative density $1/32$.
If $R_\delta$ is the rank of the corresponding matrix, then
\[
 \Sel_4(E_\delta)\simeq
 (\Z/4\Z)^{s_\delta-R_\delta}\oplus(\Z/2\Z)^{R_\delta+2}.
\]
The first three matrices are simultaneously nondegenerate with
relative density $1/16$. On this set all three curves have rank
zero, with $2$-primary Tate--Shafarevich groups of respective
orders $2^4,2^4,2^2$ and exponent two, and
$\rank E_{-2}(\Q)\le1$. The fourth matrix has rank zero with
relative density $1/8$ and rank two with relative density $7/8$.
The latter condition always implies $\rank E_{-2}(\Q)\le1$.
\end{theorem}

Section~\ref{sec:compositedescent} carries out the local descent and computes
the four ordinary Cassels matrices; Section~\ref{sec:governing} constructs
the governing field and derives the density statements. The pure Selmer
space for $E_{-2}$ has odd dimension, so its Cassels matrix is necessarily
degenerate and has no Pfaffian. The same section obtains the congruence for
$\D(pq)/32$ and isolates its conjectural comparison with the ordinary
Pfaffian as Conjecture~\ref{conj:pqbridge}. Appendix~\ref{app:classfields}
gives the complementary class field and Jacobi sum interpretations.

For composite parameters the norm conic description also opens the way to
higher descent. The prime case $4$-cover extends to a larger zero character
locus on which the auxiliary block may be singular, and a second pushout
supplies a complete row of the next pairing. In a different direction, the
case of a one dimensional ordinary radical is governed entirely by the
structure of finite abelian $2$-groups, which determines every finite
$2$-power Selmer group.

\begin{theorem}
\label{thm:introI}
Let
\[
\begin{gathered}
 n=pq_1\cdots q_r,\qquad p\equiv11,19\pmod{40},\\
 q_i\equiv1\pmod{40},\qquad
 \leg p{q_i}=\leg{q_i}{q_j}=1\quad(i\ne j).
\end{gathered}
\]
For each $\delta\in\{1,-1,2,-2\}$, the ordinary Cassels matrix of
$A_{5\delta n}$ is given explicitly by rational norm equations for the
auxiliary primes; see Theorem~\ref{thm:allfour}.

Suppose now that $p\equiv11\pmod{40}$, choose
\[
 n=a^2-5b^2,\qquad a,b>0,\qquad4\mid a,\qquad b\text{ odd},
\]
and put $k=(-1)^{(b-1)/2}$, $\alpha_i=[a/q_i]$ and
$\tau_i=[5/q_i]_4$. If $\alpha_i=\tau_i=0$ for every $i$, then
\[
 C_+=0_2\oplus H_1,\qquad
 \rad C_+=\langle\Lambda,\Lambda'\rangle\oplus\ker H_1,
\]
and the same explicit $4$-cover lifts $\Lambda'$ without assuming that
$H_1$ is nondegenerate. On this radical the complete $\Lambda'$ row of the
next pairing is explicit: the $\Lambda$ entry is
\[
 (-1)^{\mathcal B(\Lambda',\Lambda)}=-k\leg{c_0}{5},
\]
and Theorems~\ref{thm:auxcrossU} and~\ref{thm:auxcrossV} give the auxiliary
$U_i$- and $V_i$-entries by Legendre symbols attached to two rational conics.
If $H_1$ is nondegenerate, this specializes to the two dimensional radical
theorem and determines the complete $8$-Selmer group and the corresponding
$2$-primary Tate--Shafarevich group.

There are also four dimensional ordinary radicals for which the next layer
can be determined completely. In particular,
\[
\begin{array}{c|c|c}
E&\rank E(\Q)&\Sh(E)[2^\infty]\\ \hline
A_{213455}&2&(\Z/4\Z)^2\\
A_{2172655}&2&(\Z/4\Z)^2
\end{array}
\]
and for both curves
\[
 \Sel_{2^j}(E)\simeq
 (\Z/2^j\Z)^2\oplus(\Z/2^{\min(j,2)}\Z)^2\oplus(\Z/2\Z)^2
 \qquad(j\ge1),
\]
with the images in the pure $2$-Selmer group explicitly determined in
Theorems~\ref{thm:ranktwo42691} and~\ref{thm:ranktwo434531}.
Finally, if $E=A_m$, $s=\dim\Sel_2^0(E)$, and its ordinary Cassels matrix
has rank $s-1$, then for every $j\ge1$,
\[
 \Sel_{2^j}(E)\simeq\Z/2^j\Z\oplus(\Z/2\Z)^{s+1}.
\]
For $E=A_{-10pq}$ this entire tower pattern has relative density $7/8$ in
$\mathcal P_p$.
\end{theorem}

Section~\ref{sec:extensions} contains these composite calculations. On larger
radicals we determine one complete row of the next pairing; in the two
rank two examples, independent rational Kummer classes determine the
remaining entries. None of these arguments uses
Conjecture~\ref{conj:pqbridge}.

We use $(a/\ell)$ for the Legendre symbol at an odd prime and for
the Jacobi symbol at a positive odd composite denominator. Symbols
with denominator $2$ are Kronecker symbols. When $\ell\equiv1\pmod4$
is prime and $(a/\ell)=1$, the rational quartic symbol is the sign
$(a/\ell)_4\equiv a^{(\ell-1)/4}\pmod\ell$. Additive symbols
are defined by $(-1)^{[a/\ell]}=(a/\ell)$ and
$(-1)^{[a/\ell]_4}=(a/\ell)_4$. All pairing matrices and their
Pfaffians are over $\F_2$. We identify a residue in $\F_2$ with
its representative in $\{0,1\}$ when it occurs as an exponent of $-1$.
The notation $D_4$ denotes the dihedral group of order eight.

\section{Ternary representation defects and central \texorpdfstring{$L$}{L}-values}
\label{sec:forms}\label{sec:representations}

\subsection{Theta series input and genus structure}\label{subsec:theta-genus}

We start from the two ternary forms singled out by the Im--Shin theta series.
The first task is to determine their genera and then replace them, on the
relevant residue classes, by simpler diagonal forms. Throughout, representation
numbers include signs and zero coordinates. If
$Q(v)=v^tMv$ is ternary, we call $\det M$ its determinant; its level is the
least positive integer $N$ for which $N(2M)^{-1}$ is integral with even
diagonal.

Set
\begin{align*}
 Q_1&=15x^2+23y^2+23z^2-10xy-10xz+14yz,\\
 Q_2&=7x^2+23y^2+47z^2-2xy+6xz+22yz.
\end{align*}
In the notation of \cite{ImShin}, the unweighted row
$G^{(-2)}_{3;+,-}$ has auxiliary parameter $l=1$. It contains the forms
$Q_1,Q_2$ above and the generalized theta series
\begin{equation}\label{eq:SIrow}
 \mathcal G(q)
 =
 2(\theta_{Q_1}-\theta_{Q_2})
 =
 \sum_{m\ge1}a(m)q^m,
 \qquad
 a(m)=2(r_{Q_1}(m)-r_{Q_2}(m)).
\end{equation}
For the target twist $E_{5,3}^{(2m)}$ this row contains
$m\equiv7,15,23\pmod{40}$.

\begin{proposition}\label{prop:genera}
The forms $Q_1,Q_2$ satisfy
\[
 \Gen(Q_1)=\Spn(Q_1)=\{Q_1,Q_2\},
 \qquad |O(Q_1)|=|O(Q_2)|=4,
\]
and have determinant $6400$ and level $160$.
The diagonal forms $F,G$ in \eqref{eq:introFG} satisfy
\[
 \Gen(F)=\Spn(F)=\{F,G\},
 \qquad |O(F)|=|O(G)|=8,
\]
and have determinant $80$ and level $160$.
With mass computed as the sum of reciprocal automorphism orders,
one has
\[
 \operatorname{mass}(\Gen Q_1)=\frac12,\qquad
 \operatorname{mass}(\Gen F)=\frac14.
\]
\end{proposition}

\begin{proof}
The Gram matrices of $Q_1,Q_2$ are
\[
 M_1=
 \begin{pmatrix}
 15&-5&-5\\
 -5&23&7\\
 -5&7&23
 \end{pmatrix},
 \qquad
 M_2=
 \begin{pmatrix}
 7&-1&3\\
 -1&23&11\\
 3&11&47
 \end{pmatrix}.
\]
Their Smith invariants are $(1,80,80)$.

Over $\Z_2$, orthogonal splitting along an odd unit gives
\[
 M_1\sim\langle15\rangle\perp
 \frac{16}{3}
 \begin{pmatrix}4&1\\1&4\end{pmatrix},
 \qquad
 M_2\sim\langle7\rangle\perp
 \frac{80}{7}
 \begin{pmatrix}2&1\\1&4\end{pmatrix}.
\]
Both binary unit blocks are even hyperbolic blocks, and the one
dimensional blocks lie in the same squareclass $-1$; hence both local
lattices are equivalent to $\langle-1\rangle\perp16\mathbb H$.
At $5$, both consist of a nonsquare unit line plus $5$ times an
anisotropic binary unit lattice. At every other odd prime they are
unimodular ternary lattices with the same determinant, hence locally
equivalent.

We next determine the global classes in the genus. For a reduced positive
definite ternary form, the Seeber--Minkowski bound \cite{Mahler} gives
\[
 abc\le 2\det M.
\]
After sign changes one may impose
\[
 1\le a\le b\le c,\qquad
 0\le r,s\le a/2,\qquad |t|\le b/2,
\]
for
\[
 M=\begin{pmatrix}a&r&s\\r&b&t\\s&t&c\end{pmatrix}.
\]
These inequalities need not be a complete reduction theory for our purpose;
they give a finite set containing every reduced representative.
From $\det M=6400$,
\[
 c=
 \frac{6400+at^2+bs^2-2rst}{ab-r^2}.
\]
Imposing primitivity, the Smith invariants $(1,80,80)$, and the same
odd represented residue class modulo $8$, leaves the four candidates
\[
\begin{array}{c|rrrrrr}
 &a&b&c&r&s&t\\ \hline
1&7&12&80&2&0&0\\
2&7&23&47&1&3&-11\\
3&12&23&28&6&4&2\\
4&15&23&23&5&5&7.
\end{array}
\]
Splitting off the odd line in the first and third candidates
gives binary complements
\[
 16\operatorname{diag}(5/7,5),\qquad
 \frac{16}{23}\begin{pmatrix}15&5\\5&40\end{pmatrix},
\]
respectively. These scale-$16$ blocks are odd, so both candidates
are excluded by the even block in $16\mathbb H$. The second and fourth are equivalent to
$Q_2,Q_1$. The minima $7$ and $15$ show that these two classes are
inequivalent. More explicitly,
\[
 O(Q_i)=\{I,-I,S_i,-S_i\},\qquad
 S_1=\begin{pmatrix}1&0&0\\0&0&1\\0&1&0\end{pmatrix},\quad
 S_2=\begin{pmatrix}1&0&1\\0&1&1\\0&0&-1\end{pmatrix}.
\]
The same reduction also determines the automorphism groups. If $v$ is the
$j$th column of an automorphism, then
$v^tM_iv=(M_i)_{jj}$ and
$|v_l|^2\le (M_i)_{jj}(M_i^{-1})_{ll}$. These bounds leave only finitely
many columns, and their mutual inner products give the groups displayed
above.

For $F$ and $G$, at odd primes one may use the diagonal local
transformation with multiplier $\operatorname{diag}(1,2,1/2)$.
Over $\Z_2$, choose $u\in\Z_2^\times$ with $u^2=-3/5$ and put
\[
 T=\begin{pmatrix}4&-5u\\u&1\end{pmatrix},
 \qquad \det T=1.
\]
Then
\[
 T^t\operatorname{diag}(1,20)T=\operatorname{diag}(4,5),
\]
which gives the required local equivalence. Applying the same finite
reduction argument with determinant $80$ and Smith invariants
$(1,2,40)$ and odd represented residues $\{1,3\}$ modulo $8$ leaves only
\[
 (1,2,40;0,0,0),\qquad (1,8,10;0,0,0),
\]
namely $F$ and $G$. They are inequivalent since
$r_F(3)=4$ while $r_G(3)=0$.
The stated levels follow by inverting the four Gram matrices.

It remains to compare genus and spinor genus. We do this through the local
proper spinor norm groups.
At $2$, in the model $\langle-1\rangle\perp16\mathbb H$,
reflection in $(0,1,u)$, with $u$ odd, is integral and has spinor
norm $32u$. Composing with reflection in the first coordinate
gives the proper spinor norm $-2u$. Products of these elements
give all unit squareclasses. At $5$, the anisotropic binary
$5$-modular block represents $5$ times every unit squareclass;
products of two corresponding reflections give all unit spinor
norms. For $F$ at $2$, composing reflection in $e_1$ with those
in $e_2,e_3,e_1+e_2$ gives the squareclasses $2,10,3$, which
generate $\Q_2^\times/\Q_2^{\times2}$. Its unimodular binary
block at $5$ similarly gives every unit spinor norm. At every
other odd prime the unimodular ternary lattice has all unit
spinor norms. The idelic description of proper spinor genera
therefore reduces to the narrow ideal class group of $\Q$,
which is trivial. Each genus consists of a single spinor genus.
\end{proof}

\subsection{Diagonalization, genus averages, and central values}\label{subsec:diag-central}

The forms $Q_1,Q_2$ are natural from the modular construction, whereas the
congruence arguments are much clearer for diagonal forms. On the residue
classes relevant to the $3/5$ problem, the two descriptions are related as
follows.

\begin{proposition}\label{prop:diag}
For every positive integer $t\equiv3\pmod8$,
\begin{equation}\label{eq:diag}
 2r_{Q_1}(5t)=r_F(t),\qquad
 2r_{Q_2}(5t)=r_G(t).
\end{equation}
Consequently
\[
 a(5t)=\D(t).
\]
\end{proposition}

\begin{proof}
For $Q_1$ set
\[
 u=x+y+z,\qquad v=x-y-z,\qquad w=y-z.
\]
Then
\[
 Q_1=5u^2+10v^2+8w^2.
\]
The original integral lattice is described by
\[
 u\equiv v\pmod2,\qquad u-v\equiv2w\pmod4.
\]
If $Q_1=5t$, reduction modulo $5$ gives $5\mid w$.
Writing $w=5z'$ yields
\[
 t=u^2+2v^2+40z'^2.
\]
Since $t\equiv3\pmod8$, both $u$ and $v$ are odd. The involution
$v\mapsto-v$ pairs the resulting triples, and in each pair exactly one triple
satisfies the original congruence modulo $4$. Hence
$2r_{Q_1}(5t)=r_F(t)$.

For $Q_2$, modulo $5$ one has
\[
 Q_2\equiv2(x+2y+4z)^2.
\]
If $Q_2=5t$, put
\[
 u=x+y+z,\qquad v=-z,\qquad
 w=\frac{x-3y-z}{5}\in\Z.
\]
A direct computation gives
\[
 Q_2=5u^2+40v^2+50w^2
\]
and
\[
 x=\frac{3u+2v+5w}{4},\qquad
 y=\frac{u+2v-5w}{4},\qquad z=-v.
\]
For $t\equiv3\pmod8$, $u,w$ are odd, and integrality is equivalent to
\[
 u+2v-5w\equiv0\pmod4.
\]
Changing $w$ to $-w$ interchanges the two residue classes modulo $4$.
Thus each pair contains a unique point of the original lattice, proving the
second equality in \eqref{eq:diag}.
\end{proof}

\begin{proposition}\label{prop:genusaverage}
Let $m>0$ be squarefree with $m\equiv7\pmod8$. Then
\begin{equation}\label{eq:Qsum}
 r_{Q_1}(m)+r_{Q_2}(m)
 =
 \left(1-\leg{m}{5}\right)h(-m).
\end{equation}
Consequently, if $t>0$ is squarefree,
$t\equiv3\pmod8$, and $5\nmid t$, then
\begin{equation}\label{eq:FGsum}
 r_F(t)+r_G(t)=2h(-5t),
 \qquad
 \D(t)=2\bigl(r_F(t)-h(-5t)\bigr).
\end{equation}
\end{proposition}

\begin{proof}
For $Q_1$ define the local density
\[
 \alpha_\ell(m)
 =
 \lim_{k\to\infty}\ell^{-2k}
 \#\{v\bmod\ell^k:Q_1(v)\equiv m\pmod{\ell^k}\}.
\]
Since $|O(Q_1)|=|O(Q_2)|$, the Siegel mass formula \cite{Siegel} gives
\[
 \frac{r_{Q_1}(m)+r_{Q_2}(m)}2
 =
 \frac{2\pi\sqrt m}{80}
 \prod_\ell\alpha_\ell(m).
\]
Put $\chi=\chi_{-m}$. For $\ell\nmid10$,
nonsingular lifting over the finite field gives
\[
 \alpha_\ell(m)=
 \begin{cases}
 1+\chi(\ell)/\ell,&\ell\nmid m,\\
 1-\ell^{-2},&\ell\mid m,
 \end{cases}
\]
or uniformly
\[
 \alpha_\ell(m)
 =
 \frac{1-\ell^{-2}}{1-\chi(\ell)/\ell}.
\]
At $2$, use $Q_1\sim -x^2+32yz$. For every pair $(y,z)$
modulo $2^k$, the equation $x^2\equiv32yz-m\pmod{2^k}$ has
four roots when $k\ge3$, since its right hand side is $1$
modulo $8$. Thus $\alpha_2(m)=4$. At $5$ the lattice is a nonsquare unit line plus
$5$ times an anisotropic binary norm form, hence
\[
 \alpha_5(m)=
 \begin{cases}
 2,&5\nmid m,\ \leg{m}{5}=-1,\\
 0,&5\nmid m,\ \leg{m}{5}=1,\\
 6/5,&5\mid m.
 \end{cases}
\]
For the ramified case $m=5u$, reduction first forces the unit line
coordinate to vanish modulo $5$. After division by $5$, the
binary norm equation over $\F_5$ has six solutions at every
nonzero level; nonsingular lifting gives $6/5$. In the
unramified cases, the unit line has two roots or none.
Because $\chi(2)=1$,
\[
 \prod_{\ell\ne2,5}\alpha_\ell(m)
 =
 \frac{L(1,\chi)}{\zeta(2)}
 \frac{1-1/2}{1-1/4}
 \frac{1-\chi(5)/5}{1-1/25}.
\]
Using
\[
 h(-m)=\frac{\sqrt m}{\pi}L(1,\chi)
\]
gives
\[
 \frac5{12}\alpha_5(m)
 \left(1-\frac{\chi(5)}5\right)h(-m).
\]
Substitution of the three possible $5$-adic densities gives, respectively,
$h(-m)$, $0$, and $h(-m)/2$, which is \eqref{eq:Qsum}. Setting $m=5t$
and applying \eqref{eq:diag} yields \eqref{eq:FGsum}.
\end{proof}

\begin{corollary}\label{cor:Lvalue}
Let $t>0$ be squarefree, $t\equiv3\pmod8$, and $5\nmid t$.
Then
\[
 L(A_{5t},1)
 =
 \frac{\sqrt3\,L(A_{15},1)}{16\sqrt t}\D(t)^2,
 \qquad L(A_{15},1)>0.
\]
Thus $\D(t)\ne0$ implies $\rank A_{5t}(\Q)=0$, and $10t$ is not
$\theta$-congruent for $\cos\theta=3/5$.
\end{corollary}

\begin{proof}
The central value formula of Im--Shin
\cite[Theorem 6.1]{ImShin}, applied to the unweighted row
\eqref{eq:SIrow}, has a constant local factor for $m=5t$.
Since $a(5t)=\D(t)$ by \eqref{eq:diag} and
$A_{5t}\simeq_\Q E_{5,3}^{(10t)}$, it gives
\[
 L(A_{5t},1)=\frac{2\kappa}{\sqrt5}\frac{\D(t)^2}{\sqrt t}
\]
with a constant $\kappa>0$ independent of $t$.
At $t=3$ one has $r_F(3)=4$ and $r_G(3)=0$, so
$\D(3)=4$. Evaluating the same formula at $t=3$ gives
\[
 \frac{2\kappa}{\sqrt5}=\frac{\sqrt3\,L(A_{15},1)}{16},
\]
and hence fixes the normalization. Its positivity follows from
$L(A_{15},1)>0$. The rank statement is the rank zero consequence used in
\cite[Theorem 6.3]{ImShin}.
\end{proof}

\subsection{The eta product identity}\label{subsec:eta-identity}

The defect also has a useful product expansion. Besides giving a compact
formula for its coefficients, this identity provides an independent check on
the congruence patterns that appear below.

\begin{proposition}\label{prop:eta}
With $q=e^{2\pi iz}$ and
$\vartheta(q)=\sum_{x\in\Z}q^{x^2}$,
\[
 \theta_F(z)-\theta_G(z)
 =
 2\vartheta(q)\eta(8z)\eta(40z).
\]
If
\[
 E(t)=\prod_{j\ge1}(1-t^j),
\]
then for every $k\ge0$,
\[
 \D(8k+3)
 =
 4[t^k]\bigl(E(t^2)^2E(t^5)\bigr).
\]
\end{proposition}

\begin{proof}
First prove
\[
 \vartheta(q)\vartheta(q^{20})
 -
 \vartheta(q^4)\vartheta(q^5)
 =
 2\eta(4z)\eta(20z).
\]
Both sides lie in
$M_1(\Gamma_0(80),\chi_{-80})$. The eta product criterion gives the
same weight, level and character on the right. Since
\[
 [\mathrm{SL}_2(\Z):\Gamma_0(80)]=144,
\]
the Sturm bound is $12$. Including the constant term, both sides
have expansion
\[
 2q-2q^5-2q^9+O(q^{13}),
\]
so the identity follows from Sturm's theorem \cite{Sturm}.
Replace $z$ by $2z$ and multiply by $\vartheta(q)$.

Finally
\[
 \eta(8z)\eta(40z)=q^2E(q^8)E(q^{40}),
\]
and the odd square part of $\vartheta$ is
$2q\psi(q^8)$, where
\[
 \psi(t)=\sum_{j\ge0}t^{j(j+1)/2}.
\]
Gauss' product identity
$\psi(t)E(t)=E(t^2)^2$ gives the coefficient formula.
\end{proof}

\subsection{Prime and composite representation congruences}\label{subsec:repcongruences}

We turn next to congruences for the defect. The argument uses only three
arithmetic ingredients: norm representations in $\Q(\sqrt5)$, genus theory,
and elementary formulas for binary representation numbers.

\begin{lemma}\label{lem:norm5}
Let $n>0$ be squarefree, $n\equiv3\pmod8$, and suppose every
prime divisor of $n$ splits in $\Q(\sqrt5)$. There are coprime
positive integers $a,b$ such that
\[
 n=a^2-5b^2,\qquad4\mid a,\qquad b\text{ odd}.
\]
\end{lemma}
\begin{proof}
Put $K=\Q(\sqrt5)$ and $\phi=(1+\sqrt5)/2$.
The Minkowski bound $\sqrt5/2<2$ shows that $K$ has class
number one. Choose one prime above each prime factor of $n$;
their product has an integral generator $\alpha$ of norm
$\pm n$. Since $\Norm(\phi)=-1$, multiplication by a unit
and a sign makes $\alpha$ totally positive of norm $n$.
The prime $2$ is inert, and the image of $\phi^2$ generates
$(\mathcal O_K/2\mathcal O_K)^\times\simeq\F_4^\times$.
Multiplying by a power of the totally positive norm one unit
$\phi^2$, we may therefore arrange
$\alpha\equiv1\pmod{2\mathcal O_K}$. This means
$\alpha=a+b\sqrt5$ with $a,b\in\Z$ of opposite parity.
The equation $n\equiv3\pmod8$ forces $b$ odd and $4\mid a$.
Total positivity gives $a>0$, and conjugating if necessary
gives $b>0$. Finally, a common divisor of $a,b$ would have
its square dividing the squarefree integer $n$.
\end{proof}

\begin{lemma}\label{lem:h5mod8}
If $p\equiv11,19\pmod{40}$ is prime, then
\[
 h(-5p)\equiv
 \begin{cases}
 4\pmod8,&p\equiv11\pmod{40},\\
 0\pmod8,&p\equiv19\pmod{40}.
 \end{cases}
\]
\end{lemma}

\begin{proof}
Let $K=\Q(\sqrt{-5p})$. Genus theory gives $2$-rank one for
$\Cl(K)$. Let $\mathfrak p_5$ be the ramified prime above $5$.
Using $p=a^2-5b^2$, put
\[
 \alpha=\frac{5b+\sqrt{-5p}}2,
 \qquad
 \Norm(\alpha)=5\left(\frac a2\right)^2.
\]
The coprimality of $a,b$ gives
\[
 (\alpha)=\mathfrak p_5\mathfrak a^2,
 \qquad
 \Norm\mathfrak a=|a|/2,
 \qquad
 [\mathfrak a]^2=[\mathfrak p_5].
\]
The class $[\mathfrak p_5]$ is nontrivial: a generator of norm $5$
would give $x^2+5py^2=20$ with integers $x,y$, which is impossible.
Thus $4\mid h(-5p)$. Since the Sylow $2$-subgroup is cyclic,
$8\mid h(-5p)$ if and only if
$[\mathfrak a]\in\Cl(K)^2$. The unique nontrivial genus character is
\[
 \chi_5([\mathfrak a])
 =
 \leg{a/2}{5}
 =
 -\leg{a}{5}
 =
 -\left(\frac p5\right)_4.
\]
This distinguishes the two cases. This divisibility is classical as well; see, for example, \cite{Morton}.
\end{proof}

Set
\[
 T=x^2+10y^2+40z^2,\qquad
 H=x^2+10y^2+10z^2,\qquad
 r_3(n)=r_{x^2+y^2+z^2}(n).
\]

\begin{lemma}\label{lem:HT}
If $n>0$, $n\equiv3\pmod8$, and
$n\bmod5\in\{1,4\}$, then
\[
 r_H(n)=\frac13r_3(n),
 \qquad
 r_T(n)=\frac12r_H(n)=\frac16r_3(n).
\]
No squarefreeness assumption is required.
\end{lemma}

\begin{proof}
Modulo $8$, every representation by $H$ has $x$ odd and exactly one
of $y,z$ even. Exchanging $y,z$ pairs the two possibilities, while
the part with $z$ even is exactly the set of representations by $T$.
Hence $r_H=2r_T$.

Use
\[
 x^2+10y^2+10z^2
 =
 x^2+(3y+z)^2+(y-3z)^2.
\]
Put
\[
 A=3y+z,\qquad B=y-3z.
\]
Then
\[
 y=\frac{3A+B}{10},\qquad
 z=\frac{A-3B}{10}.
\]
In any three square representation of
$n\equiv3\pmod8$, all three coordinates are odd, so the integrality
condition is simply
\[
 A\equiv3B\pmod5.
\]
Let $s=n\bmod5\in\{1,4\}$. The three coordinate squares modulo $5$
have one of the two multisets
\[
 \{s,0,0\},\qquad \{s,s,-s\}.
\]
In the first case, one third of the ordered representations place the
unique entry $s$ in the $x^2$ coordinate; then $A\equiv B\equiv0\pmod5$.
In the second case, two thirds have $x^2=s$, and for these
$A/B=\pm2\pmod5$. The involution $A\mapsto-A$ exchanges the two signs,
so half of them satisfy $A/B=3$. In either case one third of the
three square representations satisfy the required lattice condition.
\end{proof}

\begin{corollary}\label{cor:Tclass}
If $n>0$ is squarefree and
$n\equiv11,19\pmod{40}$, then
\[
 r_T(n)=4h(-n).
\]
\end{corollary}

\begin{proof}
For squarefree $n\equiv3\pmod8$, Gauss' three square class number
formula gives
\[
 r_3(n)=24h(-n)
\]
\cite[Section 1]{OnoSkinner}. Apply Lemma~\ref{lem:HT}.
\end{proof}

For $c>0$ put
\[
 B_c(N)=\#\{(x,y)\in\Z^2:x^2+cy^2=N\}.
\]

\begin{lemma}\label{lem:binary}
If $N>0$, $N\equiv3\pmod8$, and
$N\bmod5\in\{1,4\}$, then
\[
 B_2(N)=2\sum_{d\mid N}\chi_{-8}(d),
 \qquad
 B_{10}(N)=2\sum_{d\mid N}\chi_{-40}(d),
\]
and
\[
 B_2(N)\equiv B_{10}(N)\pmod8.
\]
\end{lemma}

\begin{proof}
The discriminant $-8$ has one reduced primitive positive class,
represented by $x^2+2y^2$, giving the first formula by ideal norm
counting.

For discriminant $-40$ the two reduced classes are
\[
 x^2+10y^2,\qquad2x^2+5y^2.
\]
They form the two genera. Since
$\leg{N}{5}=1$, the second form cannot represent $N$ modulo $5$;
all ideals of norm $N$ lie in the principal class, giving the second
formula.

Since $N\equiv3\pmod8$, $N$ is not a square. Pair the positive
divisors as $\{d,N/d\}$. For each of
$\chi_{-8},\chi_{-40}$ one has $\chi(N)=1$, hence
\[
 \chi(N/d)=\chi(d),\qquad
 \chi(d)+\chi(N/d)\equiv2\pmod4.
\]
Thus both divisor sums are congruent modulo $4$ to the number
$d(N)$ of positive divisors of $N$, and
the claimed congruence follows after multiplication by $2$.
\end{proof}

\begin{proposition}\label{prop:FT}
If $n\equiv11,19\pmod{40}$, then
\[
 r_F(n)-r_T(n)
 \equiv
 B_2(n)-B_{10}(n)\pmod{16}.
\]
In particular, if $p$ is prime in these residue classes, then
\[
 r_F(p)\equiv4h(-p)\pmod{16}.
\]
\end{proposition}

\begin{proof}
Split according to the third coordinate:
\[
\begin{aligned}
r_F(n)-r_T(n)
={}&B_2(n)-B_{10}(n)\\
&+2\sum_{\substack{z\ge1\\40z^2<n}}
 \left(B_2(n-40z^2)-B_{10}(n-40z^2)\right).
\end{aligned}
\]
Every positive $n-40z^2$ has the same residue classes modulo $8$ and
$5$ as $n$. Lemma~\ref{lem:binary} makes every summand divisible by
$8$, hence the second line is divisible by $16$.

If $n=p$ is prime, then
$B_2(p)=B_{10}(p)=4$, so the boundary term vanishes, and
Corollary~\ref{cor:Tclass} finishes the proof.
\end{proof}

\begin{theorem}\label{thm:secondbit}
For primes $p\equiv11,19\pmod{40}$,
\[
 \D(p)\equiv8h(-p)-2h(-5p)\pmod{32}.
\]
If $p\equiv11\pmod{40}$, then
\[
 \Xi(p)=\left(\frac{4h(-p)-h(-5p)}8\bmod2\right)\in\F_2
\]
is well defined and
\[
 \frac{\D(p)}{16}\equiv\Xi(p)\pmod2.
\]
\end{theorem}

\begin{proof}
From \eqref{eq:FGsum},
\[
 \D(p)=2(r_F(p)-h(-5p)).
\]
Apply Proposition~\ref{prop:FT}. If
$p\equiv11\pmod{40}$, genus theory gives $h(-p)$ odd, while
Lemma~\ref{lem:h5mod8} gives
$h(-5p)\equiv4\pmod8$. Thus
$4h(-p)-h(-5p)$ is divisible by $8$, and division of the congruence
modulo $32$ gives the second assertion.
\end{proof}

\begin{corollary}\label{cor:classcriterion}
If $p\equiv11\pmod{40}$ and
\[
 h(-5p)\not\equiv4h(-p)\pmod{16},
\]
then
\[
 \rank A_{5p}(\Q)=0,
\]
and $10p$ is not $\theta$-congruent.
\end{corollary}

\begin{proof}
The displayed class number congruence says that $\Xi(p)=1$.
Theorem~\ref{thm:secondbit} then gives $\D(p)\ne0$, and the conclusion
follows from Corollary~\ref{cor:Lvalue}.
\end{proof}

For composite $n$ the same argument leaves one additional binary boundary
term, which can still be evaluated explicitly.

\begin{theorem}\label{thm:composite}
Let $n>0$ be squarefree with
$n\equiv11,19\pmod{40}$. Then
\[
 \D(n)\equiv
 8h(-n)-2h(-5n)
 +2B_2(n)-2B_{10}(n)
 \pmod{32}.
\]
If $n$ has $k$ distinct prime factors and
\[
 e_{-c}(n)=
 \begin{cases}
 1,&\leg{-c}{q}=1\text{ for every }q\mid n,\\
 0,&\text{otherwise},
 \end{cases}
 \qquad c=2,10,
\]
then
\[
 2B_2(n)-2B_{10}(n)
 =
 2^{k+2}\bigl(e_{-2}(n)-e_{-10}(n)\bigr).
\]
Thus the correction vanishes modulo $32$ for $k=1$ and for $k\ge3$;
for $k=2$ it may be $16$ modulo $32$.
\end{theorem}

\begin{proof}
Combine \eqref{eq:FGsum}, Corollary~\ref{cor:Tclass}, and
Proposition~\ref{prop:FT}. For squarefree $n$, the divisor sum
formulas factor as
\[
 B_2(n)=2\prod_{q\mid n}(1+\chi_{-8}(q))
   =2^{k+1}e_{-2}(n),
\]
and similarly for $B_{10}(n)$.
\end{proof}

\begin{example}
For $n=51=3\cdot17$,
\[
 r_F(51)=16,\quad r_G(51)=8,\quad
 h(-51)=2,\quad h(-255)=12,
\]
while
\[
 B_2(51)=8,\qquad B_{10}(51)=0.
\]
The uncorrected class number expression is $-8$ modulo $32$ whereas
$\D(51)=8$; the correction $16$ is therefore essential.
\end{example}

\section{Selmer descent and ordinary Cassels pairings}
\label{sec:compositedescent}\label{sec:ordinary}\label{sec:fourtwists}

We now pass from representation numbers to descent. This section first gives
a uniform matrix description of the $2$-Selmer groups of the twists $A_m$
and then computes the ordinary Cassels pairings that will enter the reciprocity
arguments.

\subsection{Selmer groups and descent matrices}\label{subsec:descentmatrices}

For an elliptic curve $E/\Q$ and an integer $r\ge2$, write
\[
 \Sel_r(E)=\ker\left(H^1(\Q,E[r])\longrightarrow
          \prod_v H^1(\Q_v,E)\right).
\]
The local maps are induced by $E[r]\hookrightarrow E$, and the
product runs over all places of $\Q$. The Kummer exact sequence is
\begin{equation}\label{eq:kummerexact}
 0\longrightarrow E(\Q)/rE(\Q)\longrightarrow\Sel_r(E)
 \longrightarrow\Sh(E)[r]\longrightarrow0.
\end{equation}
For $E=A_m$, $m\ne0$, no nonzero rational $2$-torsion point is
twice a rational point: the halving criterion requires both root
differences to be squares, whereas at $0$ they have opposite
signs, and at $m$ and $-4m$ their ratios are $5$ and $5/4$.
Hence the image $\mathcal T$ of $E(\Q)[2]$ in $\Sel_2(E)$ has dimension
two, and we set
\[
 \Sel_2^0(E)=\Sel_2(E)/\mathcal T.
\]
Its dimension is $\rank E(\Q)+\dim_{\F_2}\Sh(E)[2]$.
The rational torsion group of these twists is in fact
$E(\Q)[2]$; see \cite[Lemma 2.2]{ImShin}.

For $A_m$, the Kummer map away from the $2$-torsion points is
\[
 (x,y)\longmapsto(x,x-m,x+4m)
 \quad\text{in }(\Q^\times/\Q^{\times2})^3.
\]
The product of the three coordinates is a square. We often omit
the third coordinate, which is the product of the first two.
For a squareclass triple $d=(d_1,d_2,d_3)$, the corresponding
$2$-cover is the intersection
\begin{equation}\label{eq:general2cover}
 d_1u_1^2-d_2u_2^2=mt^2,\qquad
 d_3u_3^2-d_1u_1^2=4mt^2.
\end{equation}
At the rational $2$-torsion points, the usual limiting values give
the generators $t_0=(-1,-m)$ and $t_m=(m,5)$.

Write $[a,b]_v\in\F_2$ for the additive Hilbert symbol, so that
$(a,b)_v=(-1)^{[a,b]_v}$. We use the additive residue symbol conventions from the introduction.
The Cassels pairing on $\Sel_2^0(E)$ is the pairing induced
by the Cassels--Tate pairing on $\Sh(E)[2]$, with
$\tfrac12\Z/\Z$ identified with $\F_2$. It is alternating.
Its radical is the image of $\Sel_4(E)$ in $\Sel_2^0(E)$
\cite{CasselsSecond,Fisher}. If $C$ is its matrix in a basis,
then $\Pf(C)$ is taken over $\F_2$ whenever the dimension is
even.

\begin{lemma}\label{lem:nondegenerate}
Let $s=\dim_{\F_2}\Sel_2^0(A_m)$. The ordinary Cassels pairing
is nondegenerate if and only if
\[
 \rank A_m(\Q)=0,\qquad
 \Sh(A_m)[2^\infty]\simeq(\Z/2\Z)^s.
\]
\end{lemma}
\begin{proof}
The image of the Mordell--Weil group lies in the radical, so
nondegeneracy forces rank zero. By the divisibility property
of the Cassels--Tate pairing \cite[Theorem 3.1]{Fisher}, the
radical on $\Sh[2]$ is $2\Sh[4]$. Thus a nonzero element
of order four would give a nonzero radical element. There are
no such elements, so the $2$-primary group is the finite group
$\Sh[2]$. The converse follows from the perfect alternating
pairing on a finite primary component; equivalently, the same
divisibility criterion makes the radical zero.
\end{proof}

Multiplication by $2$ on $E[4]$ induces a map
$\Sel_4(E)\to\Sel_2(E)$. Modulo the torsion Kummer classes, its image is
the radical of the Cassels pairing \cite{CasselsSecond,Fisher}.

\begin{proposition}\label{prop:sel4structure}
Let $E=A_m$ for a nonzero rational $m$. Write
$s=\dim_{\F_2}\Sel_2^0(E)$ and let $d$ be the dimension of
the radical of its ordinary Cassels pairing. Then, as abstract
finite abelian groups,
\begin{equation}\label{eq:sel4structure}
 \Sel_4(E)\simeq(\Z/4\Z)^d\oplus(\Z/2\Z)^{s+2-d},
 \qquad \rank E(\Q)\le d.
\end{equation}
No finiteness assumption on $\Sh(E)$ is required.
\end{proposition}
\begin{proof}
By the halving criterion discussed above, $E(\Q)[4]=E(\Q)[2]$ and the torsion Kummer image
$\mathcal T\subset\Sel_2(E)$ has dimension two.

The cohomology sequence for $0\to E[2]\to E[4]\to E[2]\to0$,
with the local Kummer conditions, gives the exact sequence
\[
 0\longrightarrow \mathcal T\longrightarrow\Sel_2(E)
 \xrightarrow{i}\Sel_4(E)\xrightarrow{\pi}R\longrightarrow0,
\]
where $R$ is the full Cassels radical in $\Sel_2(E)$ and
$\dim R=d+2$. Hence $|\Sel_4(E)|=2^{s+d+2}$.
Multiplication by $2$ on $\Sel_4(E)$ is $i\circ\pi$.
Since $\mathcal T\subset R$, its image has order $2^d$.
Every group killed by $4$ is a sum of cyclic groups of orders
$4$ and $2$; its doubled subgroup determines the number of
order four summands. Comparing orders proves
\eqref{eq:sel4structure}. Finally the image of $E(\Q)/2E(\Q)$
modulo $\mathcal T$ lies in the pure radical and has dimension equal to
the Mordell--Weil rank.
\end{proof}

Let $n>0$ be squarefree, $(n,10)=1$, with distinct prime divisors
$q_1,\ldots,q_k$, and set $m=5n$. In the first two Kummer
coordinates, the nonzero rational $2$-torsion classes are generated by
\[
 t_0=(-1,-m),\qquad t_m=(m,5).
\]
The local images have the following bases:
\begin{equation}\label{eq:alllocalbases}
\begin{array}{c|l}
v&\text{basis of }\delta_v(A_m(\Q_v)/2A_m(\Q_v))\\ \hline
\infty&t_0\\
\ell\mid5n&t_0,\ t_m\\
2,\ m\not\equiv3\pmod8&t_0,\ t_m,\ (5,1)\\
2,\ m\equiv3\pmod8&t_0,\ t_m,\ (2,-1).
\end{array}
\end{equation}
The table is to be read place by place: its entries are local squareclasses,
so identical symbols at different places do not represent a common global
class.

Put $\mathcal S=(-1,2,5,q_1,\ldots,q_k)$ and $h=k+3$.
Define a $(2k+6)\times(2k+6)$ matrix $\mathcal M_n$ as follows.
For each basis element $(a_v,b_v)$ in \eqref{eq:alllocalbases},
include the row
\begin{equation}\label{eq:fullmonskyrow}
 \bigl([s,b_v]_v:s\in\mathcal S\,;\,[s,a_v]_v:s\in\mathcal S\bigr).
\end{equation}
Thus the first and last $h$ columns record the exponents in $d_1$
and $d_2$, respectively. The entries are evaluated by the standard Hilbert symbol formulas. If
$a=\ell^r u$, $b=\ell^s w$ with $u,w$ units, then
\[
 [a,b]_\ell=rs[-1/\ell]+r[w/\ell]+s[u/\ell]
 \quad(\ell\text{ odd}),
\]
while for odd $u,w$,
\[
 [2^r u,2^s w]_2
 =\frac{u-1}{2}\frac{w-1}{2}
  +r\frac{w^2-1}{8}+s\frac{u^2-1}{8}\pmod2.
\]
At infinity the symbol is one exactly when both arguments are
negative. Consequently the matrix uses only residues at $2,5$ and
Legendre symbols between the prime divisors of $n$.

\begin{theorem}\label{thm:fullmonsky}
There is a natural isomorphism
\[
 \Sel_2(A_{5n})\simeq\ker\mathcal M_n.
\]
In particular
\[
 \dim_{\F_2}\Sel_2^0(A_{5n})
 =2k+4-\rank_{\F_2}\mathcal M_n.
\]
\end{theorem}
\begin{proof}
The discriminant of $A_m$ is supported on $2,5,n$. Thus a Selmer
class has both coordinates in the group generated by $\mathcal S$.
We verify the displayed local bases. At an odd prime dividing $5n$,
the valuation vectors of $t_0,t_m$ are $(0,1),(1,0)$ for
$\ell\mid n$, and $(0,1),(1,1)$ for $\ell=5$.
They are independent, and the local quotient has dimension two.
At $2$ the torsion classes are independent and the local quotient
has dimension three. The value $x=5/4$ gives a $\Q_2$-point, since
\[
 x(x-m)(x+4m)=5(5-4m)(5+16m)/64
\]
has even valuation and odd part $1\pmod8$. Its Kummer image is
$(5,1)$, independent of the torsion images unless $m\equiv3\pmod8$.
In that remaining case $x=2$ works, because
\[
 x(x-m)(x+4m)=4(2-m)(1+2m)
\]
has odd part $1\pmod8$. Its image $(2,-1)$ is independent of the
odd torsion images. The real image consists of the two equal sign
patterns, generated by $t_0$.

On pairs of local squareclasses use the bilinear form
\[
 ((d_1,d_2),(a,b))\longmapsto[d_1,b]_v+[d_2,a]_v.
\]
The Hilbert formulas above show directly that the displayed local
basis vectors are pairwise orthogonal. Their spans have half the
dimension of the local ambient spaces, respectively $1,2,3$ at
$\infty$, odd primes, and $2$. Nondegeneracy of the Hilbert pairing
therefore identifies each span with its orthogonal complement.
Membership in the local image is exactly the vanishing of the rows
\eqref{eq:fullmonskyrow}. At good odd primes the coordinates are
units, and the unramified local condition is automatic. It follows that the rows in \eqref{eq:fullmonskyrow} are precisely the local
Selmer conditions. The two global rational $2$-torsion classes are
independent, so quotienting by them lowers the dimension by two.
\end{proof}

For $\delta\in\{1,-1,2,-2\}$, set $E_\delta(n)=A_{5\delta n}$.
Let $n>0$ be odd and squarefree, $(n,5)=1$, with prime divisors
$q_1,\ldots,q_k$, and put $m=5\delta n$. Keep
$\mathcal S=(-1,2,5,q_1,\ldots,q_k)$ and the torsion classes
$t_0=(-1,-m)$, $t_m=(m,5)$. Define
\begin{equation}\label{eq:fourdyadic}
 v_m=
 \begin{cases}
 (5,1),&m\text{ odd},\ m\not\equiv3\pmod8,\\
 (2,-1),&m\equiv3\pmod8,\\
 (m+1,1),&\ord_2(m)=1.
 \end{cases}
\end{equation}
At infinity use the basis $t_0$; at each odd prime dividing $5n$
use $t_0,t_m$; at $2$ use $t_0,t_m,v_m$. Form the matrix
$\mathcal M_{\delta,n}$ from these bases by the same rows
\eqref{eq:fullmonskyrow}.

\begin{theorem}\label{thm:fourdescent}
For all these $n$ and $\delta$,
\[
 \Sel_2(E_\delta(n))\simeq\ker\mathcal M_{\delta,n},
 \qquad
 \dim\Sel_2^0(E_\delta(n))=2k+4-\rank\mathcal M_{\delta,n}.
\]
In particular the formula covers both signs of the twist and both
possible valuations zero and one at $2$.
\end{theorem}
\begin{proof}
The odd prime argument of Theorem~\ref{thm:fullmonsky} is
unchanged: the valuation vectors of the two torsion images are
independent and give the full local image. At infinity $t_0$
generates the two components for either sign of $m$.
For odd $m$, the points with $x=5/4$ or $x=2$ used there work
also when $m<0$; their existence is a statement about $2$-adic
squareclasses, not real signs.

If $m=2u$ with $u$ odd, take $x=m+1$. Then
\[
 x(x-m)(x+4m)=(m+1)(5m+1)
 =1+6m+5m^2\equiv1\pmod8.
\]
It is a $2$-adic square, and the Kummer image is $(m+1,1)$.
The torsion images have valuation vectors $(0,1),(1,0)$, whereas
this point has vector $(0,0)$ and a nonsquare first coordinate:
$m+1\equiv3$ or $7\pmod8$. It is therefore independent of
torsion. Thus the displayed three vectors give the full local
image at $2$. The Hilbert formulas verify their isotropy, and
the dimension argument in Theorem~\ref{thm:fullmonsky} again
identifies local membership with the stated matrix equations.
Finally quotienting the two independent global torsion images
subtracts two from the kernel dimension.
\end{proof}

Suppose now that
\begin{equation}\label{eq:fourlargefamily}
 \begin{gathered}
 n=pq_2\cdots q_k,\quad p\equiv11,19\pmod{40},\quad
 q_i\equiv1\pmod{40},\\
 \leg{q_i}{q_j}=\leg p{q_i}=1\quad(i\ne j).
 \end{gathered}
\end{equation}
Use the first two Kummer coordinates and put
\[
 \Lambda=(5,1),\quad\Lambda'=(1,5),\quad
 U_i=(q_i,1),\quad V_i=(1,q_i),\quad W=(p,1).
\]

\begin{corollary}\label{cor:fourdimensions}
For the family \eqref{eq:fourlargefamily}, the pure Selmer
dimensions and bases are
\[
 \begin{array}{c|c|l}
 \delta&\dim\Sel_2^0(E_\delta(n))&\text{basis}\\ \hline
 1&2k&\Lambda,\Lambda',U_i,V_i\ (2\le i\le k)\\
 -1&2k&\Lambda,\Lambda',U_i,V_i\ (2\le i\le k)\\
 2&2k-2&U_i,V_i\ (2\le i\le k)\\
 -2&2k-1&W,U_i,V_i\ (2\le i\le k).
 \end{array}
\]
In particular for $n=pq$ these dimensions are $4,4,2,3$.
\end{corollary}
\begin{proof}
Each $q_i$ is a square at $2,5,p$, at infinity, and at every
other $q_j$. At $q_i$, the two local torsion classes reduce to
$(1,q_i),(q_i,1)$, since $5\delta n/q_i$ is a square unit.
Thus each additional prime contributes two freely chosen
coordinates, and introduces no condition on the coordinates
supported on $-1,2,5,p$.

It remains only to solve the conditions supported on $-1,2,5,p$; this can be
done directly, without computing a matrix rank. Modulo torsion, normalize both coordinates
to be positive and remove $5$ from the first one. Write them as
\[
 d_1=2^b p^d,\qquad d_2=2^f5^g p^h.
\]
Substitution in the local Hilbert rows gives the following independent
conditions over $\F_2$:
\[
 \begin{array}{c|l}
 \delta&\text{conditions on }(b,d,f,g,h)\\ \hline
 1&b=f=h=0\\
 -1&b=f=0,\ h=d\\
 2&b=d=f=g=h=0\\
 -2&b=f=g=h=0.
 \end{array}
\]
Here only $p\equiv3\pmod8$, $(5/p)=1$, and $(p/5)=1$
enter. For $\delta=1$, the normalized representatives of
$\Lambda,\Lambda'$ are $(p,5),(1,5)$; for $\delta=-1$
they are $(p,p),(1,5)$. The last row is generated by $(p,1)$.
Adding the independent $U_i,V_i$ proves the assertion.
\end{proof}

Suppose now that $n\equiv3\pmod8$. Define
\[
 e_i=[-1/q_i],\quad t_i=[5/q_i],\quad
 e=(e_i)^t,\quad t=(t_i)^t,
 \quad D_e=\operatorname{diag}(e_i),\quad
 D_t=\operatorname{diag}(t_i).
\]
Let $A=(a_{ij})$ be the $k\times k$ matrix
\[
 a_{ij}=[q_j/q_i]\ (i\ne j),\qquad
 a_{ii}=\sum_{j\ne i}a_{ij}.
\]
In particular $A\mathbf1=0$.

\begin{theorem}\label{thm:reducedmonsky}
Every pure $2$-Selmer class has a unique representative with
\[
 d_1=\prod_iq_i^{x_i},\qquad
 d_2=5^z\prod_iq_i^{y_i},\qquad d_3=d_1d_2
 \quad\text{in }\Q^\times/\Q^{\times2}.
\]
Under this identification,
\begin{equation}\label{eq:reducedmonsky}
 \Sel_2^0(A_{5n})\simeq\ker M_n,\qquad
 M_n=
 \begin{pmatrix}
 A+D_t&D_e&0\\
 D_t&A+D_e+D_t&t\\
 0^t&e^t&0
 \end{pmatrix}.
\end{equation}
The matrix has size $2k+1$, with column order $(x,y,z)$.
\end{theorem}
\begin{proof}
The real condition gives equal signs in $d_1,d_2$; multiplication
by $t_0$ makes both positive. Since $m\equiv7\pmod8$, the local
image at $2$ contains only odd coordinates, with the second
coordinate congruent to $1\pmod4$. Thus neither coordinate has a
factor $2$. Multiplication by $t_m=(5n,5)$ then removes the factor
$5$ from $d_1$, uniquely. This gives the normalization and the
remaining dyadic condition
\begin{equation}\label{eq:monsky2condition}
 e^ty=0.
\end{equation}

At $q_i$, the local torsion combination having valuation vector
$(x_i,y_i)$ is $t_m^{x_i}t_0^{y_i}$. Comparing the unit parts of
its two coordinates with $d_1,d_2$ gives, respectively,
\[
 (A+D_t)x+D_ey=0,\qquad
 D_tx+(A+D_e+D_t)y+tz=0.
\]
At $5$, whose local image is also generated by $t_0,t_m$, the
normalization $5\nmid d_1$ gives the conditions
\begin{equation}\label{eq:monsky5conditions}
 t^tx=0,\qquad t^ty+(\mathbf1^tt)z=0.
\end{equation}
Here quadratic reciprocity identifies $(q_i/5)$ with $(5/q_i)$.

These last two rows are redundant. Indeed, $n\equiv3\pmod4$
implies $\mathbf1^te=1$, while quadratic reciprocity gives
\[
 A+A^t=ee^t+D_e,\qquad \mathbf1^tA=0.
\]
The sum of the first $k$ equations is $t^tx+e^ty=0$;
the sum of both blocks is $t^ty+(\mathbf1^tt)z=0$.
Together with \eqref{eq:monsky2condition} these imply
\eqref{eq:monsky5conditions}. There are no further local conditions
by Theorem~\ref{thm:fullmonsky}, proving \eqref{eq:reducedmonsky}.
\end{proof}

For $n=p$ prime with $p\equiv3\pmod8$, the reduced matrix gives pure
Selmer dimensions $0,2,2,0$ in the residue classes
$p\equiv3,11,19,27\pmod{40}$, respectively. The middle two cases recover
Proposition~\ref{prop:Selmer}; the same calculation also gives the adjacent
classes.

\begin{corollary}\label{cor:largepurefamily}
Let $n=pq_2\cdots q_k$, where $p\equiv11,19\pmod{40}$,
$q_i\equiv1\pmod{40}$, all primes are distinct, and every Legendre
symbol between two distinct prime factors is $+1$. Then
\[
 \dim_{\F_2}\Sel_2^0(A_{5n})=2k.
\]
A basis, with third coordinate determined by the product, is
\[
 (5,1),\ (1,5),\ (q_i,1),\ (1,q_i)\quad(2\le i\le k).
\]
\end{corollary}
\begin{proof}
Here $A=D_t=0$ and $e=(1,0,\ldots,0)^t$, so the only equation
is $y_1=0$. The kernel has dimension $2k$. In the normalized
coordinates, $(5,1)$ is represented by $(n,5)$, and the displayed
classes give an independent spanning set. For each fixed $p$ and
$k$, such prime tuples exist: successively choose
$q_i\equiv1\pmod{40p q_2\cdots q_{i-1}}$ and apply Dirichlet's
theorem. This asserts existence, not a density formula.
\end{proof}

From now on we use the Kummer coordinates
\[
 (x,x-m,x+4m)
\]
for $A_m$.

\subsection{Prime twists and the first Cassels pairing}\label{subsec:primeordinary}

\begin{proposition}\label{prop:Selmer}
Let $p\equiv11,19\pmod{40}$ be prime. Then
\[
 \dim_{\F_2}\Sel_2^0(A_{5p})=2,
\]
with basis
\[
 \Lambda=(5,1,5),\qquad
 \Lambda'=(1,5,5).
\]
If instead $p\equiv7,23\pmod{40}$, then
\[
 \Sel_2^0(A_p)=0.
\]
\end{proposition}

\begin{proof}
For $A_{5p}$ the global squareclasses are supported on
$\{-1,2,5,p\}$. In the first two Kummer coordinates the local images
are generated by
\[
\begin{array}{c|l}
v&\text{generators}\\ \hline
\infty&(-1,-1)\\
2&(-1,1),(5,1),(1,5)\\
5&(5,1),(1,5)\\
p&(p,1),(-1,-p).
\end{array}
\]
At the odd bad primes these are already the images of rational
$2$-torsion and have the correct local dimension. At $2$, torsion
gives $(-1,1)$ and $(1,5)$, while the local point with $x=5/4$ gives
$(5,1)$ because
\[
 5-20p\equiv1\pmod8,\qquad
 1+16p\equiv1\pmod8.
\]
The $2$-adic and real conditions reduce the global Selmer group to
\[
 \Sel_2(A_{5p})
 =
 \langle(5,1),(p,1),(1,5),(-1,-p)\rangle_{\F_2}.
\]
Modulo the two dimensional image of rational $2$-torsion, represented
by $(-1,-5p)$ and $(5p,5)$, the classes $(5,1)$ and $(1,5)$ form a
basis.

For $A_p$ with $p\equiv7,23\pmod{40}$, the $2$-adic and real
conditions leave
\[
 (5,1)^c(p,1)^d(1,5)^e(-1,-p)^f.
\]
At $p$, both $5$ and $-1$ are nonsquare units, and the local image is
generated by $(p,5)$ and $(-1,-p)$. Comparing coordinates gives
$c=0$ and $e=d$, leaving exactly the rational $2$-torsion image.
\end{proof}

With these Selmer bases fixed, we can compute the ordinary Cassels pairing
explicitly.

\begin{theorem}\label{thm:firstpair}
For $p\equiv11,19\pmod{40}$, the Cassels matrix in the basis
$(\Lambda,\Lambda')$ is
\[
 C_p=
 \begin{pmatrix}
 0&\eps_p\\
 \eps_p&0
 \end{pmatrix},
 \qquad
 \eps_p=
 \begin{cases}
 0,&p\equiv11\pmod{40},\\
 1,&p\equiv19\pmod{40}.
 \end{cases}
\]
\end{theorem}

\begin{proof}
Write $p=a^2-5b^2$ with $4\mid a$ and $b$ odd.
When $p\equiv19\pmod{40}$, choose $a\equiv3\pmod5$;
multiplication of $a+b\sqrt5$ by $9+4\sqrt5$ changes $a$ to
$-a$ modulo $5$ and preserves the other conditions, so this
normalization is always possible. It fixes the choice used in the
local table below.
For $\Lambda=(5,1,5)$ the three conics may be written
\[
\begin{aligned}
H_1&:u_2^2-5u_3^2+25pt^2=0,\\
H_2&:u_3^2-u_1^2-4pt^2=0,\\
H_3&:5u_1^2-u_2^2-5pt^2=0.
\end{aligned}
\]
Use the rational conic points
\[
 (t,u_1,u_3)=(1,p-1,p+1),
 \qquad
 (t,u_1,u_2)=(1,a,5b).
\]
Tangent forms can be chosen as
\[
 L_2=(p+1)u_3-(p-1)u_1-4pt,
 \qquad
 L_3=au_1-bu_2-pt.
\]
Cassels' tangent--Hilbert symbol formula
\cite{CasselsSecond,WangZhang} gives
\[
 (-1)^{\langle\Lambda,\Lambda'\rangle}
 =
 \prod_v (L_2(P_v)L_3(P_v),5)_v.
\]
All places other than $2,5,p,\infty$ give $1$ by good reduction.
At $p$, $5$ is a local square, and the real contribution is also $1$.

At $2$ choose
\[
 P_2=(2,1,u_2,u_3),\qquad
 u_2^2=5-20p,\quad u_3^2=1+16p.
\]
Both square roots are odd $2$-adic units,
$L_2(P_2)\equiv2\pmod4$, and $L_3(P_2)$ is odd; hence the local
symbol is $-1$.

At $5$ choose $r^2=-p$ with $r\equiv2a\pmod5$ and take
$P_5=(1,2r,5r,0)$. One obtains
\[
\begin{array}{c|cc|c}
p\bmod5&L_2(P_5)\bmod5&L_3(P_5)\bmod5&
(L_2L_3,5)_5\\ \hline
1&1&3&-1\\
4&3&2&+1.
\end{array}
\]
Multiplying by the constant $2$-adic contribution proves the theorem.
\end{proof}

\begin{proposition}\label{prop:firstdefect}
For $p\equiv11,19\pmod{40}$,
\[
 8\mid\D(p),\qquad
 \frac{\D(p)}8\equiv\eps_p\pmod2.
\]
Hence
\begin{equation}\label{eq:firstbridge}
 \frac{\D(p)}8\equiv\Pf(C_p)\pmod2.
\end{equation}
If $p\equiv19\pmod{40}$, then
\[
 \rank A_{5p}(\Q)=0,\qquad
 \Sh(A_{5p})[2^\infty]\simeq(\Z/2\Z)^2.
\]
\end{proposition}

\begin{proof}
For $p\equiv3\pmod8$, the binary form $x^2+2y^2$ represents $p$
exactly four times. In a representation
$p=x^2+2y^2+40z^2$, the coordinates $x,y$ are odd; representations
with $z\ne0$ occur in orbits of size $8$ under sign changes. Thus
$r_F(p)\equiv4\pmod8$. Now use
\[
 \D(p)=2(r_F(p)-h(-5p))
\]
and Lemma~\ref{lem:h5mod8}. This is the first congruence.
The independent Cassels pairing calculation identifies this residue with
$\eps_p$.

When $p\equiv19\pmod{40}$ the Cassels pairing on the pure Selmer group
is nondegenerate. Lemma~\ref{lem:nondegenerate} gives the rank and
Tate--Shafarevich assertions.
\end{proof}

Take distinct primes satisfying
\begin{equation}\label{eq:pqfamily}
 p\equiv11,19\pmod{40},\qquad q\equiv1\pmod{40},\qquad
 \leg pq=1,\qquad n=pq.
\end{equation}
The preceding corollary gives the basis
\[
 \Lambda=(5,1,5),\quad \Lambda'=(1,5,5),\quad
 U=(q,1,q),\quad V=(1,q,q).
\]
Choose positive coprime $a,b$ such that
\begin{equation}\label{eq:pqnorm}
 n=a^2-5b^2,\qquad4\mid a,\qquad b\text{ odd}.
\end{equation}
The representation \eqref{eq:pqnorm} follows from the same norm argument as
Lemma~\ref{lem:norm5}. Indeed, $p$ and $q$ split in
$\Q(\sqrt5)$, whose class number is one; after multiplying a norm-$n$
generator by a norm one unit, it may be taken in $\Z[\sqrt5]$ with the
required parity. Conjugation and a totally positive norm one unit then give
$a,b>0$.

For either square root $r$ of $5$ modulo $q$, define
\begin{equation}\label{eq:pqbits}
 \alpha=[a/q],\qquad \tau=[(2+r)/q],\qquad
 \eps=
 \begin{cases}0,&p\equiv11\pmod{40},\\1,&p\equiv19\pmod{40}.
 \end{cases}
\end{equation}
The value of $\tau$ is independent of the root because
$(2+r)(2-r)=-1$ and $q\equiv1\pmod4$.
To define the remaining entry, choose primitive integers
$R,S,T$, with $T\ne0$, satisfying
\begin{equation}\label{eq:pqauxnorm}
 R^2-qS^2=5pT^2.
\end{equation}
Choose $w\in\F_q$ such that
\[
 w^2=R^2+20pT^2\pmod q,\qquad
 \ell_q=(5p+1)w+(5p-1)R-20pT\not\equiv0\pmod q,
\]
and put $\gamma=[\ell_q/q]$.
The proof of Theorem~\ref{thm:pqcassels} will show both that these choices
exist and that the resulting symbol $\gamma$ is intrinsic.

\subsection{Two prime twists and fixed residue characters}\label{subsec:twoprimeordinary}

\begin{theorem}\label{thm:pqcassels}
In the ordered basis $(\Lambda,\Lambda',U,V)$, the ordinary
Cassels pairing on $\Sel_2^0(A_{5pq})$ has matrix
\begin{equation}\label{eq:pqcasselsmatrix}
 C_{p,q}=
 \begin{pmatrix}
 0&\eps&0&\alpha\\
 \eps&0&\alpha+\tau&\tau\\
 0&\alpha+\tau&0&\gamma\\
 \alpha&\tau&\gamma&0
 \end{pmatrix},
 \qquad
 \Pf(C_{p,q})=\eps\gamma+\alpha(1+\tau).
\end{equation}
In particular, if $p\equiv11\pmod{40}$, then
\begin{equation}\label{eq:pqnondeg}
 \Pf(C_{p,q})=1
 \quad\Longleftrightarrow\quad
 \leg aq=-1,\qquad \leg{2+r}q=1.
\end{equation}
These equivalent conditions imply, unconditionally,
\[
 \rank A_{5pq}(\Q)=0,\qquad
 \Sh(A_{5pq})[2^\infty]\simeq(\Z/2\Z)^4.
\]
Consequently $10pq$ is not $\theta$-congruent for $\cos\theta=3/5$.
\end{theorem}
\begin{proof}
We evaluate the six entries by the same tangent--Hilbert symbol formula as in
Theorem~\ref{thm:firstpair}. The required tangent points are explicit; the
only exception is the auxiliary conic \eqref{eq:pqauxnorm}, whose rational
solubility follows from the Hasse principle.

First, the proof of Theorem~\ref{thm:firstpair} applies to
$\langle\Lambda,\Lambda'\rangle$ with $p$ there replaced by
$n=pq$. The norm representation \eqref{eq:pqnorm} supplies its
rational conic points. Both $p$ and $q$ make $5$ a local square,
so their contributions are trivial. The calculation at $2$ is
unchanged, and the calculation at $5$ depends only on $n\bmod5$.
When $n\equiv4\pmod5$, multiply the norm generator by
$9+4\sqrt5$ if necessary to arrange $a\equiv3\pmod5$.
Thus this entry is $\eps$.

For the four entries involving $U$ or $V$ in the second variable, the Hilbert
parameter is $q$. It is a local square at $2,5,p$, and positive at the real
place; at every other good prime the tangent value has even valuation.
Consequently the product is determined entirely by the place $q$. By Hilbert
reciprocity, rational rescaling of a tangent form does not affect the result.

For $\Lambda$, take the following tangent forms on its three conics:
\[
 \begin{aligned}
 L_1&=nt+bu_2-au_3,\\
 L_2&=(n+1)u_3-(n-1)u_1-4nt,\\
 L_3&=au_1-bu_2-nt.
 \end{aligned}
\]
The rational points giving them have, respectively,
\[
 (t,u_2,u_3)=(1,25b,5a),\quad
 (t,u_1,u_3)=(1,n-1,n+1),\quad
 (t,u_1,u_2)=(1,a,5b).
\]
At $q$, choose a square root $r$ of $5$ in $\Q_q$ with
$a-br$ a unit, and take the exact covering point
\[
 (t,u_1,u_2,u_3)=(0,1,r,1).
\]
Modulo $q$, the two possible roots make $a-br$ equal to zero
or $2a$; the latter is a unit since $q\nmid ab$.
The products for $U$ and $V$ are respectively
\[
 L_1L_3=-(a-br)^2,\qquad L_2L_3=2(a-br).
\]
Their quadratic characters are $+1$ and $(a/q)$, so
\[
 \langle\Lambda,U\rangle=0,\qquad
 \langle\Lambda,V\rangle=\alpha.
\]

For $\Lambda'$, put $A_*=5a+10b$ and $B_*=2a+5b$;
then $A_*^2-5B_*^2=5n$. Use
\[
 \begin{aligned}
 L'_1&=(5n-1)u_2-(5n+1)u_3+10nt,\\
 L'_2&=au_3-bu_1-2nt,\\
 L'_3&=A_*u_1-5B_*u_2-5nt.
 \end{aligned}
\]
The corresponding rational conic points are
\[
 \begin{gathered}
 (t,u_2,u_3)=(1,(5n-1)/2,(5n+1)/2),\\
 (t,u_1,u_3)=(1,10b,2a),\qquad
 (t,u_1,u_2)=(1,A_*,B_*).
 \end{gathered}
\]
Evaluate at $(t,u_1,u_2,u_3)=(0,r,1,1)$.
For the $U$ entry choose the root with $A_*r-5B_*$ a unit.
Then $L'_1L'_3\equiv20B_*\pmod q$, giving $(B_*/q)$.
Since $a^2=5b^2\pmod q$,
\[
 \leg{B_*}q=\leg aq\leg{2+r}q.
\]
For the $V$ entry choose $a-br$ a unit. The exact identity
\[
 A_*r-5B_*=5(r-2)(a-br)
\]
gives
\[
 L'_2L'_3=5(r-2)(a-br)^2.
\]
Its character is $\leg{2+r}q$, because $5$ and $-1$ are squares
modulo $q$. Hence
\[
 \langle\Lambda',U\rangle=\alpha+\tau,\qquad
 \langle\Lambda',V\rangle=\tau.
\]
These computations also prove that $\alpha$ is independent of the
norm representation: it is the intrinsic pairing
$\langle\Lambda,V\rangle$.

Finally consider $\langle U,V\rangle$.
Since $U$ is a Selmer class, its conic projection
\eqref{eq:pqauxnorm} is soluble everywhere locally. The
Hasse--Minkowski theorem gives a rational point and hence a
primitive integral triple $R,S,T$. The condition $T\ne0$ follows
from the nonsquareness of $q$, and primitivity forces
$q\nmid RT$: if $q\mid T$, the equation successively forces
$q\mid R$ and $q\mid S$.
On the $U$ covering use the two tangents
\[
 \begin{aligned}
 L^U_2&=(5p+1)u_3-(5p-1)u_1-20pt,\\
 L^U_3&=Ru_1-Su_2-5pTt.
 \end{aligned}
\]
They come from $(t,u_1,u_3)=(1,5p-1,5p+1)$ and
$(t,u_1,u_2)=(T,R,qS)$.
There is a $q$-adic covering point
\[
 (t,u_1,u_2,u_3)=(T,-R,qS,W),\qquad
 W^2=R^2+20pT^2,
\]
since the radicand is $25pT^2\pmod q$, a nonzero square.
Here $L^U_3=-2R^2$, a $q$-adic square.
Choose the sign of $W$ so that $L^U_2$ is a unit. Such a sign
exists: if both reductions vanished, then $q\mid5p+1$ and
$R\equiv2T\pmod q$; together with $R^2\equiv5pT^2$ this would
force $q=5$, a contradiction. The remaining local factor is $[\ell_q/q]=\gamma$. In particular the
construction of $\gamma$ is possible, and its value is independent of the
auxiliary choices because it is the intrinsic pairing $\langle U,V\rangle$.

The displayed matrix now follows by alternation. Its Pfaffian is
$\eps\gamma+\alpha(\alpha+\tau)
=\eps\gamma+\alpha(1+\tau)$.
If the Pfaffian is one, the ordinary pairing is nondegenerate. Its radical is
the image of $\Sel_4$ in the pure $2$-Selmer group, so nondegeneracy forces
rank zero and rules out both elementary divisors of order at least four and a
nonzero divisible $2$-primary subgroup. The pure Selmer group has dimension
four, which gives the asserted $\Sh[2^\infty]$.
The geometric conclusion uses $A_{5pq}\simeq E_{5,3}^{(10pq)}$.
\end{proof}

Fix a positive norm
representation
\begin{equation}\label{eq:pqfixednorm}
 p=A^2-5B^2,\qquad4\mid A,\quad B\text{ odd},\qquad
 \xi_p=A+B\sqrt5.
\end{equation}
It exists by the norm argument preceding \eqref{eq:pqbits}.
For either root $r^2=5$ in $\F_q$, define
\begin{equation}\label{eq:pqfixedbits}
 \beta_p(q)=[(A+Br)/q],\qquad
 t(q)=[5/q]_4,\qquad g_p(q)=[p/q]_4.
\end{equation}
The first symbol is independent of the root since
$(A+Br)(A-Br)=p$ is a nonzero square modulo $q$.

\begin{lemma}\label{lem:pqcyclotomic}
For $q\equiv1\pmod{40}$ and either $r^2=5$ in $\F_q$,
\[
 \leg{5+r}q=1,\qquad
 \leg{2+r}q=\leg rq=\left(\frac5q\right)_4.
\]
\end{lemma}
\begin{proof}
Set $F_0=\Q(\zeta_{40})$, $\phi=(1+\sqrt5)/2$, and
$s=\zeta_{20}+\zeta_{20}^{-1}$. Choose the embedding with
$s^2=(5+\sqrt5)/2$. Direct calculation gives
\[
 (\sqrt2s)^2=5+\sqrt5,\qquad
 (\phi s)^2=5+2\sqrt5=\sqrt5(2+\sqrt5).
\]
The prime $q$ splits completely in $F_0$. Reducing these identities
at a prime with $\sqrt5\mapsto r$ proves the first two character
identities; the reductions are nonzero since $q\ne2,5$.
Finally $r^{(q-1)/2}=5^{(q-1)/4}$.
\end{proof}

\begin{proposition}\label{prop:pqfixedsymbols}
The entries of Theorem~\ref{thm:pqcassels} satisfy
\begin{equation}\label{eq:pqfixedcomparison}
 \tau=t(q),\qquad\gamma=g_p(q),\qquad
 \alpha=\beta_p(q)+t(q).
\end{equation}
In particular the Pfaffian has the fixed symbol expression
\begin{equation}\label{eq:pqfixedpf}
 \Pf(C_{p,q})=\eps g_p(q)+\beta_p(q)(1+t(q)).
\end{equation}
\end{proposition}
\begin{proof}
The assertion about $\tau$ is Lemma~\ref{lem:pqcyclotomic}.
For $\gamma$, take the primitive triple $R,S,T$ used in
\eqref{eq:pqauxnorm}. We begin by showing $(T/q)=1$.
If an odd prime $\ell$ divides $T$, primitivity implies
$\ell\nmid RS$; reducing the norm equation gives
$q\equiv(R/S)^2\pmod\ell$. Also $q\nmid T$.
Quadratic reciprocity and $q\equiv1\pmod4$ give $(\ell/q)=1$.
The factors $-1$ and $2$ have character one modulo $q$ as well,
so $(T/q)=1$.

In $\F_q$, set $b_0=R/T$ and $c_0=w/T$. Then
$b_0^2=5p$ and $c_0^2=25p$. Put
$u=c_0/5$ and $r_0=b_0/u$, so $u^2=p$ and $r_0^2=5$.
The nonzero element $z=c_0+b_0=u(5+r_0)$ satisfies
\[
 \frac{\ell_q}{T}
 =(5p+1)c_0+(5p-1)b_0-20p
 =5p\frac{(z-2)^2}{z}.
\]
Here $z-2\ne0$ because $\ell_q$ was chosen to be a unit.
Taking characters and applying Lemma~\ref{lem:pqcyclotomic} gives
\[
 \leg{\ell_q}q=\leg zq=\leg uq
 =\left(\frac pq\right)_4.
\]
Hence $\gamma$ is given by the stated quartic symbol.

Choose a positive representation $q=c^2-5d^2$ with $c$ odd
and $4\mid d$. Class number one in $\Q(\sqrt5)$ gives an
integral generator of norm $q$; a norm one unit moves it into
$\Z[\sqrt5]$, after which reduction modulo eight gives these
parities. Conjugation and positive norm one units arrange $c,d>0$.
For each odd prime $\ell\mid d$, the identity
$q\equiv c^2\pmod\ell$ and reciprocity give $(\ell/q)=1$;
thus $(d/q)=1$. In particular, for $r_0=c/d$ modulo $q$,
\[
 \leg cq=\leg{r_0}q=\left(\frac5q\right)_4.
\]
Multiplying $\xi_p(c+d\sqrt5)$ gives a valid choice in
\eqref{eq:pqnorm}, namely
$a=Ac+5Bd$ and $b=Ad+Bc$. Modulo $q$,
\[
 a=c(A+Br_0).
\]
Its character is the product defining $\beta_p(q)+t(q)$.
The choice independence proved in Theorem~\ref{thm:pqcassels}
allows this choice of $a,b$. Finally
$(\beta+t)(1+t)=\beta(1+t)$ in $\F_2$, proving
\eqref{eq:pqfixedpf}.
\end{proof}

Return to $n=pq$ satisfying \eqref{eq:pqfamily}. Keep
$\eps,\alpha,\tau,\gamma$ from \eqref{eq:pqbits},
and set
\[
 \upsilon=[2/q]_4.
\]
The quartic symbol is defined since $q\equiv1\pmod8$.
Denote the ordinary Cassels matrices for $E_1(n),E_{-1}(n),E_2(n)$
by $C_+,C_-,C_2$, in the bases of Corollary~\ref{cor:fourdimensions}.

\subsection{The remaining twists}\label{subsec:remainingtwists}

\begin{theorem}\label{thm:threecassels}
The first matrix is $C_+=C_{p,q}$ of \eqref{eq:pqcasselsmatrix},
and the other two are
\begin{equation}\label{eq:threecassels}
 C_-=
 \begin{pmatrix}
 0&1+\eps&0&\alpha+\tau\\
 1+\eps&0&\alpha&\tau\\
 0&\alpha&0&\gamma\\
 \alpha+\tau&\tau&\gamma&0
 \end{pmatrix},\qquad
 C_2=\begin{pmatrix}0&\gamma+\upsilon\\
 \gamma+\upsilon&0\end{pmatrix}.
\end{equation}
Consequently, if $P_+=\Pf(C_+)$, $P_-=\Pf(C_-)$, and
$P_2=\Pf(C_2)$, then
\begin{equation}\label{eq:threepf}
 P_- = P_+ +\gamma,\qquad P_2=\gamma+\upsilon,
 \qquad P_+=\eps\gamma+\alpha(1+\tau).
\end{equation}
On the odd dimensional space for $E_{-2}(n)$ the pairing of $U,V$
is also $\gamma+\upsilon$. Thus $P_2=1$ forces its Cassels
matrix to have rank two and
\[
 \rank E_{-2}(n)(\Q)\le1.
\]
\end{theorem}
\begin{proof}
First consider the entry between $U=(q,1,q)$ and $V=(1,q,q)$
for any of the four twists. Put $h=5\delta p$. The conic
projection of the locally soluble $U$ covering has a primitive
integral solution
\[
 R^2-qS^2=hT^2,\qquad T\ne0,\qquad q\nmid RT.
\]
Use the tangent forms
\[
 L_2=(h+1)u_3-(h-1)u_1-4ht,\qquad
 L_3=Ru_1-Su_2-hTt.
\]
At $q$, take $(t,u_1,u_2,u_3)=(T,-R,qS,w)$, where
$w^2=R^2+4hT^2$ in $\Q_q$. This is possible because the
radicand is the nonzero square $5hT^2$ modulo $q$.
Then $L_3=-2R^2$ is a square, and the sign of $w$ can be
chosen so that $\ell=(h+1)w+(h-1)R-4hT$ is a unit.
Indeed if both signs vanished, then $h=-1$ and $R=2T$
modulo $q$, contradicting $R^2=hT^2$ since $q\ne5$.
As before, only $q$ contributes: the Hilbert parameter $q$ is a
square at $2,5,p$, is positive at infinity, and the primitive
tangents give trivial symbols at the remaining good places.

The norm equation and quadratic reciprocity give $(T/q)=1$,
exactly as in Proposition~\ref{prop:pqfixedsymbols}.
In $\F_q$, put $b=R/T$, $c=w/T$, and $z=b+c$.
Then $b^2=h$, $c^2=5h$, and
\[
 \ell/T=h(z-2)^2/z,\qquad
 z=b(1+r),\quad r=c/b,\quad r^2=5.
\]
Since $(h/q)=1$, the character of $\ell$ is that of $z$.
Now $\phi=(1+r)/2$ satisfies $\phi^3=2+r$, so
$\leg{1+r}q=\leg{2+r}q$:
the factor $2$ is a square. Thus, in additive notation,
\[
 \langle U,V\rangle_{E_\delta(n)}
 =[h/q]_4+\tau=[\delta p/q]_4.
\]
Here $[5/q]_4=\tau$ cancels, by
Lemma~\ref{lem:pqcyclotomic}. Also $[-1/q]_4=0$ because
$q\equiv1\pmod8$. This is the $U,V$ entry for all four twists and, in particular, gives the
formula for $C_2$.

For the remaining entries of $C_-$, let $n=a^2-5b^2$ be the
positive representation used before. Multiplication by $2+\sqrt5$
gives
\[
 -n=a_-^2-5b_-^2,\qquad
 a_-=2a+5b\text{ odd},\quad b_-=a+2b\equiv2\pmod4.
\]
At $q$, $[a_-/q]=\alpha+\tau$, because
$a_-=a(2+r)$ for the root $r=a/b$ modulo $q$.
All four cross entry calculations in the proof of
Theorem~\ref{thm:pqcassels} are algebraic identities in the
norm parameter. Replacing it by $-n$ gives, respectively,
\[
 \langle\Lambda,U\rangle=0,\quad
 \langle\Lambda,V\rangle=\alpha+\tau,\quad
 \langle\Lambda',U\rangle=\alpha,\quad
 \langle\Lambda',V\rangle=\tau.
\]
The same $q$-adic points with $t=0$ are valid, and all other
local symbols in these four entries are trivial for the same reasons.

It remains to check $\langle\Lambda,\Lambda'\rangle$;
this entry depends on the dyadic parity. Put $s=-n$. Use the same
tangents $L_2=(s+1)u_3-(s-1)u_1-4st$ and
$L_3=a_-u_1-b_-u_2-st$. At $2$ the point with
$t=2,u_1=1$, $u_2^2=5-20s$, $u_3^2=1+16s$ gives
$\ord_2(L_2)=1$ and $L_3$ odd; the local symbol with
parameter $5$ is $-1$. At $5$ take
$(t,u_1,u_2,u_3)=(1,2r,5r,0)$, $r^2=-s$,
$r\equiv2a_-\pmod5$. If $s\equiv4\pmod5$, first arrange
$a_-\equiv3\pmod5$ by multiplying the norm generator by
$9+4\sqrt5$ if necessary. The two tangent reductions are
\[
 \begin{array}{c|ccc}
 s\bmod5&L_2\bmod5&L_3\bmod5&(L_2L_3,5)_5\\ \hline
 1&1&3&-1\\4&3&2&1.
 \end{array}
\]
All other places contribute one. Since $s=-pq$, this gives
$1+\eps$. The displayed matrix and its Pfaffian follow.

Finally a nonzero pairing between $U,V$ on a three dimensional
alternating space forces rank two. The Mordell--Weil group
modulo torsion maps into its one dimensional radical, giving
the stated rank bound.
\end{proof}

Write $\xi_5=\xi_p=A+B\sqrt5$ for the fixed norm generator
in \eqref{eq:pqfixednorm}. Choose
\begin{equation}\label{eq:fourimaginarynorms}
 p=A_d^2+dB_d^2,\qquad \xi_d=A_d+B_d\sqrt{-d}
 \quad(d=2,10).
\end{equation}
For roots $r_d^2=-d$ modulo $q$, put
\[
 \rho_d=\rho_{d,p}(q)=[(A_d+B_dr_d)/q],\qquad \rho=\rho_2.
\]
Here, as before, a bracket denotes the additive quadratic symbol.

\begin{lemma}\label{lem:fourimaginarynorms}
The representations \eqref{eq:fourimaginarynorms} exist, as do
$q=C_d^2+dD_d^2$ for $d=2,10$. The symbols $\rho_d$ do not
depend on the choices of roots or conjugate norm generators.
If
\[
 a_d+b_d\sqrt{-d}
 =(A_d+B_d\sqrt{-d})(C_d+D_d\sqrt{-d}),
\]
then $n=a_d^2+db_d^2$ and
\begin{equation}\label{eq:fourkappas}
 [a_2/q]=\rho+\upsilon,\qquad
 [a_{10}/q]=\rho_{10}+\upsilon+\tau.
\end{equation}
\end{lemma}
\begin{proof}
Both $p$ and $q$ split in $\Q(\sqrt{-2})$ and
$\Q(\sqrt{-10})$. The reduced primitive positive forms of
discriminant $-8$ consist only of $x^2+2y^2$; those of
discriminant $-40$ are $x^2+10y^2$ and $2x^2+5y^2$.
These lists follow at once from the reduction bound
$a\le\sqrt{|D|/3}$; the standard prime representation
correspondence is recalled in \cite{Cox}. An odd value of
$2x^2+5y^2$ is $5$ or $7$ modulo $8$, whereas $p\equiv3$
and $q\equiv1\pmod8$. Thus the required norm representations exist in all
four cases.
The representations are primitive. In particular $C_d$ is odd,
$D_d$ is even, and $q\nmid C_dD_d$.

Changing a root replaces $\xi_d$ by its conjugate, whose product
with $\xi_d$ is $p$. Since $(p/q)=1$, its character is unchanged.
The only units in these two imaginary quadratic fields are
$\pm1$, also squares modulo $q$.

If an odd prime $\ell$ divides $D_d$, then
$q\equiv C_d^2\pmod\ell$ and quadratic reciprocity gives
$(\ell/q)=1$. Also $(2/q)=(-1/q)=1$, so $(D_d/q)=1$.
Set $r=C_d/D_d$ in $\F_q$; then $r^2=-d$ and
\[
 [C_d/q]=[-d/q]_4=[d/q]_4,\qquad
 a_d\equiv C_d(A_d+B_dr)\pmod q.
\]
Thus $[a_d/q]=\rho_d+[d/q]_4$. For $d=2,10$, this is
\eqref{eq:fourkappas}, since $[5/q]_4=\tau$.
\end{proof}

The three norm characters satisfy the relation below.

\begin{lemma}\label{lem:threeNormRelation}
For every eligible $q$,
\begin{equation}\label{eq:threeNormRelation}
 \rho+\rho_{10}+\beta_p(q)=0.
\end{equation}
More intrinsically, with coherent square roots
$\sqrt{-10}=\sqrt5\sqrt{-2}$, the product
$\xi_5\xi_2\xi_{10}$ is a square in $B_p=\Q(\zeta_{40},\sqrt p)$.
\end{lemma}
\begin{proof}
Put $F=\Q(\sqrt5,\sqrt{-2})$ and $T=\xi_5\xi_2\xi_{10}$.
Let $\sigma_5$ change the sign of $\sqrt5$ and let
$\sigma_2$ change the sign of $\sqrt{-2}$. Each fixes the
other displayed square root. All three norm generators have
norm $p$, whence
\[
 \frac{\sigma_5(T)}T=\left(\frac p{\xi_5\xi_{10}}\right)^2,
 \qquad
 \frac{\sigma_2(T)}T=\left(\frac p{\xi_2\xi_{10}}\right)^2.
\]
If $T$ is already a square in $F$ there is nothing to prove.
Otherwise put $w=\sqrt T$. The two automorphisms lift to
$F(w)$ by
\[
 \sigma_5(w)=\frac{pw}{\xi_5\xi_{10}},\qquad
 \sigma_2(w)=\frac{pw}{\xi_2\xi_{10}}.
\]
Both lifts have order two. They commute, since either composite
sends $w$ to $pw/(\xi_5\xi_2)$. Together with $w\mapsto-w$
they exhibit $F(w)/\Q$ as an elementary abelian extension of
degree eight.

The generators $\xi_d$ are integral and are units away from
$p$; the field $F$ is unramified outside $2,5$. Consequently
$F(w)/\Q$ is unramified outside $2,5,p$. Every quadratic
subfield therefore has a squarefree radicand supported on
$-1,2,5,p$. Since an elementary abelian extension is the
compositum of its quadratic subfields,
\[
 F(w)\subset\Q(i,\sqrt2,\sqrt5,\sqrt p)\subset B_p.
\]
Finally $q$ splits completely in $B_p$. Reducing $w^2=T$
at a prime above $q$ gives \eqref{eq:threeNormRelation}.
Root independence in the preceding lemma removes the need to
specify the coherent roots in the resulting formula.
\end{proof}

\begin{theorem}\label{thm:fourthcassels}
In the ordered pure Selmer basis $(W,U,V)$ of
$E_{-2}(n)=A_{-10n}$, its complete additive Cassels matrix is
\begin{equation}\label{eq:fourthcassels}
 C_{-2}=
 \begin{pmatrix}
 0&\beta&\rho+\beta+\gamma\\
 \beta&0&\gamma+\upsilon\\
 \rho+\beta+\gamma&\gamma+\upsilon&0
 \end{pmatrix},\qquad \beta=\beta_p(q).
\end{equation}
In particular,
\[
 \langle W,U\rangle=\beta,\qquad
 \langle W,V\rangle=\rho+\beta+\gamma.
\]
The matrix is zero precisely when
\begin{equation}\label{eq:fourthzerolocus}
 \beta=0,\qquad \rho=\gamma=\upsilon.
\end{equation}
Otherwise it has rank two and its radical is generated by
\begin{equation}\label{eq:fourthradical}
 (\gamma+\upsilon)W+(\rho+\beta+\gamma)U+\beta V.
\end{equation}
\end{theorem}
\begin{proof}
We change the representative by a rational torsion class. In three Kummer coordinates the image of the point
$(m,0)$, with $m=-10n$, is
$t_m=(-10n,5,-2n)$. Thus the class
\[
 Z=W+U+t_m=(-10,5,-2)
\]
has fixed coordinates. Torsion is in the Cassels radical, so
\begin{equation}\label{eq:fourthZchange}
 \langle W,U\rangle=\langle Z,U\rangle,\qquad
 \langle W,V\rangle=\langle Z,V\rangle+\gamma+\upsilon.
\end{equation}
The three conics of its covering are
\[
 \begin{aligned}
 H_1&:2u_3^2+5u_2^2=50nt^2,\\
 H_2&:u_3^2-5u_1^2=20nt^2,\\
 H_3&:2u_1^2+u_2^2=2nt^2.
 \end{aligned}
\]
Use $n=a^2-5b^2$ as before and the two product representations
$n=a_d^2+db_d^2$ from Lemma~\ref{lem:fourimaginarynorms}.
Put $D=a+2b$ and $J=2a+5b$, so that $5D^2-J^2=n$.
Rational points on $H_1,H_2,H_3$, respectively, are
\[
 (t,u_2,u_3)=(1,10b_{10},5a_{10}),\quad
 (t,u_1,u_3)=(1,2J,10D),\quad
 (t,u_1,u_2)=(1,a_2,2b_2).
\]
The resulting tangent forms, after rational rescaling, are
\begin{equation}\label{eq:fourthtangents}
 \begin{aligned}
 L_1&=a_{10}u_3+5b_{10}u_2-5nt,\\
 L_2&=Du_3-Ju_1-2nt,\\
 L_3&=a_2u_1+b_2u_2-nt.
 \end{aligned}
\end{equation}
Their coefficients are primitive away from $2,5$. Indeed all
three norm representations are primitive, and
$\gcd(D,J)=\gcd(a,b)=1$.

For either second argument $U$ or $V$ the Hilbert parameter is
$q$. It is a square at $2,5,p$ and is positive at infinity.
At every other prime different from $q$ the conics and their
primitive rational points have good reduction. A tangent line
meets its conic in twice its point of tangency, so nonzero
tangent values on primitive local points have even valuation.
The projections of a primitive covering point remain primitive
at these primes, by the displayed conic equations. Hence only
the place $q$ contributes to the Cassels product
\cite{CasselsSecond,FisherSchaeferStoll}.

At $q$, use the local point
$(t,u_1,u_2,u_3)=(0,1,s_2,s_5)$, where
$s_2^2=-2$ and $s_5^2=5$ in $\Q_q$.
Choose $s_2$ so that $L_3\equiv2a_2\not\equiv0\pmod q$.
For $\langle Z,U\rangle$, choose $s_5$ so that
$L_1\equiv2a_{10}s_5\not\equiv0\pmod q$.
These choices are possible because the norm equations give
the two opposite nonzero roots. Since $[s_5/q]=[5/q]_4=\tau$,
the product $L_1L_3$ gives
\[
 \langle Z,U\rangle=[a_2/q]+[a_{10}/q]+\tau
 =\rho+\rho_{10}=\beta.
\]

For $\langle Z,V\rangle$, keep $s_2$ and instead choose
the sign of $s_5$ so that $L_2\equiv-2J\pmod q$.
This is possible because $J^2=5D^2$ modulo $q$, with
$q\nmid DJ$. Also
$J\equiv a(2+r)\pmod q$ for the root $r=a/b$ of $5$.
Thus the product $L_2L_3$ gives
\[
 \langle Z,V\rangle=[a_2/q]+[J/q]
 =\rho+\upsilon+\alpha+\tau
 =\rho+\upsilon+\beta.
\]
Here we used $\alpha=\beta+\tau$ from
Proposition~\ref{prop:pqfixedsymbols}. Equation~\eqref{eq:fourthZchange}
proves the two stated entries. The $U,V$ entry was proved in
Theorem~\ref{thm:threecassels}. The zero criterion and radical
vector follow directly from this $3\times3$ alternating matrix.
\end{proof}

\section{Prime higher reciprocity and \texorpdfstring{$4$}{4}-descent}
\label{sec:higher}

\subsection{The intrinsic conic character}\label{subsec:intrinsicconic}

Throughout this section $p\equiv11\pmod{40}$ is prime. By
Theorem~\ref{thm:firstpair}, the ordinary Cassels pairing on the pure
$2$-Selmer group is zero, so ordinary $4$-descent no longer separates its
two generators. The next Cassels--Tate pairing does. We compute this pairing
explicitly and then identify its single nontrivial bit with the representation
defect, first through a Pell equation and then through a class number
congruence. Let
\[
 V=\Sel_2^0(A_{5p}).
\]
For $u,v\in V$, let $\bar u,\bar v$ denote their images in
$\Sh(A_{5p})[2]$. Choose
\[
 \widetilde u\in\Sh(A_{5p})[4],
 \qquad
 2\widetilde u=\bar u,
\]
and define
\begin{equation}\label{eq:nextpairdef}
 \mathcal B_p(u,v)
 =
 \langle\widetilde u,\bar v\rangle_{\mathrm{CT}}
 \in \tfrac12\Z/\Z\simeq\F_2.
\end{equation}

\begin{proposition}\label{prop:nextpair}
The pairing \eqref{eq:nextpairdef} is well defined, alternating, and
\[
 \rad\mathcal B_p
 =
 \operatorname{im}\bigl(\Sel_8(A_{5p})\longrightarrow V\bigr).
\]
Writing
\[
 C_p^{\mathrm{next}}
 =
 \begin{pmatrix}0&b_p\\b_p&0\end{pmatrix},
\]
one has
\[
 b_p=1
 \Longleftrightarrow
 \rank A_{5p}(\Q)=0,\qquad
 \Sh(A_{5p})[2^\infty]\simeq(\Z/4\Z)^2.
\]
If $\D(p)\ne0$ and $b_p=0$, then
\[
 \Sh(A_{5p})[2^\infty]
 \simeq(\Z/2^{a_p}\Z)^2
\]
for some $a_p\ge3$.
\end{proposition}

\begin{proof}
Because the ordinary Cassels pairing vanishes, every class of $V$ lifts to
$\Sel_4$. If two halves $\widetilde u$ are chosen, their difference lies in
$\Sh[2]$ and pairs trivially with $\bar v$. Thus
\eqref{eq:nextpairdef} is well defined. Alternation follows from
\[
 \langle\widetilde u,\bar u\rangle_{\mathrm{CT}}
 =
 \langle\widetilde u,2\widetilde u\rangle_{\mathrm{CT}}=0.
\]
The description of the radical is the divisibility criterion for the
Cassels--Tate pairing; see \cite[Theorem 3.1]{Fisher}. Nondegeneracy
forces rank zero and excludes elementary divisors of order at least
$8$; the vanishing of the ordinary pairing already excludes order
$2$. Since $\dim_{\F_2}\Sh[2]=2$, the group is
$(\Z/4\Z)^2$.

If $\D(p)\ne0$, the central value formula gives rank zero and
finiteness of $\Sh$. Perfect alternation of the finite
Cassels--Tate pairing pairs the elementary divisors, so the
$2$-primary group is $(\Z/2^{a_p}\Z)^2$. The two successive zero
pairings force $a_p\ge3$.
\end{proof}

Choose coprime positive integers $a,b$ satisfying
\begin{equation}\label{eq:normab}
 p=a^2-5b^2,\qquad 4\mid a,\qquad b\text{ odd},
 \qquad k=(-1)^{(b-1)/2}.
\end{equation}
Lemma~\ref{lem:norm5} supplies this normalization.

Put
\begin{align}
 F_p(u,v)
 &=(a+5b)u^2+10(a+b)uv+5(a+5b)v^2,\label{eq:F0}\\
 G_p(u,v)
 &=(5b-a)u^2+10(a-b)uv+5(5b-a)v^2,\label{eq:G0}\\
 \mathcal F&=kF_p,\qquad \mathcal G=kG_p.\nonumber
\end{align}
Expansion gives
\begin{align}
 F_p^2-3F_pG_p+G_p^2
 &=5p(u^2-5v^2)^2,\label{eq:keyidentity}\\
 \disc F_p=\disc G_p&=80p,\qquad
 \operatorname{Res}(F_p,G_p)=2000p^2.\label{eq:discriminants}
\end{align}

The associated conic is
\begin{equation}\label{eq:Gamma}
 \Gamma_p:\qquad c^2=\mathcal F(u,v).
\end{equation}
For the conic calculations write
\[
 A=k(a+5b),\qquad B=5k(a+b),\qquad
 \mathcal F(u,v)=Au^2+2Buv+5Av^2.
\]
Then
\begin{equation}\label{arith:basic}
 A\equiv1\pmod4,\qquad B-A=4ka,\qquad
 5A^2-B^2=-20p.
\end{equation}
We first establish that $\Gamma_p(\Q)$ is nonempty and then show that the
quadratic character extracted from a primitive point does not depend on the
point.

\begin{lemma}\label{arith:chars}
The parameters in \eqref{eq:normab} exist. They satisfy
\[
 \leg bp=k,\qquad \leg{a+b}p=1,\qquad \leg ap=1,
 \qquad \leg Ap=1.
\]
\end{lemma}
\begin{proof}
Existence follows from Lemma~\ref{lem:norm5}.

The numerators below are coprime to their odd denominators, so all Jacobi
symbols that occur are well defined. Quadratic reciprocity gives
\[
 \leg bp=(-1)^{(b-1)/2}\leg pb=k,
\]
because $p\equiv a^2\pmod b$. Similarly,
$p\equiv-4b^2\pmod{a+b}$ gives $(\frac{a+b}{p})=1$.
Writing $a=2^e a'$ with $a'$ odd, we have
\[
 \leg{a'}p=\leg{-p}{a'}=\leg5{a'}=\leg{a'}5,
 \qquad \leg2p=\leg25=-1.
\]
Hence $(\frac ap)=(\frac a5)=1$, since $a\equiv\pm1\pmod5$.
Finally
\[
 a+5b\equiv\frac{a(a+b)}b\pmod p,
\]
so $(\frac{a+5b}{p})=k$; multiplication by $k$ gives $(\frac Ap)=1$.
\end{proof}

\begin{proposition}\label{arith:exist}
The conic $\Gamma_p$ has a rational point. Every primitive
integral point $(u,v,c)$ on it satisfies $5\nmid uc$.
\end{proposition}
\begin{proof}
At $2$, if $A\equiv1\pmod8$, take $(u,v)=(1,0)$; if
$A\equiv5\pmod8$, take $(u,v)=(0,1)$. In either case $\mathcal F(u,v)$ is
an odd square in $\Q_2$. At $5$, $\mathcal F(1,0)=A\equiv ka\pmod5$ is
a nonzero square, and Hensel's lemma applies. At $p$, use $(1,0)$
and Lemma~\ref{arith:chars}. At every odd prime $q\nmid5p$, the ternary
form $c^2-\mathcal F(u,v)$ is nondegenerate modulo $q$ and has a nonzero
isotropic vector over $\F_q$; it lifts to $\Q_q$. At infinity the
binary form $\mathcal F$ is indefinite by \eqref{arith:basic}, so it takes positive
values. The Hasse principle for conics now gives a rational point.

Clear denominators and divide by the common gcd. If $5\mid u$, the
equation modulo $5$ gives $5\mid c$. Primitivity then gives
$5\nmid v$, but $v_5(\mathcal F(u,v))=1$, a contradiction. Thus $5\nmid u$,
and $c^2\equiv ka u^2\pmod5$ also gives $5\nmid c$.
\end{proof}

For two primitive integral points $P_i=(u_i,v_i,c_i)$, define
the tangent value
\begin{equation}\label{arith:tangent}
 \ell(P_1)=c_0c_1-Au_0u_1-B(u_0v_1+v_0u_1)-5Av_0v_1.
\end{equation}
The associated linear form is the tangent to $\Gamma_p$ at $P_0$.

\begin{lemma}\label{arith:dyadic}
If $\ell(P_1)\ne0$, then $v_2(\ell(P_1))$ is odd.
\end{lemma}
\begin{proof}
The parity of a primitive point $(u,v,c)$ is as follows.
Set $s=u+v$. From \eqref{arith:basic},
\begin{equation}\label{arith:parity}
 \mathcal F(u,v)=As^2+8ka sv+(4A-8ka)v^2.
\end{equation}
If $c$ is odd, then $s$ is odd. If $c$ is even, primitivity and the
equation modulo $2$ give $u,v$ odd. Reducing \eqref{arith:parity} modulo
$16$ excludes $s\equiv2\pmod4$ and gives
\[
 4\mid s,\qquad v_2(c)=1.
\]

Since $-5p\equiv1\pmod8$, choose $r\in\Z_2^\times$ with
$r^2=-5p$. Completing the square gives
\[
 Ac^2=(Au+Bv)^2+(2rv)^2.
\]
For each point put
\[
 x_i=\frac{Au_i+Bv_i}{c_i},\qquad
 y_i=\frac{2rv_i}{c_i}.
\]
There is no rational point with $c_i=0$, because the discriminant
$80p$ of $\mathcal F$ is not a rational square. The parity facts above show
that $x_i,y_i\in\Z_2$ and $x_i^2+y_i^2=A$. If $c_i$ is odd,
$x_i$ is odd and $y_i$ is even; if $c_i$ is even, these parities
are reversed.

Let $T=A-x_0x_1-y_0y_1$. Direct substitution yields
\begin{equation}\label{arith:T}
 A\ell(P_1)=c_0c_1T.
\end{equation}
If $c_0,c_1$ have opposite parity, then $T$ is odd and
$v_2(c_0c_1)=1$, proving the claim.

Otherwise $T$ is even and $v_2(c_0c_1)$ is even. The identity
\[
 T(2A-T)=(x_0y_1-y_0x_1)^2
\]
implies that $v_2(T)$ is odd: this is immediate if $v_2(T)=1$;
if $v_2(T)\ge2$, then $v_2(2A-T)=1$. Equation \eqref{arith:T} proves
the result in both cases.
\end{proof}

\begin{theorem}\label{arith:pointind}
For fixed $(a,b)$, the value $k(\frac c5)$ is independent of the
primitive integral point $(u,v,c)$ on $\Gamma_p$.
\end{theorem}
\begin{proof}
Write $\eps_a=1$ for $a\equiv1\pmod5$ and
$\eps_a=-1$ for $a\equiv4\pmod5$. Reduction of the conic
equation modulo $5$ gives
\begin{equation}\label{arith:eps}
 k\leg c5=\eps_a\leg u5.
\end{equation}
It is enough to prove $(\frac{u_0u_1}{5})=1$.

Changing the sign of $c_1$ leaves $(\frac{c_1}{5})$ unchanged, so we may
choose it such that
$c_1/u_1\equiv-c_0/u_0\pmod5$. Then the rational number
$\ell=\ell(P_1)$ is nonzero and
\[
 \ell\equiv-2ka u_0u_1\pmod5,
 \qquad (\ell,5)_5=-\leg{u_0u_1}{5}.
\]
Lemma~\ref{arith:dyadic} and the dyadic Hilbert symbol formula give
$(\ell,5)_2=(-1)^{v_2(\ell)}=-1$.

At $p$, the second entry $5$ is a square; at infinity it is positive,
so the symbols are trivial. Let $q$ be odd and $q\nmid5p$.
The conic is smooth over $\Z_q$, and its section $P_0$ identifies it
with $\mathbb P^1_{\Z_q}$. The tangent section pulls back to a unit
times the square of a primitive linear form: its divisor is $2P_0$,
and the tangent has nonzero reduction, hence no vertical component.
At a primitive local representative of $P_1$, its valuation is
therefore even. If $5$ is not already a square in $\Q_q$, the
extension $\Q_q(\sqrt5)/\Q_q$ is unramified; again $(\ell,5)_q=1$.

Applying Hilbert reciprocity to the rational number $\ell$ gives
\[
 1=\prod_w(\ell,5)_w=(-1)\left(-\leg{u_0u_1}{5}\right)
  =\leg{u_0u_1}{5}.
\]
Equation~\eqref{arith:eps} proves the assertion.
\end{proof}

\begin{theorem}\label{arith:normind}
The character of Theorem~\ref{arith:pointind} is independent of the choice
of the positive norm representation \eqref{eq:normab}.
\end{theorem}
\begin{proof}
Set $K=\Q(\sqrt5)$, $\phi=(1+\sqrt5)/2$,
$\alpha=a+b\sqrt5$, and $\xi=u+\sqrt5v$. Direct expansion gives
\begin{equation}\label{arith:trace}
 c^2=k\Tr_{K/\Q}(\alpha\phi\xi^2).
\end{equation}
Any two totally positive integral generators of a fixed prime over
$p$ differ by $\phi^{2n}$. Both are congruent to $1$ modulo
$2\mathcal O_K$, so $3\mid n$. Thus they differ by a power of
$\phi^6=9+4\sqrt5$. The other prime is obtained by conjugation.
Thus only multiplication by $\phi^6$ and conjugation need to be considered.
For the latter comparison it is convenient to allow $b$ to be any odd
integer and to define $k$ from its residue modulo $4$.

First, let $\alpha_1=(9+4\sqrt5)\alpha$. Then
\[
 a_1=9a+20b,\qquad b_1=4a+9b\equiv b\pmod4,
\]
so $k_1=k$. Since $(2+\sqrt5)^2=9+4\sqrt5$, multiplication of
$\xi$ by $2+\sqrt5$ gives the integral coordinate map
\[
 (u,v)\longmapsto(2u+5v,u+2v).
\]
Its determinant is $-1$. Equation \eqref{arith:trace} identifies the two
conics by this map in the appropriate direction, preserves
primitivity, and leaves $c$ unchanged. Hence the character is unchanged.

For conjugation, put $\alpha_1=\bar\alpha$ and $k_1=-k$.
Since $\bar\phi=-\phi^{-1}$, the substitution
$\xi_1=\bar\phi\bar\xi$ gives
\[
 k_1\Tr(\alpha_1\phi\xi_1^2)=k\Tr(\alpha\phi\xi^2).
\]
In integral homogeneous coordinates it is
\begin{equation}\label{arith:conjmap}
 (u,v,c)\longmapsto(U,V,C)=(u+5v,-u-v,2c).
\end{equation}
No odd prime divides all three new coordinates, because the
determinant of the map on $(u,v)$ is $4$. If $c$ is odd, then $u,v$
have opposite parity, and the new triple is primitive. If $c$ is
even, Lemma~\ref{arith:dyadic}'s parity calculation gives $4\mid u+v$
and $v_2(c)=1$. Thus $4\mid U,V,C$ and their common gcd is exactly
$4$. After normalization the new third coordinate is respectively
$2c$ or $c/2$. In either case
\[
 \leg{c_1}5=-\leg c5,
\]
which cancels $k_1=-k$. The theorem follows.
\end{proof}

We may therefore define, independently of every auxiliary choice,
\begin{equation}\label{eq:rhointrinsic}
 \varrho(p)=k\leg{c_0}{5}=\eps_a\leg{u_0}{5},
 \qquad (u_0,v_0,c_0)\in\Gamma_p(\Q)\text{ primitive integral},
\end{equation}
where $\eps_a$ is as in the proof of Theorem~\ref{arith:pointind}.

Consider
\begin{equation}\label{eq:canonicalcover}
 c^2=\mathcal F(u,v),\qquad
 d^2=\mathcal G(u,v),\qquad
 k h^2=c^2-cd-d^2,
\end{equation}
and let $\widetilde C_p$ denote its smooth projective normalization.

\subsection{The \texorpdfstring{$4$}{4}-covering and the next pairing}\label{subsec:primefourcover}

\begin{proposition}\label{prop:4cover}
The curve $\widetilde C_p$ is everywhere locally soluble and defines a
$4$-Selmer lift of the class $\Lambda'=(1,5,5)$.
\end{proposition}

\begin{proof}
We separate the construction into three steps: the intermediate genus one
curve, the unramified double cover, and local solubility. Let $C$ be the
intersection of $c^2=\mathcal F(u,v)$ and $d^2=\mathcal G(u,v)$ in
$\mathbb P^3$. If $M_1,M_2$ are the
symmetric matrices of these two quadrics, direct expansion gives
\[
 \det(\lambda M_1+\mu M_2)
 =-20p\lambda\mu(\lambda^2+3\lambda\mu+\mu^2).
\]
The four roots are distinct, so $C$ is a smooth geometrically
integral curve of genus one. The usual determinant construction for
a quadric intersection (see \cite[Section 7.3]{FisherInvariants}) identifies its
Jacobian with that of
$z^2=-20p\lambda\mu(\lambda^2+3\lambda\mu+\mu^2)$.
On setting $\mu=1$, $X=-20p\lambda$ and $Y=-20pz$, we obtain
\begin{equation}\label{eq:intermediatejac}
 E'_p:\quad Y^2=X^3-60pX^2+400p^2X.
\end{equation}
Writing $\xi=x-5p$ on $A_{5p}$ gives
$y^2=\xi^3+30p\xi^2+125p^2\xi$.
Consequently $E'_p=A_{5p}/\langle(5p,0)\rangle$.
Moreover $E'_p(\Q)[2]=\{O,(0,0)\}$, since the discriminant
of $X^2-60pX+400p^2$ is $2000p^2$.

The original Kummer covering for $\Lambda'=(1,5,5)$ is
\[
 B_{\Lambda'}:\quad
 u_1^2-5u_2^2=5pt_E^2,\qquad
 5u_3^2-u_1^2=20pt_E^2.
\]
Its covering map to $A_{5p}$ is
\[
 x=u_1^2/t_E^2,\qquad
 y=5u_1u_2u_3/t_E^3.
\]
The identity \eqref{eq:keyidentity} shows that
\[
 \pi:C\longrightarrow B_{\Lambda'},\qquad
 (u:v:c:d)\longmapsto
 \bigl(k(u^2-5v^2):c^2+d^2:cd:c^2-d^2\bigr)
\]
is a morphism. There are no base points: the last three coordinates
can vanish simultaneously only when $c=d=0$, which would contradict
$\operatorname{Res}(\mathcal F,\mathcal G)\ne0$.
The pullback of a hyperplane has degree $8$, whereas a hyperplane on
$B_{\Lambda'}$ has degree $4$. Thus $\deg\pi=2$.

We prove directly that adjoining $h$ in
$kh^2=H:=c^2-cd-d^2$ gives a geometrically connected unramified
double cover of $C$. Put $\phi=(1+\sqrt5)/2$ and
$A_0=2a+(a-5b)\phi$. On $C$ one has
\[
 (c-\phi d)(c+\phi d)
 =kA_0(u-\sqrt5v)^2,
\]
and the conjugate identity. The two factors on the left cannot
vanish at the same point, since that would imply $c=d=0$.
Hence the zero divisors of $c-\phi d$ and
$c-\bar\phi d$ are even. Since
$H=(c-\phi d)(c-\bar\phi d)$, the divisor of $H/v^2$ is even.

It is also not a square over $\overline{\Q}(C)$.
Indeed, a putative square root of $H/v^2$ has poles bounded by the
hyperplane divisor $(v=0)$, so it has the form $L/v$, with $L$ a
linear form in $u,v,c,d$. Thus
\[
 H=L^2+r(c^2-\mathcal F)+s(d^2-\mathcal G)
\]
for constants $r,s$. The nonzero $cd$ coefficient forces the
$c$ and $d$ coefficients of $L$ to be nonzero. The absent mixed
terms involving $u$ or $v$ force their coefficients in $L$ to vanish.
The linear independence of $\mathcal F,\mathcal G$ then forces
$r=s=0$. If $L=ec+fd$, coefficient comparison gives
$e^2=1$, $f^2=-1$ and $2ef=-1$, an impossibility.
The normalization $Y=\widetilde C_p$ is therefore geometrically connected, and
Riemann--Hurwitz gives $g(Y)=1$.

Let $J_Y$ be its Jacobian. The dual of the degree two isogeny
$J_Y\to E'_p$ has a rational subgroup of order two in $E'_p$;
this is necessarily its unique rational subgroup $\langle(0,0)\rangle$.
The same is true of the degree two isogeny $E'_p\to A_{5p}$
induced by $\pi$. It follows that $J_Y\cong A_{5p}$ over
$\Q$, with the isogenies chosen so that the composite
$J_Y\to E'_p\to A_{5p}$ is $[2]$.
Thus $Y\to B_{\Lambda'}$ is a $2$-covering and its composite with
$B_{\Lambda'}\to A_{5p}$ is a $4$-covering lifting $\Lambda'$.

It remains to prove local solubility. We begin with primes away from $10p$. Over
$\Z_q$ for $q\nmid10p$, the determinant above has distinct
roots modulo $q$, so $C$ has a smooth proper model. After the
unramified base extension adjoining $\sqrt5$, the same contact
identities have unit coefficients because $\Norm(A_0)=5p$.
They show that the quadratic extension defining $Y$ is unramified
also along the horizontal contact divisors and the special fibre.
The normalization is a finite etale double cover of this model.
The preceding nonsquare argument is valid in characteristic $q$
(the final contradiction is $5\ne0$), so its special fibre is
geometrically integral of genus one. It has an $\F_q$-point
by the Hasse bound, which lifts to a $\Q_q$-point.

The local calculation below supplies points at $2,5,p$ and at the real place.
Together with the good prime argument, this proves that $Y$ is everywhere
locally soluble and hence represents a $4$-Selmer lift of $\Lambda'$. Notice
that Proposition~\ref{arith:exist} established the rational point on the
conic projection independently.
\end{proof}

Fix a primitive integral point
\[
 P_0=(u_0,v_0,c_0)\in\Gamma_p(\Q),
 \qquad
 \gcd(u_0,v_0,c_0)=1.
\]
Then $5\nmid u_0c_0$.

\begin{theorem}\label{thm:higherpair}
With the notation above,
\begin{equation}\label{eq:higherpair}
 (-1)^{b_p}
 =
 -\varrho(p)
 =
 -\eps_a\leg{u_0}{5}
 =
 -k\leg{c_0}{5}.
\end{equation}
The value is independent of the choices of the primitive conic point
and the norm representation \eqref{eq:normab}.
\end{theorem}

\begin{proof}
The key point is to identify the divisor class of the pushout function; once
this is done, the local pairing becomes a short Hilbert symbol computation.
Write $Y=\widetilde C_p$ and
\[
 C:\quad c^2=\mathcal F(u,v),\qquad d^2=\mathcal G(u,v),
\]
and let $\pi:Y\to C$ forget $h$. Its deck transformation is
$\tau(h)=-h$. By the preceding covering construction, the kernel of
\[
 \pi_*:\operatorname{Pic}^0(Y)\longrightarrow\operatorname{Pic}^0(C)
\]
is generated by $T=(5p,0)\in A_{5p}[2]$.

Put $\phi=(1+\sqrt5)/2$ and define
\[
 (U,V)=
 \begin{cases}
 (c,-d),&k=1,\\
 (-c-d,c),&k=-1.
 \end{cases}
\]
Thus $U+\phi V=c-\phi d$ when $k=1$, and
$U+\phi V=(c-\phi d)/\phi$ when $k=-1$.
Since $\Norm(\phi)=-1$, both cases give
\begin{equation}\label{eq:pushoutconic}
 U^2+UV-V^2=h^2.
\end{equation}
Set
\begin{equation}\label{eq:explicitpushoutlinear}
 f=\frac{3U-V+2h}{5},\qquad f^\#=2U+V+2h.
\end{equation}
The form $f$ is a tangent line to the conic
\eqref{eq:pushoutconic} at $(U:V:h)=(1:1:-1)$.
In particular, its zero divisor on $Y$ is $2E_1$ for an effective
divisor $E_1$ of degree $4$.

Write $\mathcal F=F_Au^2+F_Buv+F_Cv^2$ and put
\begin{equation}\label{eq:tangent}
 \ell=c_0c-
 (F_Au_0+F_Bv_0/2)u-
 (F_Bu_0/2+F_Cv_0)v.
\end{equation}
This is the tangent line to $\Gamma_p$ at $P_0$.
If $D_0$ is the fibre over $P_0$ of $C\to\Gamma_p$, then
\[
 \operatorname{div}_0(\ell)|_C=2D_0,
 \qquad 2D_0\sim H_C,
\]
where $H_C$ is the hyperplane class of $C\subset\mathbb P^3$.
Consequently, with $E_0=\pi^*D_0$ and $H_Y=\pi^*H_C$, one has
\[
 \operatorname{div}_0(\ell)|_Y=2E_0,
 \qquad 2E_0\sim2E_1\sim H_Y.
\]
For any nonzero linear form $r$ let $H_r$ be its zero divisor on
$Y$. The rational function
\begin{equation}\label{eq:actualpushout}
 g=\frac{f\ell}{r^2}
\end{equation}
therefore satisfies
\[
 \operatorname{div}(g)=2(E_1+E_0-H_r),
 \qquad [E_1+E_0-H_r]=[E_1-E_0].
\]
We claim that this class is the nonzero point $T$.

First, the identity
\begin{equation}\label{eq:pushoutdecknorm}
 f\tau(f)=\frac{(U-V)^2}{5}
\end{equation}
implies $\pi_*E_1\sim H_C$. Also
$\pi_*E_0=2D_0\sim H_C$. Thus $[E_1-E_0]$ belongs to
$\ker(\pi_*)=\{0,T\}$.

To rule out the zero class, we use the following elementary nonsquare
calculation on $C_{\overline{\Q}}$.
Let $L(u,v)$ denote the polar form
\[
 L=(F_Au_0+F_Bv_0/2)u+(F_Bu_0/2+F_Cv_0)v,
 \qquad\ell=c_0c-L.
\]
Here $c_0\ne0$: otherwise $\mathcal F(u_0,v_0)=0$ would give a
rational root of a binary form of nonsquare discriminant $80p$.
For every $s\in\overline{\Q}^{\times}$, the function
\begin{equation}\label{eq:contactnonsquare}
 \frac{\ell(c-sd)}{v^2}
\end{equation}
is not a square in $\overline{\Q}(C)$.
Indeed, a square root would have poles bounded by the hyperplane
divisor $\{v=0\}$. The complete space of sections of $H_C$ is
spanned by $u,v,c,d$, so it would be $M/v$ with $M$ linear.
The quadratic part of the ideal of $C$ is generated by its two
defining quadrics. Hence there would be constants $R,S$ such that
\[
 \ell(c-sd)=M^2+R(c^2-\mathcal F)+S(d^2-\mathcal G).
\]
Comparing the mixed terms involving $d$, then those involving $c$,
gives
\[
 M=\lambda\ell+\mu d,\qquad
 2\lambda\mu=-s,\qquad
 \lambda^2=\frac1{2c_0},\qquad
 R=\frac{c_0}{2},\quad S=-\mu^2.
\]
The remaining binary terms would then imply
\[
 \mu^2\mathcal G
 =\frac{c_0^2\mathcal F-L^2}{2c_0}.
\]
The Gram determinant identity gives
\[
 c_0^2\mathcal F(u,v)-L(u,v)^2
 =\left(F_AF_C-\frac{F_B^2}{4}\right)(u_0v-v_0u)^2.
\]
The right side has rank one, whereas $\mathcal G$ is nondegenerate
and $\mu\ne0$. This contradiction proves
\eqref{eq:contactnonsquare}.

Now suppose that $g$ were a square in $\overline{\Q}(Y)$.
The choice of $r$ changes $g$ only by a square, so take $r=v$.
Work on $v\ne0$ and divide all linear coordinates by $v$; in the
next three equations we retain their names for these affine
coordinates. Since
\[
 \overline{\Q}(Y)=\overline{\Q}(C)(h),
\]
write a square root of $f\ell$ as $A+Bh$, with
$A,B\in\overline{\Q}(C)$. Comparing coefficients of $h$ gives
\[
 AB=\frac{\ell}{5},\qquad
 A^2+B^2h^2=\frac{\ell(3U-V)}5.
\]
Because
\[
 (3U-V)^2-4h^2=5(U-V)^2,
\]
it follows that
\[
 A^2=\frac{\ell}{10}
 \bigl(3U-V\mathbin{\pm}\sqrt5(U-V)\bigr).
\]
Each factor in parentheses is a nonzero constant times either
$U+\phi V$ or $U+\bar\phi V$. By the definition of $U,V$,
this would make one of
$\ell(c-\phi d)/v^2$ and $\ell(c-\bar\phi d)/v^2$
a square in $\overline{\Q}(C)$, contrary to
\eqref{eq:contactnonsquare}. Therefore $[E_1-E_0]\ne0$, and
$g$ is a pushout function for $T$.

It remains to evaluate the pushout locally. From
\eqref{eq:pushoutconic},
\begin{equation}\label{eq:normsimplify}
 ff^\#=(U+h)^2,\qquad
 \Norm_{\Q(\sqrt5)/\Q}(U+\phi V+h)=hf^\#.
\end{equation}
At any local point away from these zeros, the norm criterion for
the Hilbert symbol gives
\[
 (f,5)_\nu=(f^\#,5)_\nu=(h,5)_\nu.
\]
In the Kummer coordinates $(x,x-5p,x+20p)$, the class
$\Lambda=(5,1,5)$ is represented by the cocycle
$\chi_5 T$, where $\chi_5$ is the quadratic character of
$\Q(\sqrt5)/\Q$ and $T=(5p,0)$. Indeed, its Weil pairings
with the three nonzero $2$-torsion points have precisely the
squareclasses $(5,1,5)$. The pushout formula for the
Cassels--Tate pairing, applied to the $4$-Selmer lift of
$\Lambda'=(1,5,5)$, now yields
\begin{equation}\label{eq:localpair}
 (-1)^{b_p}
 =\prod_\nu(g(P_\nu),5)_\nu
 =\prod_\nu(h(P_\nu)\ell(P_\nu),5)_\nu.
\end{equation}
The local points may be chosen away from the zeroes and poles of the displayed
functions. The pushout formula is the one of \cite[Sections 7--8]{Fisher};
the divisor class and Kummer cocycle have been identified above. Projective
rescaling changes $h\ell$ only by a square, so the Hilbert symbols are
well defined. It remains to compute the finitely many local contributions.

At $2$, put $A=a+5b$. One has $kA\equiv1\pmod4$.
According to $kA\bmod8$ and $p\bmod16$, choose
\[
\begin{array}{c|c|c}
kA\bmod8&p\bmod16&(u,v)\\ \hline
1&3&(1,0)\\
1&11&(1,2)\\
5&3&(2,1)\\
5&11&(0,1).
\end{array}
\]
In each case
\[
 \mathcal F(u,v)\equiv\mathcal G(u,v)\equiv1\pmod8,
 \qquad
 \mathcal F(u,v)\mathcal G(u,v)\equiv1\pmod{16}.
\]
Choose odd square roots $c,d$ with $cd\equiv-k\pmod8$. Then
\[
 \frac{c^2-cd-d^2}{k}\equiv1\pmod8,
\]
so $h$ is an odd $2$-adic unit. The proof of Lemma~\ref{arith:dyadic} applies equally to primitive
$\Z_2$-points with nonzero third coordinate. Choose the local point
away from the tangent zero; it gives
\[
 v_2(\ell)\equiv1\pmod2.
\]
Consequently
\[
 (h\ell,5)_2=-1.
\]

At $5$, take $u=v=1$. Modulo $5$,
\[
 c^2=ka,\qquad d^2=-ka.
\]
Choose
\[
 c\equiv-c_0/u_0\pmod5,\qquad d\equiv3c\pmod5.
\]
Then
\[
 h^2\equiv-a\pmod5,
\qquad
 \leg{h}{5}=-\eps_a,
\]
and a direct reduction of \eqref{eq:tangent} gives
\[
 \ell\equiv-2ka u_0\pmod5,
 \qquad
 \leg{\ell}{5}=-\leg{u_0}{5}.
\]
Hence
\[
 (h\ell,5)_5=\eps_a\leg{u_0}{5}=\varrho(p).
\]

At $p$, choose $u=1,v=0$. Quadratic reciprocity gives
\[
 \leg{b}{p}=k,\qquad
 \leg{a+b}{p}=1,\qquad
 \leg{a}{p}=1.
\]
These imply that both $\mathcal F(1,0)$ and
$\mathcal G(1,0)$ are squares modulo $p$.
If $r=a/b\pmod p$ and $\phi_r=(1+r)/2$, one has
$\mathcal F/\mathcal G=\phi_r^2$. Choosing
$c/d=-\phi_r$ gives
\[
 c^2-cd-d^2=2\phi_r d^2,
 \qquad
 \leg{\phi_r}{p}=-k,
\]
so the third equation is also soluble.
Because $\leg5p=1$, the local pairing contribution is $1$.

At the real place, $F_p-G_p$ is twice the positive definite form
\[
 au^2+10buv+5av^2
\]
of discriminant $-20p$, while $F_p,G_p$ are indefinite.
If $k=1$, choose $(u,v)$ with $G_p>0$. Then $c^2>d^2>0$;
choosing $cd<0$ makes $c^2-cd-d^2>0$.
If $k=-1$, choose $(u,v)$ with $F_p<0$. Then $0<c^2<d^2$;
choosing $cd>0$ makes $c^2-cd-d^2<0$.
In either case $(c^2-cd-d^2)/k>0$, so $h$ exists over $\mathbf R$.
The Hilbert symbol is trivial since $5>0$.

Finally, for $q\nmid10p$, the resultant
\eqref{eq:discriminants} is a unit. Choose primitive integral local
coordinates $(u,v,c,d,h)$. The projection $(u,v,c)$ is primitive:
if $q$ divided all three, the equations would force it to divide
$d,h$ as well. If $5$ is a square locally there is nothing to prove.
Otherwise $c^2-cd-d^2$ is anisotropic modulo $q$. If $q\mid h$,
then $q\mid c,d$, and the unit resultant would force
$q\mid u,v$, a contradiction. Thus $h$ is a unit.
The conic has good reduction at $q$ and its primitive tangent line
has no vertical component; its intersection divisor is twice the
section $P_0$. Every nonzero tangent value therefore has even
valuation. Since $\Q_q(\sqrt5)/\Q_q$ is unramified,
$(h\ell,5)_q=1$.
Multiplying the local contributions gives
\[
 (-1)^{b_p}=-\varrho(p).
\]
The alternative expression in \eqref{eq:higherpair} is
\eqref{arith:eps}, and Theorems~\ref{arith:pointind} and
\ref{arith:normind} show that the resulting sign is independent of all
auxiliary choices.
\end{proof}

\subsection{Factorial and Pell descriptions}\label{subsec:factorialpell}

\begin{corollary}\label{cor:rho}
If
\[
 (-1)^{(b-1)/2}\leg{c_0}{5}=1,
\]
then
\[
 \rank A_{5p}(\Q)=0,\qquad
 \Sh(A_{5p})[2^\infty]\simeq(\Z/4\Z)^2.
\]
\end{corollary}

To compare the higher pairing with the representation defect, put
\[
 N=\frac{p-1}{5},\qquad
 \chi(t)=\leg{t}{p}.
\]

\begin{theorem}\label{thm:factorial}
One has the exact identity
\begin{equation}\label{eq:classshortsum}
 \sum_{j=1}^{N}\chi(j)
 =
 h(-p)-\frac{h(-5p)}4.
\end{equation}
Consequently,
\begin{equation}\label{eq:factorial}
 (-1)^{\Xi(p)}
 =
 -\leg{N!}{p}.
\end{equation}
\end{theorem}

\begin{proof}
For a negative fundamental discriminant $-m<-4$, Dirichlet's finite
class number formula is
\[
 h(-m)
 =
 -\frac1m\sum_{j=1}^{m-1}j\chi_{-m}(j).
\]
Let
\[
 I_j=\sum_{jN<a\le(j+1)N}\chi(a),
 \qquad 0\le j\le4.
\]
Since $\chi(-1)=-1$,
\[
 I_4=-I_0,\qquad I_3=-I_1,\qquad I_2=0.
\]
Set
\[
 W=\sum_{a=1}^{p-1}a\chi(a).
\]
Because $\chi(5)=1$, multiplication by $5$ permutes
$\F_p^\times$, giving
\[
 W=
 \sum_{a=1}^{p-1}
 \left(5a-p\left\lfloor\frac{5a}{p}\right\rfloor\right)\chi(a).
\]
Thus
\[
 4W
 =
 p\sum_{j=0}^4jI_j
 =
 p(-4I_0-2I_1),
\]
and hence
\begin{equation}\label{eq:hminusP}
 h(-p)=I_0+\frac{I_1}{2}.
\end{equation}

For $h(-5p)$, write $m=a+tp$ with
$1\le a\le p-1$ and $0\le t\le4$.
Since $p\equiv1\pmod5$ and
$\sum_t\chi_5(a+t)=0$,
\[
 h(-5p)
 =
 -\frac15\sum_{a=1}^{p-1}
 \chi(a)T_{a\bmod5},
 \qquad
 T_r=\sum_{t=0}^4t\chi_5(r+t).
\]
The five values are
\[
 (T_0,T_1,T_2,T_3,T_4)=(0,0,5,0,-5).
\]
If
\[
 R_r=\sum_{\substack{1\le a<p\\a\equiv r\pmod5}}\chi(a),
\]
then
\[
 h(-5p)=R_4-R_2=2R_4,
\]
the last equality following from $a\mapsto p-a$.
Since
\[
 5^{-1}\equiv-N\pmod p,
\]
one has
\[
 \chi(5j+4)=\chi(j+N+1),
 \qquad 0\le j<N,
\]
so $R_4=I_1$. Combining with \eqref{eq:hminusP} gives
\eqref{eq:classshortsum}.

Now $N=2n$ with
\[
 n=\frac{p-1}{10}
\]
odd. If $R$ denotes the number of quadratic nonresidues among
$1,\dots,N$, then
\[
 R=\frac{N-\sum_{j=1}^N\chi(j)}2
  =n-\frac12\left(h(-p)-\frac{h(-5p)}4\right).
\]
Therefore
\[
 \leg{N!}{p}=(-1)^R
 =
 -(-1)^{\Xi(p)}.
\]
\end{proof}

The factorial character and the intrinsic conic character admit a common
description by a Pell equation. We retain the notation $h(D)$ for the imaginary quadratic class
number of fundamental discriminant $D<0$.

\begin{theorem}\label{thm:reciprocity}
Let $p\equiv11\pmod{40}$ be prime, and choose
$p=a^2-5b^2$ with $a,b>0$, $\gcd(a,b)=1$, $4\mid a$, and $b$ odd.
Put $k=(-1)^{(b-1)/2}$. Write
$T+U\sqrt{5p}>1$ for the fundamental positive unit of
$\Q(\sqrt{5p})$. There are positive integers $d,y$ such that
\begin{equation}\label{rec:pellfactor}
 T=10d^2-1,\qquad U=2dy,\qquad 5d^2-py^2=1,
 \qquad d\text{ odd},\quad y\equiv2\pmod4.
\end{equation}
For every primitive integral point $(u_0,v_0,c_0)$ on
\begin{equation}\label{rec:conic}
 c^2=k\bigl((a+5b)u^2+10(a+b)uv+5(a+5b)v^2\bigr),
\end{equation}
one has
\begin{equation}\label{eq:reciprocity}
 \leg{((p-1)/5)!}{p}
 =\varrho(p)=(-1)^{(d-1)/2}
 =k\leg{c_0}{5}.
\end{equation}
For each normalized norm representation, the conic has a primitive
point with $c_0=1$ or $2$. Moreover,
\begin{equation}\label{rec:classcongruence}
 h(-5p)-4h(-p)\equiv4(d+1)\pmod{16}.
\end{equation}
Consequently,
\begin{equation}\label{eq:primehigherbridge}
 \frac{\D(p)}{16}\equiv\Xi(p)
 \equiv\Pf(C_p^{\mathrm{next}})=b_p
 \equiv\frac{d+1}{2}\pmod2.
\end{equation}
The residue $d\pmod4$ is the same for every positive integral solution
of $5d^2-py^2=1$.
\end{theorem}

The proof has three ingredients. First the fundamental unit of
$\Q(\sqrt{5p})$ gives the Pell equation in \eqref{rec:pellfactor} and fixes
the conic character. Next the Meyer--Zagier genus formula is applied twice
to express the class number difference as a signed sum over the ordinary
real class group. Finally the terms paired by inversion cancel modulo eight,
leaving the ramified class, whose Dedekind contribution can be evaluated
directly.

Set
\[
N=5p,\qquad K_p=\Q(\sqrt N),\qquad \mathcal O_{K_p}=\Z[\sqrt N].
\]
Since \(N\equiv7\pmod8\), the field discriminant is \(4N=20p\).
A unit of norm \(-1\) would give \(x^2\equiv-1\pmod p\), which is impossible
because \(p\equiv3\pmod4\). Thus the fundamental positive unit has norm
\(+1\); write it as \(T+U\sqrt N\) with \(T,U>0\).

\begin{lemma}\label{rec:pellshape}
The factorization and parity assertions in \eqref{rec:pellfactor} hold.
If \({r_p}=y/2\), then
\begin{equation}\label{rec:pellmod}
5d^2-4p{r_p}^2=1,\quad d,{r_p}\text{ odd},\qquad
T\equiv9\pmod{16},\quad U=4d{r_p}\equiv4\pmod8.
\end{equation}
\end{lemma}
\begin{proof}
Suppose first that \(T\) is even. Then \(U\) is odd, and the coprime odd
integers \(T+1,T-1\) have product \(5pU^2\). Allocating their squarefree parts
in all possible ways yields one of
\[
r^2-5ps^2=2,\quad5pr^2-s^2=2,\quad
5r^2-ps^2=2,\quad pr^2-5s^2=2.
\]
Every case is impossible modulo \(5\), since \(p\equiv1\pmod5\) and neither
\(2\) nor \(3\) is a square modulo \(5\). Hence \(T\) is odd and \(U\) even.

The consecutive coprime integers \((T+1)/2,(T-1)/2\) now have product
\(5p(U/2)^2\). If their squarefree parts are \((1,5p)\), we obtain
\(r^2-5ps^2=1\) and
\[
T+U\sqrt{5p}=(r+s\sqrt{5p})^2,
\]
contrary to fundamentality. The assignment \((5p,1)\) is impossible modulo
\(p\), because \(-1\) is not a square there. The assignment \((p,5)\) is
also impossible modulo \(p\), because
\(\leg{-5}{p}=-1\). The only remaining assignment is
\[
\frac{T+1}{2}=5d^2,\qquad\frac{T-1}{2}=py^2.
\]
We obtain \(5d^2-py^2=1\) and \(U=2dy\).

Modulo \(5\), one has \(y^2\equiv-1\), so \(\leg y5=-1\).
If \(y\) were odd, then every prime divisor \(\ell\) of \(y\) would satisfy
\(5d^2\equiv1\pmod\ell\), hence \(\leg5\ell=1\). Thus the Jacobi symbol
\(\leg5y\) would be \(1\), contradicting quadratic reciprocity
\(\leg5y=\leg y5=-1\). Therefore \(y\) is even, and the norm equation forces
\(d\) odd. Modulo \(8\) it gives \(y^2\equiv4\pmod8\), so
\(y\equiv2\pmod4\). The final congruences follow directly from
\(T=10d^2-1\) and \(U=2dy\).
\end{proof}

Keep the normalized representation in Theorem~\ref{thm:reciprocity},
and put $A_0=a+5b$, $B_0=a+b$. We construct a primitive
point of \eqref{rec:conic} whose last coordinate is $1$ or $2$.

Put
\[
K=\Q(\sqrt5),\qquad\phi=\frac{1+\sqrt5}{2},\qquad
\mathcal O_K=\Z[\phi],\qquad\eps_5=\phi^6=9+4\sqrt5.
\]
The Minkowski bound \(\sqrt5/2<2\) gives class number one. The unit group is
\(\mathcal O_K^\times=\{\pm\phi^m:m\in\Z\}\); consequently every totally positive
unit is a square. For the ideal and unit theory, see \cite[Chapters 4--5]{MilneANT}.

Because \(\leg5p=1\), the prime \(p\) splits. Each of its two prime ideals
has a totally positive generator of norm \(p\): the norm \(-1\) unit \(\phi\)
and an overall sign arrange the two signs. Since \(2\) is inert,
\(\mathcal O_K/2\mathcal O_K=\F_4\), and the image of \(\phi\) has order \(3\).
Multiplication by a suitable even power of \(\phi\) makes the generator
congruent to \(1\) modulo \(2\mathcal O_K\).

An element of \(1+2\mathcal O_K\) has the form \(a+b\sqrt5\), with integral
coordinates of opposite parity. The norm condition \(a^2-5b^2=p\equiv3\pmod8\)
then forces \(4\mid a\) and \(b\) odd. Multiplication by a sufficiently large
power of \(\eps_5\) makes \(b>0\), without changing these properties.
Primality of the norm gives \(\gcd(a,b)=1\). Thus admissible generators exist
for \emph{each} prime ideal above \(p\).

For a fixed prime ideal, any two admissible generators differ by
\(\eps_5^j\): they are totally positive, have the same norm, and are both
\(1\) modulo \(2\mathcal O_K\). The transformation
\[
(a,b)\longmapsto(9a+20b,\,4a+9b)
\]
preserves \(b\pmod4\). Conjugation switches the two prime ideals and replaces
\(b\) by \(-b\); it therefore reverses \(k=(-1)^{(b-1)/2}\). A subsequent power of
\(\eps_5\) restores \(b>0\) without changing this sign.

For the positive integers \(d,y\) in Lemma~\ref{rec:pellshape}, define
\[
\gamma_+=1+d\sqrt5,\qquad\gamma_-=1-d\sqrt5,
\qquad y=2{r_p}.
\]
Their product is
\begin{equation}\label{rec:gammanorm}
\gamma_+\gamma_-=1-5d^2=-py^2=-4p{r_p}^2.
\end{equation}
The two ideals \((\gamma_+)\), \((\gamma_-)\) have no common prime divisor
away from \(2\), since their sum is \(2\). At the inert prime \(2\), each
has valuation exactly one, by \eqref{rec:gammanorm} and \({r_p}\) odd.
The prime over \(5\) divides neither. Hence exactly one prime
ideal over \(p\) occurs to odd exponent in \((\gamma_+)\), and every other
odd prime exponent is even. Inert odd primes cannot occur, since they would
divide both conjugate ideals.

Choose an admissible generator \(\alpha_+=a_++b_+\sqrt5\) of this particular
prime ideal. The element
\[
\frac{\gamma_+}{(1+\sqrt5)\alpha_+}
\]
is integral, is a unit at \(2\), has a square ideal, and is totally positive.
Indeed, \(1+\sqrt5=2\phi\) has valuation one at \(2\), and both numerator
and denominator have signs \((+,-)\) in the two embeddings. Since the class
number is one and every totally positive unit is a square, there is
\(z\in\mathcal O_K\), a unit at \(2\), with
\begin{equation}\label{rec:rootsplus}
(1+\sqrt5)\alpha_+z^2=1+d\sqrt5.
\end{equation}

\begin{lemma}\label{rec:orientation}
The generator \(\alpha_+\) necessarily has \(b_+\equiv1\pmod4\).
For every admissible \(\alpha=a+b\sqrt5\), with its prescribed \(k\),
there is a square root
\begin{equation}\label{rec:Kroot}
z^2=\frac{1+kd\sqrt5}{k(1+\sqrt5)(a+b\sqrt5)},
\qquad z\in\mathcal O_K,\quad z\notin2\mathcal O_K.
\end{equation}
\end{lemma}
\begin{proof}
Write \(A_+=a_++5b_+\), \(B_+=a_++b_+\). If \(z=u+v\sqrt5\) has integral
coordinates, its being a unit at \(2\) forces \(u,v\) to have opposite
parity. Comparing rational parts in \eqref{rec:rootsplus} gives
\[
A_+(u^2+5v^2)+10B_+uv=1.
\]
Modulo \(4\), its left side is \(b_+(u+v)^2\), so \(b_+\equiv1\pmod4\).

Otherwise \(z=(r+s\sqrt5)/2\), with \(r,s\) odd, and the analogous equation
has right side \(4\). Use
\begin{equation}\label{rec:Qrewrite}
A_+(r^2+5s^2)+10B_+rs
=A_+\{(r+s)^2+4s^2\}+8a_+rs.
\end{equation}
If \(r+s\equiv2\pmod4\), this is \(8\pmod{16}\), not \(4\).
Thus \(4\mid r+s\), and reduction modulo \(16\) gives
\(4b_+\equiv4\pmod{16}\). Again \(b_+\equiv1\pmod4\).

All admissible generators of this prime ideal have \(k=1\), because powers
of \(\eps_5\) preserve \(b\pmod4\). The other prime orientation has \(k=-1\)
by conjugation. If \(k=1\), the same square ideal argument as above applies
to \eqref{rec:Kroot}. If \(k=-1\), its numerator is \(\gamma_-\), its odd
prime of odd exponent is the other prime above \(p\), and both numerator
and denominator now have signs \((-,+)\). The quotient is again integral,
totally positive, a square ideal, and a unit at \(2\). It is therefore a
square of an element of \(\mathcal O_K\), as asserted.
\end{proof}

The identity
\[
(1+\sqrt5)(a+b\sqrt5)=A_0+B_0\sqrt5
\]
shows that the rational part of \(k(A_0+B_0\sqrt5)z^2\) is the form in
\eqref{rec:conic}. Thus an integral coordinate root \(z=u+v\sqrt5\) in
\eqref{rec:Kroot} yields \((u,v,1)\), whereas a half-integral root
\(z=(r+s\sqrt5)/2\), with \(r,s\) odd, yields \((r,s,2)\).
Both triples are primitive.

To distinguish the cases, rewrite \eqref{rec:Kroot} as
\[
z^2=\frac{(1-kd)/2+kd\phi}{k\phi\alpha}.
\]
Modulo \(2\mathcal O_K\), one has \(\alpha=1\), \(k=1\), and \(kd=1\).
If \(kd\equiv1\pmod4\), then \(z^2=1\) in \(\F_4\), hence \(z=1\),
which means that \(z\) has integral \(\sqrt5\)-coordinates. If
\(kd\equiv3\pmod4\), then
\[
z^2=\frac{1+\phi}{\phi}=\phi,
\qquad z=\phi^2\quad\text{in }\F_4,
\]
which means that both coordinates of \(2z\) are odd. Thus $c=1$ when $kd\equiv1\pmod4$, and $c=2$ when $kd\equiv3\pmod4$. Because \(\leg25=-1\), the constructed point has
\begin{equation}\label{rec:constructedI}
k\leg c5=k(-1)^{(kd-1)/2}=(-1)^{(d-1)/2}.
\end{equation}
Theorems~\ref{arith:pointind} and~\ref{arith:normind} prove that
the character is unchanged by the choice of primitive point or norm
representation. Hence this equality holds for every primitive point in
Theorem~\ref{thm:reciprocity}.

We apply the genus formula in the real quadratic field
\(K_p=\Q(\sqrt{5p})\).
Put $H=h(-p)$ and $H_5=h(-5p)$. Let \(\mathcal H_p=\Cl(K_p)\) be the \emph{ordinary} ideal class group.
Its discriminant is
\begin{equation}\label{rec:twofactorizations}
20p=(-4)(-5p)=(-20)(-p).
\end{equation}
Both are decompositions into negative fundamental discriminants.

For coprime integers \(h,k\), \(k>0\), define the Dedekind sum
\[
s(h,k)=\sum_{r=1}^{k-1}\left(\frac r k-\frac12\right)
\left(\frac{hr}{k}-\left\lfloor\frac{hr}{k}\right\rfloor-\frac12\right).
\]
For a matrix \(M=\left(\begin{smallmatrix}\mu&\nu\\\lambda&\tau\end{smallmatrix}\right)\)
with determinant \(1\), positive trace greater than \(2\), and \(\lambda>0\), put
\begin{equation}\label{rec:nmatrix}
n(M)=\frac{\mu+\tau}{\lambda}-3-12s(\tau,\lambda).
\end{equation}
For an integral ideal \(I\), let \(n(I)\) be this integer for multiplication
by the fundamental positive unit on an oriented basis. In particular, for
\(I=(a+\sqrt N,b)\), \(b>0\), \(b\mid a^2-N\), the matrix is
\begin{equation}\label{rec:idealmatrix}
M_I=\begin{pmatrix}
T+aU&(N-a^2)U/b\\ bU&T-aU
\end{pmatrix}.
\end{equation}
The basis is ordered as \((a+\sqrt N,b)\); this fixes the orientation and sign.

We use the Meyer--Zagier genus formula in the following normalization:
\begin{equation}\label{rec:zagier}
h(D_1)h(D_2)=\frac{w(D_1)w(D_2)}{24}
\sum_{C\in\mathcal H_p}\chi(I_C)n(I_C),\qquad D_1D_2=20p.
\end{equation}
Here \(D_1,D_2<0\) are fundamental, \(w(D)\) is the number of roots of unity,
and \(\chi\) is their genus character. The sum is over ordinary classes, with
one ideal representative for each. Zagier's equation (23), p.~89,
uses the weighted convention \(\widetilde h(D)=2h(D)/w(D)\) and coefficient
\(1/6\); converting it gives exactly \eqref{rec:zagier}
\cite[pp.~87--89]{Zagier}. This conversion is essential at \(D_1=-4\).

Genus characters here obey
\(\chi((\alpha))=\operatorname{sgn}\Norm_{K_p/\Q}(\alpha)\).
Meyer's signed ideal series is
\[
L_I(s)=\sum_{\beta\in(I\setminus\{0\})/\langle\eps_N\rangle}
\frac{\operatorname{sgn}\Norm(\beta)}{|\Norm(\beta)|^s},\qquad L_I(0)=n(I)/3,
\]
where \(\eps_N=T+U\sqrt N\), and the value at zero is by analytic continuation. Scaling by \(\alpha\)
multiplies the series by \(\operatorname{sgn}\Norm(\alpha)|\Norm(\alpha)|^{-s}\). Hence
\(n(\alpha I)=\operatorname{sgn}\Norm(\alpha)n(I)\), and the product \(\chi(I)n(I)\) is
independent of the representative of an ordinary class. The explicit
Dedekind formulation and this invariant are also used in
\cite[Sections 2--3]{ChuaGuy}.

Let \(\chi_0,\chi_1\) correspond respectively to the two factorizations in
\eqref{rec:twofactorizations}, and let
\[
\psi=\chi_1/\chi_0.
\]
For ideals whose norm is prime to \(20p\), the characters are explicitly
\[
\chi_0(I)=\leg{-4}{\Norm I},\qquad
\chi_1(I)=\leg{-20}{\Norm I},\qquad
\psi(I)=\leg5{\Norm I}.
\]
The ratio \(\psi\) is an ordinary class character, since the two numerator
characters have the same norm signature on principal ideals. At ramified
primes, a genus character
is evaluated using the factor not divisible by that prime. In particular,
for \(\mathfrak b=(2,1+\sqrt N)\),
\begin{equation}\label{rec:charactersB}
\chi_0(\mathfrak b)=\leg{-5p}{2}=1,\qquad
\chi_1(\mathfrak b)=\leg{-p}{2}=-1,\qquad
\psi(\mathfrak b)=-1.
\end{equation}
The denominator-\(2\) symbols are Kronecker symbols.
Define
\[
t_C=\chi_0(I_C)n(I_C).
\]
Since \(w(-4)=4\), \(w(-5p)=w(-20)=w(-p)=2\),
\(h(-4)=1\), and \(h(-20)=2\), the two specializations of
\eqref{rec:zagier} are
\begin{equation}\label{rec:twoZ}
3H_5=\sum_{C\in\mathcal H_p}t_C,\qquad
12H=\sum_{C\in\mathcal H_p}\psi(C)t_C.
\end{equation}
The value \(h(-20)=2\) follows from the two primitive reduced
forms \((1,0,5)\) and \((2,2,3)\).
Subtracting gives the exact identity
\begin{equation}\label{rec:exactdifference}
3(H_5-4H)=2\sum_{\psi(C)=-1}t_C.
\end{equation}
The difference restricts the real class group sum to the negative genus.

We use two elementary Dedekind identities. For positive coprime
\(\lambda,\tau\),
\begin{equation}\label{rec:dedrec}
s(\lambda,\tau)+s(\tau,\lambda)
=-\frac14+\frac1{12}\left(\frac\lambda\tau+
\frac\tau\lambda+\frac1{\lambda\tau}\right).
\end{equation}
For positive odd \(\tau\),
\begin{equation}\label{rec:jacded}
E(\lambda,\tau):=\frac{\tau-1}{4}-3\tau s(\lambda,\tau)\in\Z,
\qquad \leg\lambda\tau=(-1)^{E(\lambda,\tau)}.
\end{equation}
These are the reciprocity law and the Jacobi--Dedekind identity;
see \cite[p.~83]{Zagier} and \cite{ChuaGuy}.

\begin{lemma}\label{rec:tpropertieslemma}
For every ordinary ideal class \(C\),
\begin{equation}\label{rec:tproperties}
t_C\equiv0\pmod4,\qquad t_{C^{-1}}=t_C.
\end{equation}
\end{lemma}
\begin{proof}
Choose an integral ideal representative coprime to \(2N\), and remove all
rational ideal factors to make it primitive. It then has a basis
\(I=(a+\sqrt N,b)\), where \(b=\Norm I\) is positive and odd.
The existence of a representative avoiding finitely many primes follows
by choosing a scalar with prescribed valuations, or equivalently by ideal
approximation. A primitive integral ideal of \(\Z[\sqrt N]\) has the displayed
basis by Hermite normal form: in a basis \(m,r+s\sqrt N\), stability under
multiplication by \(\sqrt N\) forces \(s\mid m\) and \(s\mid r\).
Primitivity therefore forces \(s=1\).

If \(b\equiv3\pmod4\), multiply by the principal ideal \((2+\sqrt N)\),
then remove rational factors again. Its absolute norm is
\(N-4\equiv3\pmod4\), coprime to \(2N\), so this operation makes the new
norm \(1\pmod4\). Removal of odd rational factors divides the norm by an
odd square and does not change that residue. Hence we may take
\begin{equation}\label{rec:chosenrep}
b\equiv1\pmod4,\qquad\gcd(b,2N)=1,\qquad\chi_0(I)=1.
\end{equation}
Replacing \(a\) by \(a-jb\) for a sufficiently large integer \(j\) also
makes \(\tau=T-aU>0\), without changing the ideal.

For the matrix \eqref{rec:idealmatrix}, \eqref{rec:pellmod} gives
\[
\nu\equiv\lambda\equiv0\pmod4,\qquad\tau\equiv1\pmod4.
\]
Combining \eqref{rec:nmatrix}, \eqref{rec:dedrec}, and
\(\mu\tau-\nu\lambda=1\) yields
\begin{align}\label{rec:keyded}
\tau n(I)&=\nu-\lambda+12\tau s(\lambda,\tau)\\
&=\nu-\lambda+\tau-1-4E(\lambda,\tau).\nonumber
\end{align}
Every term on the last line is divisible by \(4\). Since \(\tau\) is odd,
\(n(I)\), and hence \(t_C\), is divisible by \(4\).

The conjugate ideal has the oriented basis \((-a+\sqrt N,b)\), and represents
\(C^{-1}\). Its matrix has the same trace and lower left entry, but its
lower right entry is \(T+aU\) in place of \(T-aU\). These two entries are
inverses modulo \(bU\), because
\[
(T-aU)(T+aU)=1+(N-a^2)U^2\equiv1\pmod{bU}.
\]
The elementary identity \(s(h^{-1},k)=s(h,k)\) therefore gives equal
\(n\)-values. The two ideals have the same genus character, so
\(t_{C^{-1}}=t_C\).
\end{proof}

Write \(\mathcal B=[\mathfrak b]\), where
\(\mathfrak b=(2,1+\sqrt N)\).

\begin{lemma}\label{rec:twotorsion}
The self inverse ordinary classes are precisely
\begin{equation}\label{rec:Gtwo}
\mathcal H_p[2]=\{1,\mathcal B\},\qquad\mathcal B\ne1,\qquad\psi(\mathcal B)=-1.
\end{equation}
\end{lemma}
\begin{proof}
First, \(-1\) is not a rational norm from \(K_p\). Otherwise clearing
denominators would give a primitive integer solution
\(x^2-Nz^2=-w^2\). Modulo \(p\), \(x^2+w^2=0\) forces \(p\mid x,w\),
and the equation then forces \(p\mid z\), a contradiction.

Suppose \([I]^2=1\). Since \([\bar I]=[I]^{-1}\), write
\(I=\gamma\bar I\). Equality of absolute ideal norms gives
\(\Norm(\gamma)=\pm1\). The negative case was excluded, so Hilbert 90 gives
\(\gamma=\delta/\bar\delta\). For this quadratic case, this is explicit:
if \(\gamma\ne-1\), take \(\delta=1+\gamma\); if \(\gamma=-1\), take
\(\delta=\sqrt N\). Thus \(I/\delta\) is fixed by conjugation.

A conjugation fixed fractional ideal, modulo a rational principal ideal,
is a product of ramified prime ideals. Split prime exponents occur in equal
pairs, inert primes are already rational ideals, and ramified exponents can
be reduced modulo two. Here the ramified primes lie above \(2,5,p\).
The prime above \(5\) is principal, because
\begin{equation}\label{rec:prime5principal}
\eta=5d+y\sqrt N\in\mathcal O_{K_p},\qquad\Norm(\eta)=5.
\end{equation}
The prime above \(p\) is then principal as well, since their product is
\((\sqrt N)\). Hence every self inverse class is either \(1\) or
\(\mathcal B\).

Finally, \(\mathfrak b^2=(2)\), while \(\mathfrak b\) is not principal.
An integral generator would have field norm \(\pm2\), which would give
\(x^2\equiv\pm2\pmod5\), impossible. The value of \(\psi(\mathcal B)\)
is already computed in \eqref{rec:charactersB}.
\end{proof}

\begin{proposition}\label{rec:pairing}
Let \(n_{\mathfrak b}=n(\mathfrak b)\). Then
\begin{equation}\label{rec:pairingresult}
\sum_{\psi(C)=-1}t_C\equiv n_{\mathfrak b}\pmod8,
\qquad
3(H_5-4H)\equiv2n_{\mathfrak b}\pmod{16}.
\end{equation}
\end{proposition}
\begin{proof}
The set \(\psi=-1\) is stable under inversion. Each nonfixed pair contributes
\(2t_C\), a multiple of \(8\) by Lemma~\ref{rec:tpropertieslemma}.
Lemma~\ref{rec:twotorsion} shows that \(\mathcal B\) is its only fixed point.
Since \(\chi_0(\mathfrak b)=1\), its contribution is \(n_{\mathfrak b}\).
The second congruence follows from \eqref{rec:exactdifference}.
\end{proof}
Thus the class number difference modulo $16$ depends only on the
ramified class $\mathcal B$, even when the real class group has
nontrivial odd part.

\begin{proposition}\label{rec:bridge}
For the Pell coordinate $d$ in \eqref{rec:pellfactor},
\[
 h(-5p)-4h(-p)\equiv4(d+1)\pmod{16}.
\]
\end{proposition}
\begin{proof}
The basis \((1+\sqrt N,2)\) of \(\mathfrak b\) gives
\begin{equation}\label{rec:BM}
M_{\mathfrak b}=\begin{pmatrix}
T+U&(N-1)U/2\\2U&T-U
\end{pmatrix}.
\end{equation}
In the notation of \eqref{rec:nmatrix},
\[
\lambda=2U=8d{r_p},\qquad\tau=T-U=10d^2-1-4d{r_p},
\qquad\nu=(N-1)U/2.
\]
Here \(\tau>0\) because \(T>U\), and \(\tau\equiv5\pmod8\).
Also
\begin{align*}
\nu-\lambda+\tau-1
&=(N-5)U/2+T-U-1\equiv0\pmod8.
\end{align*}
Indeed, \((N-5)U/2=5(p-1)U/2\equiv4\pmod8\), and
\(T-U-1\equiv4\pmod8\).
Using \eqref{rec:keyded}, therefore,
\begin{equation}\label{rec:nBjacobi}
(-1)^{n_{\mathfrak b}/4}
=(-1)^{E(\lambda,\tau)}=\leg\lambda\tau.
\end{equation}
To justify division by four here, the same identity first shows
\(4\mid n_{\mathfrak b}\). Modulo \(8\), multiplication by the odd
\(\tau\) does not change the parity of \(n_{\mathfrak b}/4\).

It remains to evaluate an ordinary Jacobi symbol. Since \(\tau\equiv5\pmod8\)
and \(d,{r_p}\) are positive odd integers,
\[
\leg\lambda\tau=\leg8\tau\leg d\tau\leg{r_p}\tau.
\]
Quadratic reciprocity applies without a sign in the last two factors because
\(\tau\equiv1\pmod4\). Moreover,
\[
\tau\equiv-1\pmod d,\qquad\tau\equiv T\equiv1\pmod{r_p},
\]
where \(T\equiv1\pmod{r_p}\) follows from \(5d^2-4p{r_p}^2=1\).
Thus
\begin{equation}\label{rec:finaljacobi}
\leg\lambda\tau
=-\leg{-1}{d}\leg1{r_p}
=-(-1)^{(d-1)/2}=(-1)^{(d+1)/2}.
\end{equation}
The formula includes \({r_p}=1\), with the conventional Jacobi symbol of
any integer modulo \(1\) equal to \(1\).

Equations \eqref{rec:nBjacobi}--\eqref{rec:finaljacobi} imply
\begin{equation}\label{rec:nBmod}
n_{\mathfrak b}\equiv2(d+1)\pmod8.
\end{equation}
Let \(\Delta=H_5-4H\). Theorem~\ref{thm:secondbit} gives \(8\mid\Delta\), so
\(3\Delta\equiv\Delta\pmod{16}\). Proposition~\ref{rec:pairing} now gives
\[
\Delta\equiv2n_{\mathfrak b}\equiv4(d+1)\pmod{16},
\]
which proves \eqref{rec:classcongruence}.

\end{proof}

\begin{proof}[Completion of the proof of Theorem~\ref{thm:reciprocity}]
Lemma~\ref{rec:pellshape} gives \eqref{rec:pellfactor}, and the construction
following Lemma~\ref{rec:orientation}, together with
Theorem~\ref{arith:pointind}, gives
$\varrho(p)=(-1)^{(d-1)/2}$. Set $\Delta=h(-5p)-4h(-p)$.
Theorem~\ref{thm:factorial} shows that $8\mid\Delta$ and
\[
 \leg{((p-1)/5)!}{p}=(-1)^{1+\Delta/8}.
\]
Proposition~\ref{rec:bridge} gives $\Delta/8\equiv(d+1)/2\pmod2$.
It follows that
\[
 \leg{((p-1)/5)!}{p}
 =(-1)^{1+(d+1)/2}=(-1)^{(d-1)/2}=\varrho(p).
\]
Combining Theorem~\ref{thm:secondbit} with the higher pairing formula
\eqref{eq:higherpair} yields \eqref{eq:primehigherbridge}.

Finally, every positive solution of $5d_1^2-py_1^2=1$ has $d_1$ odd
and $y_1\equiv2\pmod4$, by the parity argument in
Lemma~\ref{rec:pellshape}. The totally positive element
$5d_1+y_1\sqrt{5p}$ has norm $5$, so it generates the unique prime
ideal above $5$. Any two such generators differ by a power of
$T+U\sqrt{5p}$. Multiplication by this unit sends
\[
 (d_1,y_1)\longmapsto(Td_1+pUy_1,\ 5Ud_1+Ty_1).
\]
Since $T\equiv1\pmod4$ and $4\mid U$, this preserves $d_1\pmod4$,
also for inverse powers. The final assertion follows.
\end{proof}

\begin{corollary}\label{cor:pellrank}
Let $p\equiv11\pmod{40}$ be prime and $5d^2-py^2=1$ with $d,y>0$.
Then
\[
 d\equiv1\pmod4
 \quad\Longleftrightarrow\quad
 \rank A_{5p}(\Q)=0\ \text{ and }\
 \Sh(A_{5p})[2^\infty]\simeq(\Z/4\Z)^2.
\]
Under these equivalent conditions, $10p$ is not $3/5$ congruent.
If $d\equiv3\pmod4$, then $32\mid\D(p)$; this alone does not
determine the rank of $A_{5p}$.
\end{corollary}
\begin{proof}
Theorem~\ref{thm:reciprocity} makes $d\equiv1\pmod4$ equivalent to
$b_p=1$, and Proposition~\ref{prop:nextpair} gives the group theoretic
equivalence. The final divisibility follows from $16\mid\D(p)$ and
\eqref{eq:primehigherbridge}.
\end{proof}

\begin{remark}
The proof uses the genus formula for the composite real radicand $5p$.
The inversion method has a precedent in Chua--Gunby--Park--Yuan
\cite{ChuaGuy}, but their prime radicand theorem is not applied to $5p$.
The comparison is unconditional. The usual R\'edei symbol
$[2,5,p]_R$ is unavailable here: $(2,5)_5=-1$ and
$(2,p)_p=-1$, so its standard pairwise local solubility hypotheses fail.
\end{remark}

The same argument also evaluates each class contribution modulo eight.
Keep $\mathcal H_p=\Cl(\Q(\sqrt{5p}))$ and the ordinary class character $\psi$
and the integers $t_C$ defined above \eqref{rec:twoZ}.

\subsection{Class number consequences}\label{subsec:classconsequences}

\begin{proposition}\label{prop:classwisepell}
For every ordinary ideal class $C$,
\begin{equation}\label{eq:classwisepell}
 t_C\equiv2\bigl(1-\varrho(p)\psi(C)\bigr)\pmod8.
\end{equation}
In particular, $\#\mathcal H_p\equiv2\pmod4$.
\end{proposition}
\begin{proof}
Choose a primitive representative $I=(m+\sqrt N,q)$, where
$N=5p$, $q>0$, $\gcd(q,10p)=1$, and $m\ge0$; put
$\lambda=(m^2-N)/q$ and $t=T+mU>0$. The unit matrix is
\[
 \begin{pmatrix}T+mU&-U\lambda\\qU&T-mU\end{pmatrix}.
\]
Its determinant shows that $T-mU$ is inverse to $t$ modulo $qU$.
Dedekind reciprocity therefore gives the exact identity
\begin{equation}\label{eq:classwiseded}
 n(I)=-\frac{U(q+\lambda)}{t}+12s(qU,t).
\end{equation}
Define $e\in\{0,1\}$ by $(qU/t)=(-1)^e$. By
\eqref{rec:jacded},
\[
 t\,n(I)\equiv-U(q+\lambda)+t-1-4e\pmod8.
\]
Now $q+\lambda\equiv m\pmod2$, $T\equiv1\pmod8$, and
$U\equiv4\pmod8$. The first three terms on the right sum to zero
modulo eight. Thus $n(I)\equiv4e\pmod8$.

The Pell factorization implies
\[
 5t=(5d+ym)^2-y^2q\lambda.
\]
Since $\gcd(t,q)=1$, this gives $(t/q)=(5/q)$, and quadratic
reciprocity gives $(q/t)=\psi(C)$ because $t\equiv1\pmod4$.
Write $y=2{r_p}$, so $U=4d{r_p}$ with $d,{r_p}$ odd. Again by reciprocity,
\[
 \leg Ut=\leg dt\leg{r_p} t
 =\leg td\leg t{r_p}
 =\leg{-1}{d}=(-1)^{(d-1)/2}=\varrho(p).
\]
Here $t\equiv-1\pmod d$ and $t\equiv1\pmod{r_p}$ follow from
$T=10d^2-1=1+2py^2$. All symbols are nonzero; a denominator one
has the usual value one. Hence $(-1)^e=\varrho(p)\psi(C)$.
Multiplication of $n(I)$ by $\chi_0(I)=\pm1$ does not change its
residue, which is zero or four modulo eight. This is
\eqref{eq:classwisepell}.

The nontrivial character $\psi$ has equally large fibers, say of
size $M$. Summing the congruence and using \eqref{rec:twoZ} gives
$3H_5\equiv4M\pmod8$. The class number $H=h(-p)$ is odd, and
\eqref{eq:classshortsum}, whose left side is even, gives
$H_5\equiv4\pmod8$. Thus $M$ is odd and $\#\mathcal H_p=2M\equiv2\pmod4$.
The oddness of $H$ also follows directly by reducing
$-pH=\sum_{j=1}^{p-1}j(j/p)$ modulo two.
\end{proof}

The classwise congruence also gives a second derivation of
\eqref{rec:classcongruence}: the negative genus has odd cardinality and each
of its terms is $2(1+\varrho(p))$ modulo eight. This is compatible with the
Hirzebruch sum normalization of the genus formula in
\cite[Theorem 2.2 in the preprint version, $f=1$]{KimMizuno}.

We record a few examples to illustrate the reciprocity law and the resulting
rank information.

\begin{center}
\begin{tabular}{r|rrrrr}
\toprule
$p$&$h(-p)$&$h(-5p)$&$\D(p)$&$\Xi(p)$&$b_p$\\
\midrule
11&1&4&0&0&0\\
131&5&12&-16&1&1\\
211&3&36&16&1&1\\
251&7&12&0&0&0\\
491&9&28&16&1&1\\
1091&17&36&32&0&0\\
1531&11&76&32&0&0\\
\bottomrule
\end{tabular}
\end{center}

For $p=11$, the curve $A_{55}$ contains the nontorsion point
$(100,1200)$. The corresponding integer $110$ is
$\theta$-congruent: the triangle
\[
 24,\qquad \frac{275}{6},\qquad \frac{221}{6}
\]
has the prescribed angle with cosine $3/5$ and area
$440=4\cdot110$.

For $p=131$, Theorem~\ref{thm:higherpair} gives
$b_{131}=1$, hence
\[
 \rank A_{655}(\Q)=0,\qquad
 \Sh(A_{655})[2^\infty]\simeq(\Z/4\Z)^2.
\]
For $p=491$, one may take
\[
 491=44^2-5\cdot17^2,\qquad
 (u_0,v_0,c_0)=(3,-2,9),
\]
which lies on
\[
 c^2=129u^2+610uv+645v^2.
\]
Since
\[
 (-1)^{(17-1)/2}\leg95=1,
\]
again
\[
 \rank A_{2455}(\Q)=0,\qquad
 \Sh(A_{2455})[2^\infty]\simeq(\Z/4\Z)^2.
\]

For $p=1091$ and $1531$ one has $\D(p)=32\ne0$ but
$b_p=0$. Thus the central value formula gives rank zero and finite
Tate--Shafarevich group, while the two successive degenerate pairings
show
\[
 \Sh(A_{5p})[2^\infty]
 \simeq(\Z/2^{a_p}\Z)^2,
 \qquad a_p\ge3.
\]
The exact exponents require a further descent.

The Pell evaluations below use the norm pairs
$(a,b)=(4,1)$, $(16,5)$, $(16,3)$, $(16,1)$, and $(36,7)$,
respectively. The listed points are those constructed in the proof of
Theorem~\ref{thm:reciprocity}, with $c_0=1$ or $2$.
\begin{center}
\begin{tabular}{rrrrrr}
\toprule
$p$&$d$&$(u_0,v_0,c_0)$&$\varrho(p)$&$b_p$&
$h(-5p)-4h(-p)$\\
\midrule
11&3&$(1,-1,2)$&$-1$&0&0\\
131&8589&$(21,-16,1)$&1&1&$-8$\\
211&13&$(1,-1,2)$&1&1&24\\
251&1764483&$(223,-331,2)$&$-1$&0&$-16$\\
1051&29&$(1,-1,2)$&1&1&72\\
\bottomrule
\end{tabular}
\end{center}
At $p=1051$, the ordinary class group of $\Q(\sqrt{5p})$ has order six.
The contributions $t_C$ in the positive and negative genera are
respectively $\{168,0,0\}$ and $\{84,12,12\}$. The negative genus sum is $108$ and $n_{\mathfrak b}=84$; their
difference $24$ illustrates the cancellation modulo eight in
\eqref{rec:pairingresult}.

As a numerical check, the congruence
$\D(p)/16\equiv(d+1)/2\pmod2$ and the Pell description agree for all
$4916$ primes $p\equiv11\pmod{40}$ below $10^6$. The factorial and class
number identities also agree throughout the corresponding range below
$10^5$. These calculations are not used in the proof.

\section{Governing fields and two prime comparisons}
\label{sec:governing}\label{sec:remaining}

\label{sec:pqgoverning}\label{sec:fourcomplete}

\subsection{Governing fields and joint distributions}\label{subsec:governingjoint}
Fix $p\equiv11,19\pmod{40}$ and let $q$ vary in $\mathcal P_p$. The
residue symbols occurring in the four Cassels matrices are Frobenius
characters. We realize them simultaneously in a tower
$L_p\subset\widetilde L_p\subset\widehat L_p$, determine the degrees of
these fields, and then apply Chebotarev to obtain the joint distribution of
the matrices.

Keep $p$ fixed, and put
\begin{equation}\label{eq:pqfields}
 \begin{split}
 F_0&=\Q(\zeta_{40}),\qquad K=\Q(\sqrt5),\qquad
 \eta=2+\sqrt5,\\
 B_p&=F_0(\sqrt p),\qquad
 L_p=F_0(\sqrt[4]p,\sqrt{\xi_p},\sqrt\eta).
 \end{split}
\end{equation}

\begin{theorem}\label{thm:pqfield}
The extension $L_p/\Q$ is Galois, unramified outside $2,5,p$, and
\[
 [B_p:\Q]=32,\qquad[L_p:\Q]=256.
\]
Its subgroup
$N_p=\operatorname{Gal}(L_p/B_p)\simeq(\Z/2\Z)^3$
is central in $\operatorname{Gal}(L_p/\Q)$. For every prime in
\[
 \mathcal P_p=\{q:q\equiv1\pmod{40},\ (p/q)=1\},
\]
the three signs of its Frobenius on
$\sqrt[4]p,\sqrt{\xi_p},\sqrt\eta$ are respectively
$(-1)^{g_p(q)},(-1)^{\beta_p(q)},(-1)^{t(q)}$.
Each character triple occurs with natural density $1/256$ among all
primes, or relative density $1/8$ in $\mathcal P_p$.
The same uniform distribution holds for $(\alpha,\tau,\gamma)$.
\end{theorem}
\begin{proof}
We begin with the degree. Since $\Norm_{K/\Q}(\eta)=-1$, neither $\eta$
nor $-\eta$ is a square in $K$; in particular $\eta$ is not a square in
$F_0$.
It follows that $K(i,\sqrt\eta)$ has degree eight over $\Q$.
This extension is Galois and nonabelian: writing $w=\sqrt\eta$,
there are automorphisms
\[
 \begin{array}{c|ccc}
 &\sqrt5&i&w\\ \hline
 \sigma&-\sqrt5&i&i/w\\
 \rho&\sqrt5&-i&w
 \end{array}
\]
with $\sigma\rho(w)=i/w$ and $\rho\sigma(w)=-i/w$.
If $w$ belonged to the abelian extension $F_0/\Q$, this
nonabelian Galois extension would be a subfield, a contradiction.

The prime $p$ is unramified in $F_0$ and splits into two primes
in $K$. The principal ideal $(\xi_p)$ is one of these primes,
because $\Norm_{K/\Q}(\xi_p)=p$. At primes of $F_0$ over the
two respective primes of $K$, the valuation parities of
$\xi_p,p,\eta$ have the patterns
\[
 (1,0),\qquad(1,1),\qquad(0,0).
\]
Thus $F_0(\sqrt{\xi_p},\sqrt\eta)/F_0$ is biquadratic,
and none of its three quadratic subfields is $F_0(\sqrt p)$.
Indeed the fields defined by $\xi_p$ and $\xi_p\eta$ have
the first valuation pattern, while the field defined by $\eta$
is unramified at $p$; the squareclass of $p$ has the second pattern.

Since $i\in F_0$ and $v_{\mathfrak p}(p)=1$ at every
$\mathfrak p\mid p$, the polynomial $X^4-p$ is Eisenstein
there. Thus $F_0(\sqrt[4]p)/F_0$ is cyclic of degree
four, with unique quadratic subfield $F_0(\sqrt p)$.
Its intersection with the preceding biquadratic extension must
have degree at most two and cannot be that quadratic subfield.
The extensions are therefore linearly disjoint, and hence
$[L_p:F_0]=16$ and $[L_p:\Q]=256$.
Also $\sqrt p\notin F_0$ by ramification at $p$, so
$[B_p:\Q]=32$.

Normality follows from the behaviour of the generators under conjugation: the nontrivial automorphism of
$K/\Q$ sends $\xi_p$ to $p/\xi_p$ and $\eta$ to
$-1/\eta$. All conjugates of the three radicals in
\eqref{eq:pqfields} are thus already in $L_p$.
The cyclotomic base is unramified outside $2,5$;
$\xi_p$ has support only above $p$, $\eta$ is a unit, and
$X^4-p$ is unramified outside $2,p$. This also proves the
ramification assertion.

Put $u=\sqrt[4]p$, $v=\sqrt{\xi_p}$, $w=\sqrt\eta$.
Every element of $N_p$ acts by independent sign changes on
$u,v,w$. To prove centrality, an arbitrary automorphism of
$L_p/\Q$ sends these generators to elements of the forms
\[
 i^j u,\qquad \pm v\ \text{or}\ \pm u^2/v,
 \qquad\pm w\ \text{or}\ \pm i/w,
\]
respectively. A sign change in $N_p$ fixes $i,u^2,F_0$ and
commutes with each displayed operation. Hence $N_p$ is central.

A prime $q$ lies in $\mathcal P_p$ precisely when it splits completely in
$B_p$. Such primes have density $1/32$.
At a prime of $B_p$ above $q$, Frobenius on $u$ has sign
$(p/q)_4$, on $v$ has sign $\leg{A+Br}q$, and on $w$ has
sign $\leg{2+r}q$. The latter is $(5/q)_4$ by
Lemma~\ref{lem:pqcyclotomic}. The root choices do not change
these signs. Each of the eight elements of the central subgroup
$N_p$ is a single conjugacy class. Chebotarev's theorem
\cite[Theorem 8.31]{MilneANT} therefore gives natural density
$1/256$ for each triple, and relative density $1/8$ after
conditioning on $q\in\mathcal P_p$. Finally
$(\beta,t,g)\mapsto(\alpha,\tau,\gamma)=(\beta+t,t,g)$
is a bijection of $\F_2^3$.
\end{proof}

\begin{corollary}\label{cor:pqdensity}
For each fixed prime $p\equiv11,19\pmod{40}$, the ranks of
the Cassels matrices in \eqref{eq:pqcasselsmatrix} have the
following relative natural densities as $q$ varies in $\mathcal P_p$:
\[
 \begin{array}{c|ccc|c}
 p\bmod40&\rank C_{p,q}=0&\rank C_{p,q}=2&\rank C_{p,q}=4
 &\text{absolute density for rank }4\\ \hline
 11&1/8&5/8&1/4&1/128\\
 19&0&1/2&1/2&1/64
 \end{array}
\]
Here ``absolute density'' uses all rational primes $q$ as denominator
and still imposes $q\in\mathcal P_p$. In either row,
\begin{equation}\label{eq:pqarithmeticdensity}
 \rank C_{p,q}=4
 \quad\Longleftrightarrow\quad
 \rank A_{5pq}(\Q)=0\ \text{and}\
 \Sh(A_{5pq})[2^\infty]\simeq(\Z/2\Z)^4.
\end{equation}
In particular the rank four cases give unconditional families of
non-$\theta$-congruent integers $10pq$ with these exact prime densities.
\end{corollary}
\begin{proof}
For $\eps=0$, the Pfaffian is one precisely when
$\beta=1,t=0$, giving two of the eight triples.
The matrix is zero precisely when $\alpha=\tau=\gamma=0$,
giving one triple. The other five matrices have rank two.
For $\eps=1$ the matrix is never zero, and
$g+\beta(1+t)$ is one on four triples. These counts give the table.

The forward implication of \eqref{eq:pqarithmeticdensity} is
the nondegeneracy argument in Theorem~\ref{thm:pqcassels},
which applies to both values of $\eps$. Conversely, the
stated rank and $\Sh$ condition make the image of $\Sel_4$
in the pure $2$-Selmer group zero: the Mordell--Weil rank is
zero and $2\Sh[4]=0$. This image is the radical of the Cassels
pairing, so its four dimensional matrix is nondegenerate.
\end{proof}

\begin{remark}\label{rem:pqsmallerfield}
For $p\equiv11\pmod{40}$, the Pfaffian alone is governed by
\[
 L_p^0=F_0(\sqrt p,\sqrt{\xi_p},\sqrt\eta),
 \qquad [L_p^0:\Q]=128.
\]
The three quadratic squareclasses are independent by the same
valuation argument. The degree-$256$ field is needed here to
govern the character $g_p(q)$ and hence the full Cassels matrix.
These densities characterize the stated Tate--Shafarevich
groups as well as rank zero. A degenerate matrix can also occur
at rank zero with a larger $2$-primary group.
\end{remark}

Use the fixed norm generator $\xi_p=A+B\sqrt5$ from
\eqref{eq:pqfixednorm}, and write
\begin{equation}\label{eq:threefield}
 \widetilde L_p=F_0(\sqrt[4]p,\sqrt{\xi_p},
            \sqrt{2+\sqrt5},\sqrt[4]2),
 \qquad F_0=\Q(\zeta_{40}).
\end{equation}

\begin{theorem}\label{thm:threefield}
The field $\widetilde L_p$ is Galois over $\Q$, unramified
outside $2,5,p$, and has degree $512$. Its subgroup over
$B_p=F_0(\sqrt p)$ is a central elementary abelian group of
order $16$. As $q$ varies in $\mathcal P_p$, the four characters
\[
 (\beta_p(q),\tau,\gamma,\upsilon)
\]
are jointly uniform on $\F_2^4$.
\end{theorem}
\begin{proof}
The only new point, beyond Theorem~\ref{thm:pqfield}, is the independence of
$\sqrt[4]2$. Let $H=\Q(i,\sqrt[4]2)$, the degree eight dihedral
splitting field of $X^4-2$. It contains $\Q(\zeta_8)$ and
is unramified outside $2$. Hence $H\cap\Q(\zeta_5)=\Q$,
since a nontrivial subfield of the cyclic quartic field
$\Q(\zeta_5)$ contains its quadratic field $\Q(\sqrt5)$,
which is ramified at $5$. Consequently
\[
 F_1:=F_0(\sqrt[4]2)=H\Q(\zeta_5),\qquad[F_1:\Q]=32,
 \quad\operatorname{Gal}(F_1/\Q)\simeq D_4\times C_4.
\]
We claim that $\eta=2+\sqrt5$ is not a square in $F_1$.
There is an automorphism $\sigma$ fixing $H$ and sending
$\zeta_5$ to $\zeta_5^2$; it sends $\sqrt5$ to $-\sqrt5$
and fixes $i$. It commutes with complex conjugation $\rho$.
If the positive real number $w=\sqrt\eta$ lay in $F_1$,
then $\rho(w)=w$ whereas $\sigma(w)=\pm i/w$.
These two actions would not commute, a contradiction.

The field $F_1$ is unramified at $p$. The valuations of
$\xi_p,p,\eta$ above the two primes of $\Q(\sqrt5)$ over
$p$ still have the patterns $(1,0),(1,1),(0,0)$.
Exactly the disjointness argument in Theorem~\ref{thm:pqfield},
now over $F_1$, gives
\[
 [F_1(\sqrt[4]p,\sqrt{\xi_p},\sqrt\eta):F_1]=16.
\]
Thus $[\widetilde L_p:\Q]=512$. Normality and ramification
follow as before, together with the splitting field of $X^4-2$.

The kernel over $B_p$ consists of the sixteen independent sign
changes on $\sqrt[4]p,\sqrt{\xi_p},\sqrt\eta,\sqrt[4]2$.
The first three commute with every automorphism by the proof
of Theorem~\ref{thm:pqfield}; the fourth does too, since every
conjugate of $\sqrt[4]2$ is $i^j\sqrt[4]2$ and the kernel
fixes $i$ and $\sqrt2$. Hence this kernel is central.
For $q\in\mathcal P_p$, the additional Frobenius sign is
$(-1)^{\upsilon}=(2/q)_4$. Chebotarev gives absolute density
$1/512$, hence relative density $1/16$ in $\mathcal P_p$,
for each value in $\F_2^4$.
\end{proof}

\begin{corollary}\label{cor:threejoint}
For each fixed $p$, the relative natural density in $\mathcal P_p$
of a prescribed triple $(P_+,P_-,P_2)=(a,b,c)$ is given by
the following table. Each entry holds for each of $c=0$ and $c=1$
separately:
\[
 \begin{array}{c|cccc}
 p\bmod40&(a,b)=(0,0)&(0,1)&(1,0)&(1,1)\\ \hline
 11&3/16&3/16&1/16&1/16\\
 19&3/16&1/16&3/16&1/16.
 \end{array}
\]
In particular all three pairings are nondegenerate with relative
density $1/16$, or absolute prime density $1/512$. On this set,
\[
 \begin{gathered}
 \rank E_1(n)(\Q)=\rank E_{-1}(n)(\Q)
            =\rank E_2(n)(\Q)=0,\\
 \Sh(E_1(n))[2^\infty]\simeq\Sh(E_{-1}(n))[2^\infty]
            \simeq(\Z/2\Z)^4,\\
 \Sh(E_2(n))[2^\infty]\simeq(\Z/2\Z)^2,\qquad
 \rank E_{-2}(n)(\Q)\le1.
 \end{gathered}
\]
\end{corollary}
\begin{proof}
Put $b_0=\beta_p(q)(1+\tau)=\alpha(1+\tau)$.
It equals one on one quarter of the independent $(\beta,\tau)$
pairs. The two signs $\gamma,\upsilon$ are independent of it.
Now
\[
 (P_+,P_-,P_2)
 =(b_0+\eps\gamma,
  b_0+(1+\eps)\gamma,\gamma+\upsilon).
\]
The third coordinate is uniform and independent of the first
two. When $\gamma=0$ those first coordinates agree, and when
$\gamma=1$ they differ. This yields the table.
Nondegeneracy and the pure dimensions from
Corollary~\ref{cor:fourdimensions} give the rank and $\Sh$
statements by the same Cassels pairing argument as before.
The last rank bound is Theorem~\ref{thm:threecassels}.
\end{proof}

\begin{remark}\label{rem:fourodd}
The pure Selmer group of $E_{-2}(pq)$ has dimension three, so its
ordinary pairing is necessarily degenerate and has no full space
Pfaffian. The field $\widetilde L_p$ governs its $U,V$ entry;
the complete matrix also involves the character $\rho$.
If this matrix has rank two and $\Sh(E_{-2}(pq))[2^\infty]$ is
finite, then
\[
 \rank E_{-2}(pq)(\Q)=1,\qquad
 \Sh(E_{-2}(pq))[2^\infty]\simeq(\Z/2\Z)^2.
\]
Indeed the rank is at most one, while a finite $2$-primary
Tate--Shafarevich group has even $2$-rank. The pure Selmer
dimension forces rank one; the nondegenerate quotient then
forces the $2$-primary group to be killed by $2$.
\end{remark}

Define
\begin{equation}\label{eq:fourcompletefield}
 \widehat L_p=\widetilde L_p(\sqrt{\xi_2})
 =F_0(\sqrt[4]p,\sqrt{\xi_5},\sqrt{2+\sqrt5},
       \sqrt[4]2,\sqrt{\xi_2}).
\end{equation}

\begin{theorem}\label{thm:fourcompletefield}
The extension $\widehat L_p/\Q$ is Galois, unramified outside
$2,5,p$, and has degree $1024$. Its subgroup over $B_p$ is
central and isomorphic to $(\Z/2\Z)^5$.
As $q$ varies over $\mathcal P_p$, the five characters
\begin{equation}\label{eq:fourcompletebits}
 (\beta,\tau,\gamma,\upsilon,\rho)
\end{equation}
are jointly uniform: every value has relative natural density
$1/32$, or absolute prime density $1/1024$.
They determine all four Cassels matrices, and their $32$
values give $32$ distinct ordered quadruples of matrices.
\end{theorem}
\begin{proof}
Theorem~\ref{thm:threefield} gives
$[\widetilde L_p:F_0]=32$. Over $F_0$ its Galois group is
$\Z/4\Z\times(\Z/2\Z)^3$, with quadratic subfields
corresponding to the squareclass span
\[
 p,\quad \xi_5,\quad 2+\sqrt5,\quad \sqrt2
 \quad\text{in }F_0^*/F_0^{*2}.
\]
We show that $\xi_2$ is outside this span. The prime $p$
splits completely in $F=\Q(\sqrt5,\sqrt{-2})$ and is
unramified in $F_0$. Group the primes of $F_0$ above $p$
according to the four primes of $F$ below them. After labelling
these four groups suitably, the valuation parities are
\[
 \begin{array}{c|cccc}
 &1&2&3&4\\ \hline
 p&1&1&1&1\\
 \xi_5&1&1&0&0\\
 \xi_2&1&0&1&0.
 \end{array}
\]
Indeed each norm generator is divisible by exactly one of the
two primes in its quadratic field, with valuation one. The
other two proposed generators are units at $p$.
The last row is not a linear combination of the first two.
Thus $\sqrt{\xi_2}\notin\widetilde L_p$ and
$[\widehat L_p:\Q]=2\cdot512=1024$.

The conjugate of $\xi_2$ is $p/\xi_2$, and $\sqrt p$ belongs
to $B_p$. Hence all conjugates of $\sqrt{\xi_2}$ lie in
$\widehat L_p$, proving normality in conjunction with
Theorem~\ref{thm:threefield}. The norm and integrality of
$\xi_2$ also give the claimed ramification bound.
Over $B_p$, the five displayed radicals can have their signs
changed independently. This kernel has order $32$ and exponent
two. It is central on the old generators by the preceding
theorem. On the new generator every rational Galois automorphism
acts by one of
\[
 \sqrt{\xi_2}\longmapsto\pm\sqrt{\xi_2},\qquad
 \sqrt{\xi_2}\longmapsto\pm\frac{\sqrt p}{\sqrt{\xi_2}}.
\]
These transformations commute with every kernel element because
the latter fixes $\sqrt p$. Hence the subgroup is central.

The primes in $\mathcal P_p$ are exactly those splitting
completely in $B_p$, apart from the excluded ramified primes.
The five signs of their Frobenius elements are given by the characters
\eqref{eq:fourcompletebits}, in the order prescribed by the
corresponding radicals. Centrality makes each element a singleton
conjugacy class, so Chebotarev gives absolute density $1/1024$.
Since $[B_p:\Q]=32$, the relative density is $1/32$.

Finally $C_+$ determines $\alpha,\tau,\gamma$, hence
$\beta=\alpha+\tau$. The $U,V$ entry of $C_2$ then recovers
$\upsilon$, and the $W,V$ entry of $C_{-2}$ recovers $\rho$.
The map from the five characters to the ordered quadruple of matrices
is therefore injective. The additional character $\rho$ is therefore necessary to
determine the fourth matrix.
\end{proof}

Put $R_+=\rank C_+$, $R_-=\rank C_-$,
$R_2=\rank C_2$, and $R_{-2}=\rank C_{-2}$.

\begin{corollary}\label{cor:fourcompleteranks}
For $p\equiv11\pmod{40}$ the joint relative natural densities
are given by the following table, whose entries are to be divided
by $32$:
\begin{equation}\label{eq:fourranktable}
 \begin{array}{c|ccc}
 (R_+,R_-)&(R_2,R_{-2})=(0,0)&(0,2)&(2,2)\\ \hline
 (0,2)&1&1&2\\
 (2,2)&1&3&4\\
 (2,4)&2&4&6\\
 (4,2)&0&2&2\\
 (4,4)&0&2&2
 \end{array}
\end{equation}
For $p\equiv19\pmod{40}$ interchange the two coordinates
of each row label. All unlisted rank patterns have density zero.
In particular $R_{-2}=0$ and $R_{-2}=2$ have respective
relative densities $1/8$ and $7/8$.
\end{corollary}
\begin{proof}
For $\eps=0$, the triples $(\beta,\tau,\gamma)$ in the
order $000,001,010,011,100,101,110,111$ give rank pairs
\[
 (0,2),(2,4),(2,2),(2,4),(4,4),(4,2),(2,2),(2,4).
\]
This follows from the two Pfaffians and from the fact that
$C_+=0$ precisely at $000$, whereas $C_-$ has a constant
nonzero entry. For each triple the two remaining characters are free.
If $\beta=0$, their four choices give last rank counts
$(1,1,2)$ in the three displayed columns, by
\eqref{eq:fourthzerolocus}. If $\beta=1$, these counts are
$(0,2,2)$. Summing gives \eqref{eq:fourranktable}.
For $\eps=1$ the two Pfaffians exchange roles, and the
possible zero matrix is instead $C_-$ at $000$, so the first
two ranks exchange roles. The last assertion also follows
directly because
$(\beta,\rho+\beta+\gamma,\gamma+\upsilon)$ is a surjective
linear image onto $\F_2^3$ of the uniform vector in $\F_2^5$.
\end{proof}

\begin{corollary}\label{cor:fourcompleteselmer}
For $E_{-2}=A_{-10pq}$ the exact relative distribution of its
$4$-Selmer group is
\[
 \begin{array}{c|c|c}
 \text{condition}&\Sel_4(E_{-2})&\text{density}\\ \hline
 \beta=0,\ \rho=\gamma=\upsilon
 &(\Z/4\Z)^3\oplus(\Z/2\Z)^2&1/8\\
 \text{otherwise}
 &\Z/4\Z\oplus(\Z/2\Z)^4&7/8.
 \end{array}
\]
On the second set, $\rank E_{-2}(\Q)\le1$, and its pure
radical is exactly the line \eqref{eq:fourthradical}.
The joint distribution of all four $4$-Selmer groups is
obtained from \eqref{eq:fourranktable} by substituting
\begin{equation}\label{eq:fourjointselmer}
 \Sel_4(E_\delta)\simeq
 (\Z/4\Z)^{s_\delta-R_\delta}
 \oplus(\Z/2\Z)^{R_\delta+2},
 \quad(s_1,s_{-1},s_2,s_{-2})=(4,4,2,3).
\end{equation}
These conclusions require no conjectural input. If
$\Sh(E_{-2})[2^\infty]$ is finite and $R_{-2}=2$, then
\[
 \rank E_{-2}(\Q)=1,\qquad
 \Sh(E_{-2})[2^\infty]\simeq(\Z/2\Z)^2.
\]
\end{corollary}
\begin{proof}
The radical dimensions are three and one, respectively.
Apply Proposition~\ref{prop:sel4structure}; doing so for all
four twists also gives \eqref{eq:fourjointselmer}.
For the conditional assertion use the alternating perfect
pairing on the finite $2$-primary Tate--Shafarevich group,
as in Remark~\ref{rem:fourodd}.
\end{proof}

The density $7/8$ refers to the displayed Selmer group condition.
The Mordell--Weil rank can also be at most one when $R_{-2}=0$.

For $E_1(pq)$, write vectors in the basis
$(\Lambda,\Lambda',U,V)$. The kernels below are the images of $4$-Selmer groups in the
pure $2$-Selmer groups.
\subsection{Ordinary radicals for two prime twists}\label{subsec:pqradicals}

\begin{proposition}\label{prop:pqradicals}
If $\eps=1$ and $\Pf(C_+)=0$, the radical has basis
\[
 U+(\alpha+\tau)\Lambda,\qquad
 V+\tau\Lambda+\alpha\Lambda'.
\]
If $\eps=0$, its degenerate cases are
\[
 \begin{array}{c|l}
 (\alpha,\tau,\gamma)&\text{radical basis}\\ \hline
 (0,0,0)&\Lambda,\Lambda',U,V\\
 (0,0,1)&\Lambda,\Lambda'\\
 (0,1,\gamma)&\Lambda,\ \gamma\Lambda'+U+V\\
 (1,1,\gamma)&\Lambda+\Lambda',\ \gamma\Lambda+U.
 \end{array}
\]
For either $C_+$ or $C_-$, let $d$ denote the radical dimension.
The $4$-Selmer groups are
\[
 \begin{array}{c|c|c}
 \rank C&d&\Sel_4\\ \hline
 4&0&(\Z/2\Z)^6\\
 2&2&(\Z/4\Z)^2\oplus(\Z/2\Z)^4\\
 0&4&(\Z/4\Z)^4\oplus(\Z/2\Z)^2.
 \end{array}
\]
Whenever $P_2=1$, one also has
\[
 \Sel_4(E_2(pq))\simeq(\Z/2\Z)^4,\qquad
 \Sel_4(E_{-2}(pq))\simeq\Z/4\Z\oplus(\Z/2\Z)^4.
\]
\end{proposition}
\begin{proof}
Substitution in \eqref{eq:pqcasselsmatrix} shows that each
listed vector is in the kernel. They are independent and their
number is $4-\rank C_+$, which proves completeness. The same
dimension formula and Proposition~\ref{prop:sel4structure}
give the table. If $P_2=1$, the pure dimensions of $E_2,E_{-2}$
are two and three, and their matrix ranks are both two, giving
radical dimensions zero and one, respectively.
\end{proof}

\begin{example}\label{ex:pqfour}
Take $(p,q)=(11,401)$ and $(a,b)=(84,23)$, so
$84^2-5\cdot23^2=4411$. The symbols in
\eqref{eq:pqbits} give $(\alpha,\tau)=(1,0)$.
For the final entry one may use
\[
 (R,S,T)=(67,3,4),\qquad w=104\pmod{401},\qquad\gamma=1.
\]
Thus
\[
 C_{11,401}=
 \begin{pmatrix}
 0&0&0&1\\0&0&1&0\\0&1&0&1\\1&0&1&0
 \end{pmatrix},\qquad\Pf(C_{11,401})=1.
\]
In particular $44110$ is not $\theta$-congruent, and
$\Sh(A_{22055})[2^\infty]\simeq(\Z/2\Z)^4$.
The pair $(p,q)=(11,521)$, with $(a,b)=(76,3)$, gives a second
example: $57310$ is not $\theta$-congruent.
\end{example}

\begin{example}\label{ex:threesimultaneous}
Take $p=11$, $q=3001$, and $\xi_p=4+\sqrt5$. The root
$r=245$ of $5$ modulo $3001$ gives
\[
 (\beta,\tau,\gamma,\upsilon)=(1,0,0,1).
\]
Thus $(P_+,P_-,P_2)=(1,1,1)$. For $n=33011$, the three
curves $A_{165055},A_{-165055},A_{330110}$ all have rank zero,
with $2$-primary Tate--Shafarevich groups of respective orders
$16,16,4$, all killed by $2$. The fourth curve $A_{-330110}$
has rank at most one.
\end{example}

\begin{example}\label{ex:fourthcomplete}
For $p=11$ take $\xi_2=3+\sqrt{-2}$.
At $q=401$ the five characters are $(1,0,1,1,0)$. Here $C_2=0$,
but $\langle W,U\rangle=1$ and the fourth radical is
$\langle V\rangle$. Thus $A_{-44110}$ has rank at most one,
a conclusion not supplied by its $U,V$ submatrix.
For the simultaneous example $q=3001$ in
Example~\ref{ex:threesimultaneous}, one has $\rho=1$.
Its fourth matrix has upper triangular entries $(1,0,1)$,
and the radical is $\langle W+V\rangle$.
\end{example}

Write $\tau_0(M)$ for the number of positive divisors of $M$, and
$\chi_5(d)=\leg d5$. For a positive integer $M\equiv3\pmod8$ with
$M\equiv\pm1\pmod5$, put
\begin{equation}\label{eq:pqnu}
 \nu_5(M)=\frac{\tau_0(M)-\sum_{d\mid M}\chi_5(d)}4\pmod2,
 \qquad
 \Psi(n)=\sum_{\substack{z\ge1\\40z^2<n}}\nu_5(n-40z^2).
\end{equation}
The sums defining $\Psi$ take values in $\F_2$.

\subsection{The two prime representation bridge}\label{subsec:pqbridge}

\begin{proposition}\label{prop:pqdefectbit}
For the two prime family \eqref{eq:pqfamily},
\[
 4\mid h(-n),\qquad16\mid h(-5n),\qquad32\mid\D(n).
\]
The quantity in \eqref{eq:pqnu} is well defined, and
\begin{equation}\label{eq:pqdefectbit}
 \frac{\D(n)}{32}
 \equiv\frac{4h(-n)-h(-5n)}{16}+\Psi(n)\pmod2.
\end{equation}
\end{proposition}
\begin{proof}
The prime discriminants of $-n$ are $-p,q$, and those of $-5n$
are $-p,5,q$. All off diagonal symbols in their R\'edei matrices
are trivial: $(5/p)=(5/q)=(p/q)=(q/p)=1$ and
$(-p/5)=(-p/q)=1$. Their diagonal entries, the sums of the other
entries in each column, are also zero. Genus theory and the R\'edei
rank formula \cite[Theorem 3.1]{Stevenhagen} therefore give $2$-rank
and $4$-rank both equal to one for $\Cl(-n)$, and both equal to
two for $\Cl(-5n)$. Every nontrivial cyclic factor of these $2$-class groups therefore has order
at least four, which gives the two asserted class number divisibilities.

Both $p$ and $q$ split in $\Q(\sqrt{-2})$ and
$\Q(\sqrt{-10})$, so the divisor sum formulas give
$B_2(n)=B_{10}(n)=8$. Theorem~\ref{thm:composite} now proves
$32\mid\D(n)$.

For an integer $M$ occurring in \eqref{eq:pqnu}, one has
$\chi_5(M)=\chi_{-8}(M)=1$, and $M$ is not a square. The
divisors with $\chi_5(d)=-1$ consequently form pairs $\{d,M/d\}$.
If their number is $N_-(M)$, then
$\tau_0(M)-\sum_{d\mid M}\chi_5(d)=2N_-(M)$ is divisible by four.
Moreover, Lemma~\ref{lem:binary} gives
\[
 B_2(M)-B_{10}(M)
 =4\sum_{\substack{d\mid M\\\chi_5(d)=-1}}\chi_{-8}(d).
\]
The two values of $\chi_{-8}$ in each divisor pair agree. Hence
\[
 \frac{B_2(M)-B_{10}(M)}8
 \equiv\frac{N_-(M)}2=\nu_5(M)\pmod2.
\]
The exact splitting by the third coordinate in the proof of
Proposition~\ref{prop:FT}, together with
$r_T(n)=4h(-n)$ and \eqref{eq:FGsum}, yields
\[
 \D(n)=8h(-n)-2h(-5n)
 +4\sum_{\substack{z\ge1\\40z^2<n}}
 \bigl(B_2(n-40z^2)-B_{10}(n-40z^2)\bigr).
\]
Reduce this identity modulo $64$ and divide by $32$.
\end{proof}

The representation formula and the Cassels matrix now produce two natural
bits. The following conjecture asserts that they agree.
\begin{conjecture}
\label{conj:pqbridge}\label{prob:pqbridge}
Let $p,q$ be primes satisfying
\[
 p\equiv11,19\pmod{40},\qquad q\equiv1\pmod{40},\qquad
 \leg pq=1,
\]
and put $n=pq$. Let $C_{p,q}$ be the ordinary Cassels matrix of
$A_{5pq}$ in the ordered basis $(\Lambda,\Lambda',U,V)$ of
Theorem~\ref{thm:pqcassels}. Then
\begin{equation}\label{eq:pqconjecturebridge}
 \frac{\D(pq)}{32}\equiv\Pf(C_{p,q})\pmod2.
\end{equation}
Equivalently, with $\Psi$ as in \eqref{eq:pqnu} and the fixed
residue characters of \eqref{eq:pqfixedbits},
\begin{equation}\label{eq:pqbridge}
 \frac{4h(-pq)-h(-5pq)}{16}+\Psi(pq)
 =\eps g_p(q)+\beta_p(q)(1+t(q))
 \quad\text{in }\F_2.
\end{equation}
Here $\eps=0$ for $p\equiv11\pmod{40}$ and $\eps=1$ for
$p\equiv19\pmod{40}$.
\end{conjecture}
Proposition~\ref{prop:pqdefectbit} makes the normalized representation bit
well defined, and Proposition~\ref{prop:pqfixedsymbols} gives the Pfaffian in
fixed residue characters. What remains conjectural is their equality.
There is substantial finite evidence for the conjecture. The fixed character
calculation, including the auxiliary conics, agrees for all $122$ admissible
pairs with $p<300$ and $pq\le200000$. Direct signed counts of $F$ and $G$
extend the comparison to all $4469$ admissible pairs with
$pq\le5\cdot10^6$: $2383$ have $p\equiv11\pmod{40}$ and $2086$ have
$p\equiv19\pmod{40}$. In this range both $32\mid\D(pq)$ and
\eqref{eq:pqconjecturebridge} hold in every case. These computations are
only evidence. For instance, $\D(11\cdot401)=-32$ whereas
$\D(11\cdot3041)=64$; the normalized residues are one and zero, and so are
the corresponding Pfaffians. The second example also shows that nonvanishing
of the representation defect is not the same as nondegeneracy of the ordinary
Cassels pairing.

Before returning to descent, it is useful to understand the correction term
$\Psi$ on its own. Its support admits a simple factorization description,
independent of Conjecture~\ref{conj:pqbridge}.
For $M\equiv3\pmod8$ with $M\bmod5\in\{1,4\}$, recall
\[
 \nu_5(M)=\frac{\tau_0(M)-\sum_{d\mid M}(d/5)}4\pmod2.
\]

\begin{proposition}\label{prop:sparse}
The value $\nu_5(M)$ equals one if and only if exactly one of the following
two disjoint descriptions holds:

\smallskip\noindent
\textup{(I)} $M=\ell_1\ell_2w^2$, where $\ell_1\ne\ell_2$ are primes
with $(\ell_1/5)=(\ell_2/5)=-1$, and both $v_{\ell_1}(w)$ and
$v_{\ell_2}(w)$ are even.

\smallskip\noindent
\textup{(II)} $M=\ell w^2$, where $\ell$ is prime with $(\ell/5)=1$,
$v_\ell(w)$ is even, and $(w/5)=-1$.

\smallskip\noindent
Here $w>0$; in either description it is determined by $M$ and its
squarefree kernel. No squarefree hypothesis on $M$ is imposed.
\end{proposition}
\begin{proof}
Write $M=\prod\ell^{e_\ell}$, and separate the primes according to
$\chi_5(\ell)=(\ell/5)$.
The two multiplicative expressions are
\[
 \tau_0(M)=\prod_\ell(e_\ell+1),\qquad
 \sum_{d\mid M}\chi_5(d)=
 \prod_{\chi_5(\ell)=1}(e_\ell+1)
 \prod_{\chi_5(\ell)=-1}\left(\sum_{j=0}^{e_\ell}(-1)^j\right).
\]
Suppose a nonsplit prime has odd exponent. Since $\chi_5(M)=1$, at least
two such primes do. The character sum is zero, and $\nu_5(M)=1$ precisely
when $v_2(\tau_0(M))=2$. This forces exactly two odd exponents, both at
nonsplit primes and both $1$ modulo $4$, giving (I).

Otherwise all nonsplit exponents are even. Put
$A=\prod_{\chi_5(\ell)=1}(e_\ell+1)$ and
$B=\prod_{\chi_5(\ell)=-1}(e_\ell+1)$.
Then $B$ is odd and $\nu_5(M)=A(B-1)/4\pmod2$.
The number $M\equiv3\pmod8$ is not a square, so $A$ is even.
The residue is one exactly when $v_2(A)=1$ and $B\equiv3\pmod4$.
The first condition gives one odd exponent at a split prime, congruent
to $1$ modulo $4$. Writing $M=\ell w^2$, the second is
$B\equiv\prod_{\chi_5(t)=-1}(-1)^{v_t(w)}=(w/5)=-1\pmod4$.
This is exactly (II).
\end{proof}

Thus $\Psi(n)=\sum_{40z^2<n,\ z>0}\nu_5(n-40z^2)$ is the parity of the
shifts having one of the two factorization patterns above. This description
is useful, but it does not express that parity in terms of the ordinary
residue characters and does not produce the positive definite quaternary
model asked for in Problem~\ref{prob:quaternary}. Consequently the identity
\begin{equation}\label{eq:stillopen}
 \frac{\D(pq)}{32}\equiv\eps\gamma+\beta(1+\tau)\pmod2
\end{equation}
remains open. The higher pairing calculation lies one level deeper in the
descent and does not imply this representation congruence.

\begin{remark}\label{rem:higherbitexperiment}
On the two prime locus of Corollary~\ref{cor:twohigher}, exact computations
suggest the stronger comparison
\[
 64\mid\D(pq),\qquad \frac{\D(pq)}{64}\equiv e_{pq}\pmod2.
\]
This statement is not used in any proof and remains open. The expanded
exact experiment in the finite checks described in the text verifies the displayed
congruence for all $227$ parameters on this two dimensional locus with
$pq\le5\cdot10^6$, but this is evidence only. The proved result is the
conic evaluation of the higher Cassels--Tate pairing in
Theorem~\ref{thm:compositehigher}; no additional divisibility of
$\D(pq)$ follows formally from that pairing calculation.
\end{remark}

\section{Composite higher descent and Selmer towers}
\label{sec:extensions}

\subsection{Ordinary Cassels matrices with several auxiliary primes}\label{subsec:multimatrices}

We now pass from one auxiliary prime to several. The ordinary Cassels matrices
still admit uniform descriptions in terms of norm conics. On a natural zero
character locus, the $4$-cover constructed in the prime case remains valid
and can be used to compute the next pairing. This gives higher pairings on
larger radicals; when the ordinary radical is one dimensional, a separate
group theoretic argument determines the entire finite $2$-power Selmer tower.

Throughout this section, write
\begin{equation}\label{eq:family}
\begin{gathered}
 n=pq_1\cdots q_r,\qquad p\equiv11,19\pmod{40},\\
 q_i\equiv1\pmod{40},\qquad
 \leg p{q_i}=\leg{q_i}{q_j}=1\quad(i\ne j).
\end{gathered}
\end{equation}
where all primes are distinct and $r\ge1$. Put
$E_\delta=A_{5\delta n}$ for $\delta\in\{1,-1,2,-2\}$. We use the first
two Kummer coordinates and set
\[
 \Lambda=(5,1),\qquad \Lambda'=(1,5),\qquad
 U_i=(q_i,1),\qquad V_i=(1,q_i),\qquad W=(p,1).
\]
After relabelling the auxiliary primes in
Corollary~\ref{cor:fourdimensions}, the pure $2$-Selmer dimensions and
bases are
\[
\begin{array}{c|c|l}
\delta&\dim\Sel_2^0(E_\delta)&\text{basis}\\ \hline
1,-1&2r+2&\Lambda,\Lambda',U_1,\ldots,U_r,V_1,\ldots,V_r\\
2&2r&U_1,\ldots,U_r,V_1,\ldots,V_r\\
-2&2r+1&W,U_1,\ldots,U_r,V_1,\ldots,V_r.
\end{array}
\]
All matrices in this section are over $\F_2$. We begin with the
ordinary Cassels matrices. Fix $\delta$, and put $h_i=5\delta n/q_i$.
Choose primitive integral
solutions of the three equations
\begin{align}
 R_i^2-q_iS_i^2&=h_iT_i^2,\label{eq:R}\\
 K_i^2-q_iL_i^2&=5h_iM_i^2,\label{eq:K}\\
 J_i^2-q_iN_i^2&=-h_iP_i^2.\label{eq:J}
\end{align}
The third coordinates can be chosen nonzero. Each of the three equations is a
conic projection of a $U_i$- or $V_i$-covering, hence is everywhere locally
soluble. By the Hasse principle it has a rational point, and clearing
denominators gives the primitive integral triples used below.

\begin{lemma}\label{lem:local}
Let $\ell\equiv1\pmod8$ be prime, and let $a,b\in\Z$ satisfy
$v_\ell(b)=1$ and $(a/\ell)=1$. Write $b=\ell b_0$ with
$b_0\in\Z_\ell^\times$, and assume $(b_0/\ell)=1$.
Suppose $X^2-aY^2=bZ^2$ is a primitive integral solution.
For either $s\in\Q_\ell$ with $s^2=a$ and $X-Ys\ne0$,
\[
 (X-Ys,\ell)_\ell=\leg X\ell.
\]
Here $X$ and $Y$ are units at $\ell$.
\end{lemma}
\begin{proof}
If $\ell\mid X$, then $\ell\mid Y$. Primitivity forces $Z$ to be a unit,
and the two sides of the norm equation have valuations at least two and
one, a contradiction. Hence $X,Y$ are units. If $X-Ys$ is a unit its
reduction is $2X$, proving the result. Otherwise $X+Ys$ is a unit, and
\[
 X-Ys=\frac{bZ^2}{X+Ys}.
\]
Its valuation is $1+2v_\ell(Z)$, and its unit part has character
$(b_0/\ell)(2X/\ell)=(X/\ell)$. For $\ell\equiv1\pmod8$,
$(-1/\ell)=(2/\ell)=1$, so this is also its Hilbert symbol with $\ell$.
\end{proof}

Define $r\times r$ matrices $A_\delta,B_\delta$ by
\begin{align}
 (A_\delta)_{ii}&=0,&
 (A_\delta)_{ij}&=[R_iK_i/q_j]&& (i\ne j),\label{eq:AB1}\\
 (B_\delta)_{ii}&=[\delta n/q_i\,/q_i]_4,&
 (B_\delta)_{ij}&=[R_i/q_j]&& (i\ne j).\label{eq:AB2}
\end{align}
The numerator in the quartic symbol in \eqref{eq:AB2} is the integer
$\delta n/q_i$. It is a quadratic residue modulo $q_i$ by \eqref{eq:family}.
All quadratic symbols in \eqref{eq:AB1}--\eqref{eq:AB2} are nonzero,
as follows from the valuation argument in Lemma~\ref{lem:local}.

\begin{theorem}\label{thm:UV}
The matrix $A_\delta$ is alternating, and the ordinary Cassels matrix on
$\langle U_1,\ldots,U_r,V_1,\ldots,V_r\rangle$ is
\begin{equation}\label{eq:H}
 H_\delta=\begin{pmatrix}
 A_\delta&B_\delta\\ B_\delta^t&B_\delta+B_\delta^t
 \end{pmatrix}.
\end{equation}
In particular the entries are independent of all choices in
\eqref{eq:R}--\eqref{eq:J}. For $i\ne j$ one has the reciprocity identities
\begin{equation}\label{eq:normrec}
 [J_i/q_j]=[R_j/q_i],\qquad
 [R_iK_i/q_j]=[R_jK_j/q_i].
\end{equation}
\end{theorem}
\begin{proof}
In covering coordinates $(t,u_1,u_2,u_3)$, tangent forms for $U_i$ can be
chosen as
\begin{align*}
 L_1&=K_i u_3-L_i u_2-5h_iM_it,\\
 L_2&=(h_i+1)u_3-(h_i-1)u_1-4h_it,\\
 L_3&=R_i u_1-S_i u_2-h_iT_it.
\end{align*}
Their conic points are respectively
$(t,u_2,u_3)=(M_i,q_iL_i,K_i)$,
$(t,u_1,u_3)=(1,h_i-1,h_i+1)$, and
$(t,u_1,u_2)=(T_i,R_i,q_iS_i)$.
For $V_i$ use
\begin{align*}
 L'_1&=(5h_i-1)u_2-(5h_i+1)u_3+10h_it,\\
 L'_2&=R_i u_3-S_i u_1-2h_iT_it,\\
 L'_3&=J_i u_2-N_i u_1+h_iP_it.
\end{align*}
The corresponding points are
$(t,u_2,u_3)=(1,(5h_i-1)/2,(5h_i+1)/2)$,
$(t,u_1,u_3)=(T_i,2q_iS_i,2R_i)$, and
$(t,u_1,u_2)=(P_i,q_iN_i,J_i)$.
These are verified directly from \eqref{eq:R}--\eqref{eq:J}.

For second argument $U_j$ or $V_j$, the Hilbert parameter is $q_j$.
It is a square at $2,5,p$, and at every $q_i\ne q_j$, and is positive
at infinity. At any other prime the conics have good reduction and the
chosen tangent sections are primitive. Their divisors are twice a point,
so their nonzero values have even valuation. To justify primitivity away
from $2,5,n$, note that a prime dividing both first coordinates of one
of \eqref{eq:R}--\eqref{eq:J} must divide the third coordinate as well;
the possible square factor at $5$ in \eqref{eq:K} is irrelevant here.
Thus only the place $q_j$ contributes.

Suppose $i\ne j$. Choose $s^2=q_i$ in $\Q_{q_j}$.
At $q_j$ use the exact point $(0,1,s,1)$ on the $U_i$ covering.
The products for $U_j,V_j$ are
\[
 (K_i-L_is)(R_i-S_is),\qquad 2(R_i-S_is).
\]
Every prime factor of $h_i/q_j$ is a square modulo $q_j$;
this includes the factors $-1$, $2$, and $5$ when they occur.
Lemma~\ref{lem:local}, applied also to $5h_i$, gives
\[
 \langle U_i,U_j\rangle=[R_iK_i/q_j],\qquad
 \langle U_i,V_j\rangle=[R_i/q_j].
\]
A single choice of the square root need not make both tangent values units.
This causes no difficulty, because Lemma~\ref{lem:local} also treats the
branch on which one reduction vanishes.

For $V_i$ use $(0,s,1,1)$. The analogous products are
$-2(J_i-N_is)$ and $(R_i-S_is)(J_i-N_is)$, giving
\[
 \langle V_i,U_j\rangle=[J_i/q_j],\qquad
 \langle V_i,V_j\rangle=[R_iJ_i/q_j].
\]
Symmetry of the additive Cassels pairing gives \eqref{eq:normrec}. Hence the
lower right block is $B_\delta+B_\delta^t$, while $A_\delta$ is symmetric
with zero diagonal, and therefore alternating over $\F_2$.

It remains to evaluate $\langle U_i,V_i\rangle$. Write
$R=R_i,S=S_i,T=T_i,h=h_i,q=q_i$. Primitivity implies $q\nmid RT$.
At $q$, take
\[
 (t,u_1,u_2,u_3)=(T,-R,qS,w),\qquad w^2=R^2+4hT^2.
\]
The radicand reduces to $5hT^2$, a nonzero square. Choose the sign of
$w$ for which $\ell=(h+1)w+(h-1)R-4hT$ is a unit. Both signs could
vanish only if $h=-1$ and $R=2T$ modulo $q$, forcing $q=5$.
Here $L_3=-2R^2$ is a square. Every odd prime dividing $T$ sees $q$ as
a square by the norm equation and primitivity. Quadratic reciprocity,
together with $q\equiv1\pmod8$, therefore gives $(T/q)=1$.
Set $b=R/T,c=w/T,z=b+c$ modulo $q$. Then
\[
 b^2=h,\quad c^2=5h,\quad
 \ell/T=h(z-2)^2/z,\quad z=b(1+s_5),\quad s_5^2=5.
\]
Therefore its character is $(h/q)_4(1+s_5/q)$. By
Lemma~\ref{lem:pqcyclotomic}, $(1+s_5/q)=(5/q)_4$, so the product is
$(h/5\,/q)_4$. This is the diagonal entry in \eqref{eq:AB2}; the remaining
diagonal terms follow from alternation.
\end{proof}

Choose primitive norm representations
\begin{equation}\label{eq:globalnorms}
 n=a^2-5b^2=a_2^2+2b_2^2=a_{10}^2+10b_{10}^2,\qquad
 a,b>0,\quad4\mid a,\quad b\text{ odd}.
\end{equation}
They can be obtained by multiplying primitive prime norm generators.
Define column vectors $\alpha,\tau\in\F_2^r$ by
\[
 \alpha_j=[a/q_j],\qquad \tau_j=[5/q_j]_4,
\]
and put $\eps=0$ or $1$ according as $p\equiv11$ or $19\pmod{40}$.

\begin{theorem}\label{thm:allfour}
In the bases displayed above, the two odd parameter matrices are
\begin{align}
 C_+&=\begin{pmatrix}
 0&\eps&0^t&\alpha^t\\
 \eps&0&(\alpha+\tau)^t&\tau^t\\
 0&\alpha+\tau&A_1&B_1\\
 \alpha&\tau&B_1^t&B_1+B_1^t
 \end{pmatrix},\label{eq:Cplus}\\
 C_-&=\begin{pmatrix}
 0&1+\eps&0^t&(\alpha+\tau)^t\\
 1+\eps&0&\alpha^t&\tau^t\\
 0&\alpha&A_{-1}&B_{-1}\\
 \alpha+\tau&\tau&B_{-1}^t&B_{-1}+B_{-1}^t
 \end{pmatrix}.\label{eq:Cminus}
\end{align}
For $\delta=2$ the full matrix is $C_2=H_2$.
For $\delta=-2$, put $J_*=2a+5b$ and define
\begin{align}
 z^U_j&=[a_2a_{10}/q_j]+\tau_j,&
 z^V_j&=[a_2J_*/q_j],\label{eq:zrows}\\
 w^U_j&=z^U_j+\sum_i(A_{-2})_{ij},&
 w^V_j&=z^V_j+\sum_i(B_{-2})_{ij}.\label{eq:wrows}
\end{align}
Then
\begin{equation}\label{eq:Cfourth}
 C_{-2}=\begin{pmatrix}
 0&(w^U)^t&(w^V)^t\\
 w^U&A_{-2}&B_{-2}\\
 w^V&B_{-2}^t&B_{-2}+B_{-2}^t
 \end{pmatrix}.
\end{equation}
All symbols and matrices are independent of the permitted choices of
primitive norm representations.
\end{theorem}
\begin{proof}
The $U,V$ blocks are Theorem~\ref{thm:UV}.
For the rows involving $\Lambda,\Lambda'$, the tangent computations of
Theorems~\ref{thm:pqcassels} and~\ref{thm:threecassels} apply at each $q_j$ separately.
The $\Lambda$ tangents are
\[
 nt+bu_2-au_3,\quad (n+1)u_3-(n-1)u_1-4nt,\quad
 au_1-bu_2-nt.
\]
At $(0,1,s_5,1)$ the two required products have characters $1$ and
$(a/q_j)$. The analogous $\Lambda'$ calculation gives
$(a/q_j)(2+s_5/q_j)$ and $(2+s_5/q_j)$, respectively.
Thus its entries are $\alpha_j+\tau_j$ and $\tau_j$.
For the $\Lambda,\Lambda'$ entry all primes dividing $n$ give a trivial
symbol with parameter $5$; the contributions at $2,5$ depend only on
$n\bmod40$, yielding $\eps$.
For the negative twist use
$a_-=2a+5b$, $b_-=a+2b$, with $a_-^2-5b_-^2=-n$.
At $q_j$, $[a_-/q_j]=\alpha_j+\tau_j$; the dyadic/5-adic calculation
replaces $\eps$ by $1+\eps$, giving \eqref{eq:Cplus}--\eqref{eq:Cminus}.

For the last matrix add a rational torsion class:
\[
 Z=W+\sum_iU_i+t_m=(-10,5,-2),\qquad m=-10n.
\]
Its covering conics are
\[
 2u_3^2+5u_2^2=50nt^2,\quad
 u_3^2-5u_1^2=20nt^2,\quad
 2u_1^2+u_2^2=2nt^2.
\]
Writing $D_*=a+2b$, so that $5D_*^2-J_*^2=n$, take tangents
\[
 L_1=a_{10}u_3+5b_{10}u_2-5nt,\quad
 L_2=D_*u_3-J_*u_1-2nt,\quad
 L_3=a_2u_1+b_2u_2-nt.
\]
The rational points are respectively $(t,u_2,u_3)=(1,10b_{10},5a_{10})$,
$(t,u_1,u_3)=(1,2J_*,10D_*)$, and $(t,u_1,u_2)=(1,a_2,2b_2)$.
At $q_j$ take $(0,1,s_2,s_5)$ with $s_2^2=-2,s_5^2=5$.
Choose $s_2$ so $L_3=2a_2$ modulo $q_j$. For the $U_j$ entry choose
$s_5$ so $L_1=2a_{10}s_5$; for the $V_j$ entry choose it so $L_2=-2J_*$.
All these reductions are nonzero by primitivity. Only $q_j$ contributes,
as in Theorem~\ref{thm:UV}. Hence \eqref{eq:zrows} holds. Bilinearity, together with the vanishing of
the pairing on rational torsion, gives
\eqref{eq:wrows} and \eqref{eq:Cfourth}. Choice independence follows
from these intrinsic pairing computations.
\end{proof}

\begin{remark}
For $r=1$ the formulas recover exactly the four two prime matrices in
Theorems~\ref{thm:pqcassels},~\ref{thm:threecassels}, and~\ref{thm:fourthcassels}.
For $r\ge2$ the off diagonal entries are genuinely additional:
$B_\delta$ need not be symmetric, and $A_\delta$ need not vanish.
Thus juxtaposing independent copies of the two prime matrix is incorrect.
The theorem is a formula for the matrices. No joint distribution statement is
claimed when several auxiliary primes vary simultaneously.
\end{remark}

\begin{example}\label{ex:ordinary}
Let $n=59\cdot241\cdot281=3995539$ and $\delta=2$.
The three rows of norm data $(R_i,S_i,T_i)$,
$(K_i,L_i,M_i)$, and $(J_i,N_i,P_i)$ can be chosen as
\[
\begin{array}{c|rrr|rrr|rrr}
q_i&R_i&S_i&T_i&K_i&L_i&M_i&J_i&N_i&P_i\\\hline
241&2449&11&6&2382245&85565&2172&2341&349&12\\
281&7547&17&20&58645&735&68&4639&291&4.
\end{array}
\]
They give
\[
 A_2=\begin{pmatrix}0&1\\1&0\end{pmatrix},\quad
 B_2=\begin{pmatrix}0&0\\1&1\end{pmatrix},\quad
 C_2=\begin{pmatrix}0&1&0&0\\1&0&1&1\\0&1&0&1\\0&1&1&0\end{pmatrix}.
\]
Its Pfaffian is one. Hence
\[
 \rank A_{39955390}(\Q)=0,\qquad
 \Sh(A_{39955390})[2^\infty]\simeq(\Z/2\Z)^4.
\]
All four ordinary ranks for this parameter are $(4,4,4,4)$; their pure
space dimensions are $(6,6,4,5)$. The last matrix therefore has a
one dimensional radical, to which Theorem~\ref{thm:tower} applies.
\end{example}

\subsection{The next pairing on a two dimensional radical}\label{subsec:twodimnext}

\begin{lemma}\label{lem:next}
Let $E/\Q$ be an elliptic curve and let $R\subset\Sel_2^0(E)$ be the ordinary
Cassels radical. For $u,v\in R$, write $\bar u,\bar v$ for their images in
$\Sh(E)[2]$, and choose $\widetilde u\in\Sh(E)[4]$ with
$2\widetilde u=\bar u$. Then
\[
 \mathcal B(u,v)=\langle\widetilde u,\bar v\rangle_{\mathrm{CT}}
 \in\tfrac12\Z/\Z\simeq\F_2
\]
is well defined and alternating. Its radical is
\[
 \operatorname{im}(\Sel_8(E)\longrightarrow\Sel_2^0(E)).
\]
\end{lemma}
\begin{proof}
Write $S=\Sh(E)[2]$ and $R_S=2\Sh(E)[4]$.
The ordinary radical on $S$ is $R_S$ by the divisibility criterion.
Two halves of $\bar u$ differ by an element of $S$, which pairs trivially
with $\bar v\in R_S$. Bilinearity is therefore well defined, and
$\langle\widetilde u,2\widetilde u\rangle=0$ proves alternation.

Suppose a half $x$ of $\bar u$ pairs trivially with $R_S$.
The functional $s\mapsto\langle x,s\rangle$ on $S$ factors through
$S/R_S$. The ordinary pairing on this finite quotient is nondegenerate,
so some $t\in S$ cancels that functional on all of $S$.
The divisibility criterion now implies $x+t=2z$ with
$z\in\Sh(E)[8]$. Thus $\bar u=4z$.
Conversely, if $\bar u=4z$, choose the half $2z$.
For $\bar v=2y$, $y\in\Sh(E)[4]$, its value is
$\langle2z,2y\rangle=4\langle z,y\rangle=0$.
The Mordell--Weil contribution is in every radical, and the Kummer
sequences identify the corresponding preimage in $\Sel_2^0(E)$ with the stated
$8$-Selmer image. No finiteness assumption on $\Sh$ has been used.
\end{proof}

We next compute the higher pairing on the simplest composite locus. Assume
$p\equiv11\pmod{40}$, take \eqref{eq:globalnorms}, and suppose
\begin{equation}\label{eq:higherconditions}
 \alpha_i=\tau_i=0\quad(1\le i\le r),\qquad
 H_1\text{ is nondegenerate}.
\end{equation}
Theorem~\ref{thm:allfour} then gives
\begin{equation}\label{eq:rad2}
 C_+=0_2\oplus H_1,\qquad
 R=\langle\Lambda,\Lambda'\rangle.
\end{equation}
Put $k=(-1)^{(b-1)/2}$ and define
\begin{align}
 F_n(u,v)&=(a+5b)u^2+10(a+b)uv+5(a+5b)v^2,\label{eq:Fn}\\
 G_n(u,v)&=(5b-a)u^2+10(a-b)uv+5(5b-a)v^2,\label{eq:Gn}\\
 \mathcal F_n&=kF_n,\qquad \mathcal G_n=kG_n.\nonumber
\end{align}

\begin{theorem}\label{thm:compositehigher}
Under \eqref{eq:family} and \eqref{eq:higherconditions}, the conic
\begin{equation}\label{eq:newconic}
 \Gamma_n:\quad c^2=\mathcal F_n(u,v)
\end{equation}
has a rational point. If $(u_0,v_0,c_0)$ is any primitive integral point,
then $5\nmid u_0c_0$, and the next pairing on the radical
\eqref{eq:rad2} has matrix
\begin{equation}\label{eq:higherformula}
 \begin{pmatrix}0&e_n\\e_n&0\end{pmatrix},\qquad
 (-1)^{e_n}=-k\leg{c_0}{5}.
\end{equation}
The right side is therefore independent of the primitive point and
of the admissible norm representation of $n$.
Moreover,
\begin{equation}\label{eq:shaexact}
 e_n=1\quad\Longleftrightarrow\quad
 \rank A_{5n}(\Q)=0\ \text{ and }\
 \Sh(A_{5n})[2^\infty]\simeq
 (\Z/2\Z)^{2r}\oplus(\Z/4\Z)^2.
\end{equation}
Both implications are unconditional.
\end{theorem}

\begin{proof}
The prime case construction extends to this composite setting, but its local
solubility must be checked anew. We organize the proof into the norm
characters, the conic, the covering, and the local pairing.

\emph{Norm characters.}
Quadratic reciprocity with Jacobi symbols gives
\begin{equation}\label{eq:jaccharacters}
 \leg bn=k,\qquad \leg{a+b}n=1,\qquad
 \leg an=\leg a5=1.
\end{equation}
Indeed $n\equiv a^2\pmod b$ and $n\equiv-4b^2\pmod{a+b}$.
Writing $a=2^ea'$ with $a'$ odd, one has
$(a'/n)=(-n/a')=(5/a')=(a'/5)$, while
$(2/n)=(2/5)=-1$. Since $n\equiv1\pmod5$, $a\equiv\pm1\pmod5$.
All numerators and denominators in these symbols are coprime.

At $q_i$, let $s_i=a/b$, so $s_i^2=5$ modulo $q_i$.
The conditions $\alpha_i=\tau_i=0$ imply
\begin{equation}\label{eq:qchars}
 \leg a{q_i}=\leg b{q_i}=\leg{a+b}{q_i}=1.
\end{equation}
Here $(s_i/q_i)=(5/q_i)_4=1$, and
$(1+s_i/q_i)=(2+s_i/q_i)=(5/q_i)_4=1$.
Dividing \eqref{eq:jaccharacters} by all these $q_i$ contributions yields
\begin{equation}\label{eq:pchars}
 \leg ap=\leg{a+b}p=1,\qquad\leg bp=k.
\end{equation}

\emph{The conic.}
Write $A=k(a+5b)$, $B=5k(a+b)$, so
$\mathcal F_n=Au^2+2Buv+5Av^2$.
Then $A\equiv1\pmod4$ and $5A^2-B^2=-20n$.
At $2$, one of $(u,v)=(1,0),(0,1)$ gives an odd square value.
At $5$, $A\equiv ka$ is a nonzero square. For every $\ell\mid n$,
\eqref{eq:qchars}--\eqref{eq:pchars} and
$a+5b\equiv a(a+b)/b\pmod\ell$ show that $A$ is a nonzero square.
At other odd primes the ternary conic is smooth and isotropic over the
finite field; at infinity the binary form is indefinite.
The Hasse principle gives a rational point.
If a primitive integral point had $5\mid u_0$, its equation would force
$5\mid c_0$ and then give valuation one on the right with $v_0$ a unit,
a contradiction. Thus $5\nmid u_0c_0$.

\emph{The four cover and its pushout.}
Consider the smooth normalization $Y_n$ of
\begin{equation}\label{eq:Y}
 c^2=\mathcal F_n(u,v),\qquad d^2=\mathcal G_n(u,v),\qquad
 kh^2=c^2-cd-d^2.
\end{equation}
The geometric constructions in Proposition~\ref{prop:4cover} and Theorem~\ref{thm:higherpair}
apply here without a primality assumption on the parameter.
The identities used below are
\begin{gather*}
 F_n^2-3F_nG_n+G_n^2=5n(u^2-5v^2)^2,\\
 \disc F_n=\disc G_n=80n,\qquad
 \operatorname{Res}(F_n,G_n)=2000n^2,\\
 \det(\lambda M_1+\mu M_2)
 =-20n\lambda\mu(\lambda^2+3\lambda\mu+\mu^2).
\end{gather*}
The first two equations define a smooth genus one curve $C_n$ with
Jacobian $Y^2=X^3-60nX^2+400n^2X$, whose only nonzero rational
$2$-torsion point is $(0,0)$. The map
\[
 (u:v:c:d)\longmapsto
 (k(u^2-5v^2):c^2+d^2:cd:c^2-d^2)
\]
is a degree two morphism to the $\Lambda'$ covering, in coordinates
$(t_E:u_1:u_2:u_3)$. It has no base points by the resultant identity.
The contact factorizations of $c^2-cd-d^2$ show that adjoining $h$
gives an unramified double cover. It is geometrically connected:
a putative square root with poles in a hyperplane would be linear in
$u,v,c,d$, and coefficient comparison would force
$e^2=1,f^2=-1,2ef=-1$, which is impossible in characteristic different
from $5$. The unique rational order two subgroup of the intermediate
Jacobian identifies the composite Jacobian isogeny with $[2]$ on $A_{5n}$.
Thus $Y_n$ will be a $4$-covering lifting $\Lambda'$ once local solubility is
proved. The geometric argument uses only $n\ne0$ and the nonsquareness of
$5$ and $80n$; primality of the parameter plays no role.

Let $\ell_0$ be the tangent to \eqref{eq:newconic} at $(u_0,v_0,c_0)$:
\[
 \ell_0=c_0c-(Au_0+Bv_0)u-(Bu_0+5Av_0)v.
\]
On $Y_n$, put $(U,V)=(c,-d)$ for $k=1$ and $(U,V)=(-c-d,c)$ for
$k=-1$. Then $U^2+UV-V^2=h^2$. The function
\[
 f=(3U-V+2h)/5,\qquad g=f\ell_0/r_0^2
\]
for any nonzero linear section $r_0$ has divisor twice a divisor class
representing $T=(5n,0)$. The proof in Theorem~\ref{thm:higherpair} identifies
this class as the nonzero element of the kernel of
$\operatorname{Pic}^0(Y_n)\to\operatorname{Pic}^0(C_n)$.
Its nonsquare check uses only the nondegeneracy of $\mathcal G_n$,
$c_0\ne0$, and the rank one identity for the polar form of
$\mathcal F_n$; all hold for the present squarefree $n$.
Hence it is the pushout for $T$. The norm identities
\[
 f(2U+V+2h)=(U+h)^2,\qquad
 \Norm_{\Q(\sqrt5)/\Q}(U+\phi V+h)=h(2U+V+2h),
 \quad\phi=(1+\sqrt5)/2,
\]
therefore give
\begin{equation}\label{eq:prod}
 (-1)^{e_n}=\prod_v(g(P_v),5)_v
      =\prod_v(h(P_v)\ell_0(P_v),5)_v.
\end{equation}
Here $\Lambda=\chi_5T$, and $\Lambda,\Lambda'\in R$ by
\eqref{eq:rad2}. Thus Lemma~\ref{lem:next} identifies the computed
pairing with the next pairing on that radical.

\emph{Local points and symbols.}
At a prime $\ell\mid n$ take $u=1,v=0$. Both
$\mathcal F_n(1,0)$ and $\mathcal G_n(1,0)$ are nonzero squares.
With $s=a/b\pmod\ell$ and $\phi_s=(1+s)/2$, their ratio is $\phi_s^2$.
Choose $c/d=-\phi_s$. Then
\[
 c^2-cd-d^2=2\phi_s d^2.
\]
At $p$, $(\phi_s/p)=-k$ by \eqref{eq:pchars}, whereas at each $q_i$,
$(\phi_s/q_i)=1$ by \eqref{eq:qchars}. In both cases $(2\phi_s/k\,/\ell)=1$.
All three square roots in \eqref{eq:Y} therefore exist locally.
The Hilbert symbol at these primes is trivial because $5$ is a local square.

At $2$, choose $(u,v)$ from
\[
\begin{array}{cc|c}
k(a+5b)\bmod8&n\bmod16&(u,v)\\\hline
1&3&(1,0)\\1&11&(1,2)\\5&3&(2,1)\\5&11&(0,1).
\end{array}
\]
The first two values are $1$ modulo $8$ and their product is $1$ modulo
$16$. Choose odd roots $c,d$ with $cd\equiv-k\pmod8$; the resulting $h$
is odd. The tangent parity proof in Lemma~\ref{arith:dyadic} depends only on
$4\mid a$, odd $b$, and $n\equiv11\pmod{40}$, and gives
$v_2(\ell_0)\equiv1\pmod2$ away from the tangent zero.
Hence $(h\ell_0,5)_2=-1$.
At $5$ take $u=v=1$, $c\equiv-c_0/u_0$, $d\equiv3c$.
If $\eps_a=1$ for $a\equiv1\pmod5$ and $-1$ otherwise, then
\[
 (h/5)=-\eps_a,\quad(\ell_0/5)=-(u_0/5),\quad
 (h\ell_0,5)_5=\eps_a(u_0/5)=k(c_0/5).
\]
The roots are units, so they lift. At infinity the indefiniteness of
$F_n,G_n$ and positive definiteness of $F_n-G_n$ permit signs of $c,d$
for which $(c^2-cd-d^2)/k>0$; the Hilbert symbol is trivial.

For $\ell\nmid10n$, the determinant and contact identities give a smooth
proper genus one model of $Y_n$ over $\Z_\ell$, exactly as in the proof
of Proposition~\ref{prop:4cover}; its finite field point lifts.
If $5$ is nonsquare locally, anisotropy of $c^2-cd-d^2$ modulo $\ell$
and the unit resultant force $h$ to be a unit at a primitive point.
The tangent $\ell_0$ has even valuation by good reduction. Its Hilbert
symbol is therefore one. Points may be moved off the finitely many
excluded divisors. Hence the four cover is locally soluble, and
\eqref{eq:prod} reduces to \eqref{eq:higherformula}.

The left side is the intrinsic pairing in the fixed basis of $R$; hence the
character is independent of both the conic point and the chosen norm
representation of $n$.
Finally the ordinary pairing has rank $2r$, and its next pairing is
nondegenerate exactly when $e_n=1$. The finite nondivisible $2$-primary
part of $\Sh$ has paired elementary divisors. Ordinary rank $2r$
accounts for $2r$ cyclic factors of order $2$; next nondegeneracy on the
remaining two dimensional radical accounts for exactly two factors of
order $4$, and excludes Mordell--Weil rank and divisible $2$-primary
classes. This establishes \eqref{eq:shaexact} in both directions.
\end{proof}

\begin{corollary}
\label{cor:sel8rad2}
Under the hypotheses of Theorem~\ref{thm:compositehigher},
\[
 \Sel_8(A_{5n})\simeq
 \begin{cases}
 (\Z/4\Z)^2\oplus(\Z/2\Z)^{2r+2},&e_n=1,\\
 (\Z/8\Z)^2\oplus(\Z/2\Z)^{2r+2},&e_n=0.
 \end{cases}
\]
The image in the pure $2$-Selmer group is zero in the first case and
$\langle\Lambda,\Lambda'\rangle$ in the second.
\end{corollary}
\begin{proof}
The image assertion follows from Lemma~\ref{lem:next}.
Use the noncanonical splitting of Lemma~\ref{lem:split} below.
The ordinary rank accounts for $2r$ cyclic factors of order two in the
finite $2$-primary part of $\Sh$. If $e_n=1$, apply
\eqref{eq:shaexact} directly. If $e_n=0$, the next pairing excludes
paired elementary divisors of order four: such a pair would contribute
rank two to that pairing. The remaining two radical dimensions therefore
come either from paired elementary divisors of order at least eight or
from the Mordell--Weil and divisible-$\Sh$ contributions. In either
case they give exactly two cyclic factors of order eight in $\Sel_8$.
The rational $2$-torsion contributes two additional order two factors.
The two group formulas follow, without any finiteness assumption on $\Sh$.
\end{proof}

\begin{corollary}\label{cor:twohigher}
Let $p\equiv11\pmod{40}$ be fixed, and let
$q\equiv1\pmod{40}$ with $(p/q)=1$. In the fixed character notation of \eqref{eq:pqfixedbits}, assume
\[
 \beta_p(q)=0,\qquad [5/q]_4=0,\qquad[p/q]_4=1.
\]
Then $R=\langle\Lambda,\Lambda'\rangle$, and \eqref{eq:higherformula}
computes its complete next pairing. In particular
\[
 k(c_0/5)=1\quad\Longleftrightarrow\quad
 \rank A_{5pq}(\Q)=0,\quad
 \Sh(A_{5pq})[2^\infty]\simeq(\Z/2\Z)^2\oplus(\Z/4\Z)^2.
\]
The displayed ordinary character locus has relative density $1/8$
in $\mathcal P_p$. This does not determine the relative density of the
additional condition $k(c_0/5)=1$.
\end{corollary}
\begin{proof}
Here $\alpha=\beta+\tau=0$ and
$H_1=\left(\begin{smallmatrix}0&1\\1&0\end{smallmatrix}\right)$.
Apply Theorem~\ref{thm:compositehigher}. The density follows from the
independent uniform ordinary characters already proved in
Theorem~\ref{thm:pqfield}; no distribution assertion about the new character
is used.
\end{proof}

\begin{example}\label{ex:twohigher}
Take $p=211,q=3001$, so $n=633211$.
A root of $5$ modulo $3001$ is $245$, and a fixed norm generator of $211$
is $16+3\sqrt5$. The three additive characters are $(\beta,\tau,\gamma)=(0,0,1)$.
Choose
\[
 n=796^2-5\cdot9^2,\quad k=1,\qquad
 \Gamma_n:\ c^2=841u^2+8050uv+4205v^2.
\]
The primitive point $(1,0,29)$ has $(29/5)=1$. Hence
\[
 \rank A_{3166055}(\Q)=0,\qquad
 \Sh(A_{3166055})[2^\infty]\simeq(\Z/2\Z)^2\oplus(\Z/4\Z)^2.
\]
Direct representation counting gives $\D(633211)=64$.
That last count is a consistency check, not an ingredient in the descent proof.
\end{example}

\begin{example}\label{ex:threehigher}
Let
\[
 n=11\cdot521\cdot6961=39893491=33544^2-5\cdot14733^2.
\]
All pairwise quadratic residue conditions hold, and
$\alpha_1=\alpha_2=\tau_1=\tau_2=0$.
The auxiliary prime block is
\[
 H_1=\begin{pmatrix}
 0&1&0&1\\1&0&0&0\\0&0&0&1\\1&0&1&0
 \end{pmatrix},\qquad \Pf(H_1)=1.
\]
Here $k=1$, and the conic has primitive point
\[
 (u_0,v_0,c_0)=(-278203,142067,5992946),\qquad (c_0/5)=1.
\]
Theorem~\ref{thm:compositehigher} gives
\[
 \rank A_{199467455}(\Q)=0,\qquad
 \Sh(A_{199467455})[2^\infty]\simeq(\Z/2\Z)^4\oplus(\Z/4\Z)^2.
\]
Independent direct counts are $r_F(n)=15360$, $r_G(n)=15616$, so
$\D(n)=-256$. Again the conclusion about $\Sh$ uses the proved pairing,
not a presumed equality between a representation bit and that pairing.
\end{example}

\subsection{Higher pairings on larger ordinary radicals}
\label{subsec:largeradicals}

In Theorem~\ref{thm:compositehigher}, the nondegeneracy of the auxiliary
block serves only to identify the whole ordinary radical with
$\langle\Lambda,\Lambda'\rangle$. The $4$-cover itself, and the evaluation
of its distinguished higher entry, remain valid when the auxiliary block is
singular. This observation is what allows us to pass to four dimensional
radicals.

\begin{theorem}
\label{thm:largeradical}
Assume \eqref{eq:family}, let $p\equiv11\pmod{40}$, choose
\[
 n=a^2-5b^2,\qquad a,b>0,\quad (a,b)=1,\quad4\mid a,\quad b\text{ odd},
 \qquad k=(-1)^{(b-1)/2},
\]
and suppose
\[
 \alpha_i=[a/q_i]=0,\qquad \tau_i=[5/q_i]_4=0
 \qquad(1\le i\le r).
\]
No nondegeneracy assumption is imposed on $H_1$. Then
\[
 C_+=0_2\oplus H_1,\qquad
 R=\rad C_+=\langle\Lambda,\Lambda'\rangle\oplus\ker H_1.
\]
The conic
\[
 \Gamma_n:\quad c^2=kF_n(u,v),
\]
with $F_n$ as in \eqref{eq:Fn}, has a rational point, and the smooth
normalization $Y_n$ of \eqref{eq:Y} is an everywhere locally soluble
$4$-covering lifting $\Lambda'$. For every primitive integral point
$P_0=(u_0,v_0,c_0)$ on $\Gamma_n$,
\begin{equation}\label{eq:largeradicalentry}
 (-1)^{\mathcal B(\Lambda',\Lambda)}=-k\leg{c_0}{5}.
\end{equation}
The value is independent of the primitive point and of the admissible norm
representation. If $\mathcal B(\Lambda',\Lambda)=1$, equivalently if the
right side of \eqref{eq:largeradicalentry} is $-1$, then
\[
 \rank\mathcal B\ge2,
 \qquad
 \dim\operatorname{im}\bigl(\Sel_8(A_{5n})\to\Sel_2^0(A_{5n})\bigr)
 \le \dim\ker H_1.
\]
\end{theorem}

\begin{proof}
The block decomposition follows from Theorem~\ref{thm:allfour}. It remains
to check that no later step uses the rank of $H_1$. Put
$s_i=a/b\pmod{q_i}$. The assumptions $\alpha_i=\tau_i=0$ imply
\[
 \leg a{q_i}=\leg b{q_i}=\leg{a+b}{q_i}=1,
\]
using $(s_i/q_i)=(1+s_i/q_i)=(2+s_i/q_i)=1$. Dividing the global Jacobi
identities \eqref{eq:jaccharacters} by these auxiliary prime contributions
gives \eqref{eq:pchars}. Hence the local solubility proof of
Theorem~\ref{thm:compositehigher} applies unchanged to the conic
$c^2=kF_n(u,v)$.

The geometric identities in \eqref{eq:Y}--\eqref{eq:prod}, including the
quadric intersection determinant, the contact factorizations and the
pushout divisor for $T=(5n,0)$, depend on $n$, $a$, $b$ and the two binary
forms, but not on the rank of $H_1$. The same is true of the local points:
at primes dividing $n$ they use only the character identities above; at
$2$ and $5$ they use $4\mid a$, $b$ odd and $n\equiv11\pmod{40}$; away
from $10n$ they use good reduction. Thus $Y_n$ remains everywhere locally
soluble and the pushout evaluation gives
\[
 (-1)^{\mathcal B(\Lambda',\Lambda)}
 =\prod_v(h(P_v)\ell_0(P_v),5)_v
 =-k\leg{c_0}{5}.
\]
Since $\Lambda,\Lambda'\in R$ regardless of the rank of $H_1$, this is the
intrinsic next pairing on the ordinary radical. Its radical is the image
of $\Sel_8$ by Lemma~\ref{lem:next}; a nonzero alternating $2\times2$
subblock has rank two, which proves the final bound.
\end{proof}

The same pushout also controls the auxiliary $U_i$-directions: only the
quadratic character in the second argument changes.
For $x\in\F_q^\times$ write $[x/q]\in\F_2$ for the additive Legendre
symbol.

\begin{theorem}
\label{thm:auxcrossU}
Under the hypotheses of Theorem~\ref{thm:largeradical}, choose
$P_0=(u_0,v_0,c_0)$ with $q_i\nmid c_0$ for every $i$. Put
\[
 s_i=a/b\pmod{q_i},\qquad z_i^2=k(s_i-1)\pmod{q_i},
\]
and define
\begin{equation}\label{eq:muaux}
 \mu_i=
 \begin{cases}
 ((7-s_i)/2+2z_i)/5,&k=1,\\[2pt]
 ((3s_i-11)/2+2z_i)/5,&k=-1.
 \end{cases}
\end{equation}
A nonzero choice of $\mu_i$ is available. If $U_i\in R$, then
\begin{equation}\label{eq:auxcrossU}
 \mathcal B(\Lambda',U_i)=\left[\frac{c_0\mu_i}{q_i}\right].
\end{equation}
More generally, for $w=\sum_i x_iU_i\in R$,
\[
 \mathcal B(\Lambda',w)=\sum_i x_i[c_0\mu_i/q_i].
\]
On this radical subspace the right side is independent of all permitted
choices.
\end{theorem}

\begin{proof}
Weak approximation on the rational conic $\Gamma_n$ permits the primitive
point to be chosen with $q_i\nmid c_0$ simultaneously. The Kummer class
$U_i=(q_i,1,q_i)$ is $\chi_{q_i}T$ for the same
$T=(5n,0)$ whose pushout function $g=f\ell_0/r_0^2$ was identified in the
proof of Theorem~\ref{thm:compositehigher}. Hence
\[
 (-1)^{\mathcal B(\Lambda',U_i)}=\prod_v(g(P_v),q_i)_v.
\]
At $2,5,p$, at every $q_j\ne q_i$, and at infinity, the second Hilbert
parameter is a local square or is positive. At an odd prime
$\ell\nmid10n$, both tangent sections have even valuation on the smooth
model, so only $q=q_i$ contributes.

Write $s=s_i$ and $z=z_i$. At $q$ use the reduction
\[
 u=1,\qquad v=0,\qquad d=\frac{1-s}{2}c,\qquad h=zc,
 \qquad c^2=k(a+5b).
\]
All entries are units. Choose the sign of $c$ so that
$\ell_0(P_q)\equiv2c_0c\pmod q$. Substitution into the explicit tangent
$f=(3U-V+2h)/5$ gives $f(P_q)/c=\mu_i$. Therefore, modulo local squares,
\[
 g(P_q)\equiv2c_0c^2\mu_i\equiv c_0\mu_i,
\]
because $(2/q)=1$. This is \eqref{eq:auxcrossU}. The two possible
values of $\mu_i$ differ by $4z_i/5$, so at least one is nonzero; moreover,
replacing $z_i$ by $-z_i$ does not change the character of a nonzero value,
since $f(h)f(-h)=(U-V)^2/5$ is a square modulo $q$. Bilinearity gives the
formula for a linear combination, and intrinsicness of the next pairing
gives choice independence on $R$.
\end{proof}

\subsection{A second pushout and the complete \texorpdfstring{$\Lambda'$}{Lambda prime} row}
\label{subsec:secondpushout}

The previous calculation sees the rational $2$-torsion point $T_m=(5n,0)$.
To reach the $V_i$-directions we need the other independent point
$T_0=(0,0)$. We now construct a second pushout for $T_0$; together with the
first one, it completes the $\Lambda'$ row of the next pairing.

Put
\begin{equation}\label{eq:Gamma0}
 \Gamma_0:\quad A_c(X^2+5Y^2)+2B_cXY=10nZ^2,
 \qquad A_c=3a+5b,\quad B_c=5(a+3b).
\end{equation}
Then
\begin{equation}\label{eq:Gamma0det}
 5A_c^2-B_c^2=20n.
\end{equation}
On the dense open subset of $Y_n$ on which
$\Omega=u^2-5v^2$ and $h$ are nonzero, put
\begin{equation}\label{eq:Rsections}
 J_0=\frac{\Omega}{h},\qquad
 \mathcal R_1=\frac{h((2c+d)u+5dv)}{2\Omega},\qquad
 \mathcal R_2=\frac{h((2c+d)v+du)}{2\Omega}.
\end{equation}

\begin{lemma}
\label{lem:secondsections}
The sections $J_0,\mathcal R_1,\mathcal R_2$ extend everywhere on the normalization $Y_n$.
With
\[
 \kappa_1=\frac{5(a-b)}2,\qquad \kappa_2=\frac{a-5b}2,
\]
one has
\begin{align}
 J_0^2&=\frac{c^2+cd-d^2}{5kn},\label{eq:Jsquaremain}\\
 \mathcal R_1^2-5\mathcal R_2^2&=5kn\Omega,\label{eq:Rnormmain}\\
 \mathcal R_1^2+5\mathcal R_2^2&=\kappa_1(c^2+d^2)+5\kappa_2 cd,\label{eq:Rsummain}\\
 2\mathcal R_1\mathcal R_2&=\kappa_2(c^2+d^2)+\kappa_1 cd.\label{eq:Rprodmain}
\end{align}
Consequently
\begin{equation}\label{eq:secondconicmap}
 A_c(\mathcal R_1^2+5\mathcal R_2^2)+2B_c\mathcal R_1\mathcal R_2=10n(c+d)^2,
\end{equation}
and
\[
 \nu_0:Y_n\longrightarrow\Gamma_0,\qquad
 (u:v:c:d:h)\longmapsto(\mathcal R_1:\mathcal R_2:c+d)
\]
is a degree four morphism. In particular $\Gamma_0(\Q)\ne\varnothing$.
\end{lemma}

\begin{proof}
Let $s=\sqrt5$, $\phi=(1+s)/2$ and
$A_0=2a+(a-5b)\phi=\kappa_1+\kappa_2 s$. The contact identity is
\[
 F_n-\phi^2G_n=A_0(u-sv)^2,\qquad \Norm(A_0)=5n.
\]
The definitions give
\[
 \mathcal R_1+s\mathcal R_2=\frac{h(c+\phi d)}{u-sv}.
\]
Using
$h^2=k(c-\phi d)(c-\bar\phi d)$ and
$(c-\phi d)(c+\phi d)=kA_0(u-sv)^2$ yields
\[
 (\mathcal R_1+s\mathcal R_2)^2=A_0(c^2+s cd+d^2).
\]
Equating the rational and $s$-parts gives
\eqref{eq:Rsummain}--\eqref{eq:Rprodmain}; taking norms gives
\eqref{eq:Rnormmain}. The identity
\[
 (c^2-cd-d^2)(c^2+cd-d^2)=5n\Omega^2
\]
gives \eqref{eq:Jsquaremain}. Thus the squares of the displayed rational
sections are regular quadratic expressions in the projective coordinates;
on a smooth curve the sections themselves are regular. Substitution gives
\eqref{eq:secondconicmap}.

The three sections $\mathcal R_1,\mathcal R_2,c+d$ have no common zero.
Indeed, a common zero would force $c=d=0$ by
\eqref{eq:Rsummain}--\eqref{eq:Rprodmain} and
$\kappa_1^2-5\kappa_2^2=5n$, contradicting the nonzero resultant of
$F_n,G_n$. The pullback of the conic hyperplane bundle has degree eight,
whereas a hyperplane on the conic has degree two, so $\nu_0$ has degree
four. Since $Y_n$ is everywhere locally soluble, its image gives a local
point on $\Gamma_0$ at every place. The Hasse principle for conics then
gives $\Gamma_0(\Q)\ne\varnothing$.
\end{proof}

Choose primitive integral points
\[
 P_0=(u_0,v_0,c_0)\in\Gamma_n(\Q),\qquad
 Q_0=(X_0,Y_0,Z_0)\in\Gamma_0(\Q),
\]
and let $\ell_0$ be the tangent to $\Gamma_n$ at $P_0$. The tangent to
$\Gamma_0$ at $Q_0$, pulled back to $Y_n$, is
\begin{equation}\label{eq:tangentsecond}
 t_0=(A_cX_0+B_cY_0)\mathcal R_1+(B_cX_0+5A_cY_0)\mathcal R_2-10nZ_0(c+d).
\end{equation}
For a nonzero linear section $r_0$ put
\begin{equation}\label{eq:gsecond}
 g_0=\frac{\ell_0t_0}{nr_0^2}\in\Q(Y_n)^\times.
\end{equation}

\begin{theorem}
\label{thm:secondpushout}
The function $g_0$ satisfies
\[
 \operatorname{div}(g_0)=2\mathfrak b_0,\qquad
 [\mathfrak b_0]=T_0=(0,0)
 \in A_{5n}[2]\simeq\operatorname{Pic}^0(Y_n)[2].
\]
Thus $g_0$ is a pushout for $T_0$, independent of the pushout for
$T_m=(5n,0)$. Their product is, up to a rational scalar and a square, a
pushout for $T_{-4m}=(-20n,0)$.
\end{theorem}

\begin{proof}
Let $D_F$ and $D_0$ be the fibres over $P_0$ and $Q_0$ of the two conic
maps. Both have degree four and satisfy $2D_F\sim2D_0\sim L$, where $L$
is the degree eight hyperplane bundle on $Y_n$. Tangency gives
\[
 \operatorname{div}(g_0)=2(D_F+D_0-H_{r_0}),
\]
so the half divisor class is $[D_0-D_F]\in\operatorname{Pic}^0(Y_n)[2]$.
The point is to determine which rational $2$-torsion class this is.

Consider the commuting fixed point free involutions
\[
 \tau:(u,v,c,d,h)\mapsto(u,v,c,d,-h),
 \qquad
 \sigma:(u,v,c,d,h)\mapsto(-u,-v,c,d,h).
\]
The established $4$-covering construction identifies $\tau$ with translation
by $T_m$. The involution $\sigma$ is fixed point free: a fixed projective
point would force either $u=v=0$ or $c=d=h=0$, and both alternatives are
excluded by \eqref{eq:Y} and the nonzero resultant of $F_n,G_n$. Hence
$\sigma$ is also translation by a nonzero rational $2$-torsion point.
Both $\tau$ and $\sigma$ negate $\mathcal R_1,\mathcal R_2$ and fix $c,d$.
Take $r_0=u$ and put
\[
 W_F=u_0v-v_0u,\qquad W_0=X_0\mathcal R_2-Y_0\mathcal R_1.
\]
The two polar form identities give
\[
 \ell_0\sigma(\ell_0)=-20nW_F^2,
 \qquad
 t_0\tau(t_0)=t_0\sigma(t_0)=20nW_0^2.
\]
Hence
\begin{align}
 g_0\sigma(g_0)&=-400\left(\frac{W_FW_0}{u^2}\right)^2,\label{eq:sigmanormmain}\\
 g_0\tau(g_0)&=20n\left(\frac{\ell_0W_0}{nu^2}\right)^2.\label{eq:taunormmain}
\end{align}
The square root in \eqref{eq:sigmanormmain} is $\sigma$-invariant, so the
half divisor class lies in the kernel of the quotient isogeny by $\sigma$.
For $\tau$ the situation is different: the displayed square root in
\eqref{eq:taunormmain} is $\tau$-odd, so the norm is not a square in the
quotient function field. Hence the class is nonzero and is the nontrivial
kernel point of the $\sigma$-quotient.

To identify that kernel point, use the $\sigma$-invariant sections
$(c,d,h,J_0)$. They realize the quotient as the intersection
\[
 c^2-cd-d^2=kh^2,\qquad
 c^2+cd-d^2=5knJ_0^2.
\]
Its pencil determinant is
\[
 -\frac{5n}{4}\lambda\mu(5\lambda^2+6\lambda\mu+5\mu^2),
\]
and its Jacobian has $j$-invariant $237276/625$. This is the
$j$-invariant of
\[
 A_{5n}/\langle T_0\rangle:\quad
 y^2=x^3-30nx^2+625n^2x.
\]
The quotients by $T_m$ and $T_{-4m}$ have respectively the distinct
$j$-invariants $55296/5$ and $132304644/5$. Therefore the $\sigma$ kernel
is precisely $T_0$, proving the assertion.
\end{proof}

\begin{theorem}
\label{thm:auxcrossV}
The primitive points $P_0,Q_0$ may be chosen with
$q_i\nmid c_0Z_0$ for every $i$. If $V_i\in R$, then
\begin{equation}\label{eq:auxcrossV}
 \displaystyle
 \mathcal B(\Lambda',V_i)=\left[\frac{c_0Z_0}{q_i}\right].
\end{equation}
More generally, for $w=\sum_i y_iV_i\in R$,
\[
 \mathcal B(\Lambda',w)=\sum_i y_i[c_0Z_0/q_i].
\]
The resulting functional on the radical is independent of the primitive
points and the admissible norm representation.
\end{theorem}

\begin{proof}
Weak approximation on the two rational conics permits the nonvanishing
conditions at the finitely many $q_i$. Since
$V_i=(1,q_i,q_i)=\chi_{q_i}T_0$, Theorem~\ref{thm:secondpushout} gives
\[
 (-1)^{\mathcal B(\Lambda',V_i)}=\prod_v(g_0(P_v),q_i)_v.
\]
As in Theorem~\ref{thm:auxcrossU}, every place except $q=q_i$ contributes
trivially. It remains to evaluate $g_0$ at $q$.

For each prime $\ell\mid n$, reduction of \eqref{eq:Gamma0} shows that
both $A_cX_0+B_cY_0$ and $B_cX_0+5A_cY_0$ are divisible by $\ell$; squarefreeness
of $n$ gives divisibility by $n$. Put
\[
 a_0=\frac{A_cX_0+B_cY_0}{n},\qquad
 b_0=\frac{B_cX_0+5A_cY_0}{n}.
\]
The adjugate of the binary Gram matrix gives
\begin{equation}\label{eq:gradnormmain}
 5A_c a_0^2-2B_c a_0b_0+A_c b_0^2=200Z_0^2.
\end{equation}
Write $s=a/b\pmod q$. Use the local reduction
\[
 u=1,\qquad v=0,\qquad c^2=k(a+5b),
 \qquad d=\frac{1-s}{2}c,\qquad h^2=4a,
\]
choosing the sign of $c$ so that
$\ell_0(P_q)\equiv2c_0c\pmod q$. Direct substitution into
\eqref{eq:Rsections} and \eqref{eq:tangentsecond} gives
\[
 \frac{t_0(P_q)}n
 =c\left(\frac{(1-s)h(b_0-sa_0)}4-5Z_0(3-s)\right).
\]
Reduction of \eqref{eq:gradnormmain} yields
$A_c(b_0-sa_0)^2=200Z_0^2$. Since
$A_c=b(3s+5)$, $h^2=4bs$ and $(3-s)(3s+5)=4s$, we may choose the sign of
$h$ so that the first term in parentheses equals $-5Z_0(3-s)$. Hence
\[
 \frac{t_0(P_q)}n=-10Z_0(3-s)c\not\equiv0\pmod q.
\]
Now $3-s=(s-1)^2/2$ and $q\equiv1\pmod{40}$, so the constant factor
$-20(3-s)$ is a square modulo $q$. Therefore
\[
 (g_0(P_q),q)_q
 =\leg{-20c_0Z_0(3-s)c^2}{q}
 =\leg{c_0Z_0}{q},
\]
which proves \eqref{eq:auxcrossV}. Bilinearity and intrinsicness give the
remaining assertions.
\end{proof}

\begin{corollary}
\label{cor:completeLambdarow}
Under the hypotheses above, put
\[
 e=\mathcal B(\Lambda',\Lambda),\qquad
 c_i=[c_0\mu_i/q_i],\qquad d_i=[c_0Z_0/q_i].
\]
Then for every
\[
 w=x\Lambda+y\Lambda'+\sum_i x_iU_i+\sum_i y_iV_i\in R
\]
one has
\begin{equation}\label{eq:completeLambdarow}
 \mathcal B(\Lambda',w)
 =xe+\sum_i x_ic_i+\sum_i y_id_i,
 \qquad (-1)^e=-k\leg{c_0}{5}.
\end{equation}
Thus one full row of the next pairing is explicit even when $H_1$ is
singular.
\end{corollary}

\begin{remark}\label{rem:fourdimnextmatrix}
For a four dimensional ordinary radical, in the ordered basis
$(\Lambda,\Lambda',U,V)$ write
\[
 C^{\mathrm{next}}=
 \begin{pmatrix}
 0&e&a_*&b_*\\
 e&0&c_*&d_*\\
 a_*&c_*&0&f_*\\
 b_*&d_*&f_*&0
 \end{pmatrix}.
\]
Theorems~\ref{thm:largeradical},~\ref{thm:auxcrossU}, and~\ref{thm:auxcrossV}
give uniform formulas for $e,c_*,d_*$, whereas
\[
 \Pf(C^{\mathrm{next}})=ef_*+a_*d_*+b_*c_*.
\]
On the locus $(e,c_*,d_*)=(0,0,1)$ this reduces to
$\Pf(C^{\mathrm{next}})=a_*=\mathcal B(\Lambda,U)$. Thus a single
additional entry decides nondegeneracy there. In general, however, the known
$\Lambda'$ row does not determine the remaining entries.
\end{remark}

\subsection{Four dimensional radicals and explicit full Selmer towers}
\label{subsec:fourdimtowers}

The next proposition is a useful group theoretic consequence of the
higher pairing filtration. It will turn the two explicit four dimensional
calculations below into complete Selmer tower statements.

\begin{proposition}
\label{prop:fourradicalcompletion}
Let $E=A_m$, let $s=\dim\Sel_2^0(E)$, and suppose that the ordinary radical
has dimension four. If the next pairing on this radical is nonzero and
$\rank E(\Q)\ge2$, then
\[
 \rank E(\Q)=2,\qquad
 \Sh(E)[2^\infty]\simeq(\Z/2\Z)^{s-4}\oplus(\Z/4\Z)^2,
\]
and for every $j\ge1$,
\begin{equation}\label{eq:fourradicaltower}
 \Sel_{2^j}(E)\simeq
 (\Z/2^j\Z)^2\oplus(\Z/2^{\min(j,2)}\Z)^2
 \oplus(\Z/2\Z)^{s-2}.
\end{equation}
For $j\ge3$, the image in the pure $2$-Selmer group is the
Mordell--Weil subspace.
\end{proposition}

\begin{proof}
Write
\[
 \Sh(E)[2^\infty]\simeq(\Q_2/\Z_2)^c
 \oplus\bigoplus_i(\Z/2^{e_i}\Z)^2.
\]
A paired block with $e_i=1$ contributes two to the ordinary pairing rank,
whereas a block with $e_i\ge2$, the Mordell--Weil group and the divisible
part contribute to its radical. Hence
\[
 4=\rank E(\Q)+c+2\#\{i:e_i\ge2\}.
\]
A nonzero next pairing forces at least one paired block with $e_i=2$.
Together with $\rank E(\Q)\ge2$ this gives rank two, $c=0$, and exactly one
order four pair. The ordinary rank $s-4$ supplies the remaining order two
pairs. Lemma~\ref{lem:split} gives \eqref{eq:fourradicaltower}; for
$j\ge3$ the finite order four part maps trivially to $\Sh[2]$, leaving the
Mordell--Weil plane.
\end{proof}

\begin{theorem}
\label{thm:ranktwo42691}
Let
\[
 n=11\cdot3881=42691,\qquad m=5n=213455,\qquad E=A_m.
\]
In the basis $(\Lambda,\Lambda',U,V)$ the ordinary Cassels matrix is zero
and the next matrix is
\begin{equation}\label{eq:next42691}
 \begin{pmatrix}
 0&1&0&1\\
 1&0&0&1\\
 0&0&0&0\\
 1&1&0&0
 \end{pmatrix}.
\end{equation}
Moreover
\[
 \rank E(\Q)=2,\qquad
 \Sh(E)[2^\infty]\simeq(\Z/4\Z)^2,
\]
and for every $j\ge1$,
\begin{equation}\label{eq:tower42691}
 \Sel_{2^j}(E)\simeq
 (\Z/2^j\Z)^2\oplus(\Z/2^{\min(j,2)}\Z)^2
 \oplus(\Z/2\Z)^2.
\end{equation}
The image in $S=\Sel_2^0(E)$ is $S$ for $j=1,2$ and is
$\langle U,\Lambda+\Lambda'+V\rangle$ for every $j\ge3$.
\end{theorem}

\begin{proof}
Using the norm generator $4+\sqrt5$ of $11$ and the root $1070$ of $5$
modulo $3881$, one finds
\[
 \leg{1074}{3881}=1,\qquad
 5^{970}\equiv11^{970}\equiv1\pmod{3881}.
\]
Thus $(\beta,\tau,\gamma)=(0,0,0)$, so the ordinary matrix is $0_4$.
Choose
\[
 42691=636^2-5\cdot269^2,\qquad k=1.
\]
The conic
\[
 c^2=1981u^2+9050uv+9905v^2
\]
has the primitive point $(14,-5,49)$. Since $(49/5)=1$,
Theorem~\ref{thm:largeradical} gives
$\mathcal B(\Lambda',\Lambda)=1$, so the next pairing is nonzero and the
rank is at most two.

The rational points
\[
 P=(455455,379879500),\qquad
 Q=\left(-\frac{7188005}{9},\frac{5700945250}{27}\right)
\]
have Kummer triples
\[
 \delta(P)=(55,5,11),\qquad
 \delta(Q)=(-5,-11,55).
\]
Modulo the rational $2$-torsion classes their pure images are
$[P]=U$ and $[Q]=\Lambda+\Lambda'+V$, hence are independent.
Good reduction at $3$ and $7$ gives
$\#E(\F_3)=\#E(\F_7)=4$, so the rational torsion is exactly
$(\Z/2\Z)^2$. Thus the two points give rank at least two, hence exactly
two. Proposition~\ref{prop:fourradicalcompletion} proves the group and
tower assertions. The two rational classes span the whole next radical;
therefore the $U$ row is zero and pairing
$\Lambda+\Lambda'+V$ with $\Lambda$ and $\Lambda'$ forces the two remaining
entries in \eqref{eq:next42691} to be one.
\end{proof}

\begin{corollary}\label{cor:triangle42691}
The integer $426910$ is $3/5$ congruent. One rational triangle has sides
\[
 \frac{151800}{91},\qquad
 \frac{353171}{138},\qquad
 \frac{25764721}{12558},
\]
the angle between the first two sides has cosine $3/5$, and its area is
$1707640=4\cdot426910$.
\end{corollary}

\begin{theorem}
\label{thm:ranktwo434531}
Let
\[
 p=571,\quad q=761,\quad n=434531,\quad m=2172655,\quad E=A_m.
\]
In the basis $(\Lambda,\Lambda',U,V)$ the ordinary matrix is zero, while
\begin{equation}\label{eq:next434531}
 C^{\mathrm{next}}=
 \begin{pmatrix}
 0&0&0&0\\
 0&0&0&1\\
 0&0&0&1\\
 0&1&1&0
 \end{pmatrix}.
\end{equation}
Moreover
\[
 \rank E(\Q)=2,\qquad
 \Sh(E)[2^\infty]\simeq(\Z/4\Z)^2,
\]
and the abstract groups $\Sel_{2^j}(E)$ are again those in
\eqref{eq:tower42691}. Their images in $S=\Sel_2^0(E)$ are $S$ for
$j=1,2$ and $\langle\Lambda,\Lambda'+U\rangle$ for $j\ge3$.
\end{theorem}

\begin{proof}
A root of $5$ modulo $761$ is $183$, and the ordinary fixed characters
are $(\beta,\tau,\gamma)=(0,0,0)$; explicitly
\[
 \leg{207}{761}=1,\qquad
 5^{190}\equiv571^{190}\equiv1\pmod{761}.
\]
Thus $S$ has dimension four and the ordinary matrix is $0_4$.
Use
\[
 434531=716^2-5\cdot125^2,\qquad k=1.
\]
The two conics have primitive integral points
\[
 P_0=(-8,9,153),\qquad Q_0=(73075,-31571,892),
\]
where
\[
 \Gamma_n:\quad c^2=1341u^2+8410uv+6705v^2
\]
and
\[
 \Gamma_0:\quad 2773X^2+10910XY+13865Y^2=4345310Z^2.
\]
Since $(153/5)=-1$, one has
$\mathcal B(\Lambda',\Lambda)=0$. In Theorem~\ref{thm:auxcrossU},
$s=578$, $z=378$ and $\mu=18$ modulo $761$, and
$(153\cdot18/761)=1$, so $\mathcal B(\Lambda',U)=0$. On the other hand
$153\cdot892\equiv257\pmod{761}$ and $(257/761)=-1$; hence
Theorem~\ref{thm:auxcrossV} gives
\[
 \mathcal B(\Lambda',V)=1.
\]
Thus the second pushout detects a nonzero next pairing in a case where the
two previously available entries both vanish.

The rational points
\[
 P=\left(\frac{11597645}{4},\frac{39534985875}{8}\right),
 \qquad Q=(839002500,24396051246000)
\]
have pure Kummer images $\Lambda$ and $\Lambda'+U$, respectively, and
are independent. Again $\#E(\F_3)=\#E(\F_7)=4$, so the rational torsion
is exactly the full rational $2$-torsion. Hence the rank is at least two;
the nonzero next pairing bounds it by two.
Proposition~\ref{prop:fourradicalcompletion} gives the claimed $\Sh$ and Selmer groups.
The plane $\langle\Lambda,\Lambda'+U\rangle$ is the whole next radical,
which forces \eqref{eq:next434531}.
\end{proof}

\begin{corollary}
\label{cor:onebit8selmer}
On a four dimensional zero ordinary matrix locus with
$(e,c_*,d_*)=(0,0,1)$ in the notation of
Remark~\ref{rem:fourdimnextmatrix}, put
$a_*=\mathcal B(\Lambda,U)$. Then
\[
 \Sel_8(E)\simeq
 \begin{cases}
 (\Z/4\Z)^4\oplus(\Z/2\Z)^2,&a_*=1,\\
 (\Z/8\Z)^2\oplus(\Z/4\Z)^2\oplus(\Z/2\Z)^2,&a_*=0.
 \end{cases}
\]
If $a_*=1$, then
$\rank E(\Q)=0$ and
$\Sh(E)[2^\infty]\simeq(\Z/4\Z)^4$.
If $a_*=0$, the next pairing has rank two; the displayed abstract
$8$-Selmer group is determined, although the precise two dimensional
$8$-Selmer image still depends on the remaining matrix entries.
\end{corollary}

\begin{proof}
If $a_*=1$, the Pfaffian is one by
Remark~\ref{rem:fourdimnextmatrix}, so the next pairing is nondegenerate.
This forces rank and divisible $2$-primary $\Sh$ to vanish and gives two
paired order four blocks. If $a_*=0$, the next pairing has rank exactly
two, accounting for one paired order four block. The two residual radical
dimensions arise either from Mordell--Weil/divisible contributions or from
a paired block of exponent at least eight; at level eight either case gives
exactly two cyclic factors of order eight. The rational $2$-torsion gives
the final two order two factors. Apply Lemma~\ref{lem:split}.
\end{proof}

For reference, among the $214$ two prime parameters with ordinary matrix
$0_4$ and $pq\le5\cdot10^6$, the three explicit entries
$(e,c_*,d_*)$ have the following distribution:
\[
\begin{array}{c|rrrrrrrr}
(e,c_*,d_*)&000&001&010&011&100&101&110&111\\\hline
\text{count}&26&23&25&20&23&35&32&30.
\end{array}
\]
The $V$-entry detects $23$ cases missed by the first two entries. Hence the
next pairing is known to be nonzero in $188$ of the $214$ examples. In the
remaining $26$ cases the $\Lambda'$ row vanishes, although the full next
pairing may still be nonzero. The table is included only as finite evidence,
not as a density statement.

\subsection{One dimensional radicals and all finite levels}\label{subsec:corankone}

A one dimensional ordinary radical is considerably more rigid: no further
covering is needed to determine the abstract Selmer groups at every finite
$2$-power level.

\begin{lemma}\label{lem:split}
For an elliptic curve $E/\Q$ and every $a\ge1$, the Kummer subgroup in
$\Sel_{2^a}(E)$ is a direct summand as an abstract finite abelian group.
For $E=A_m$ this gives a noncanonical isomorphism
\begin{equation}\label{eq:split}
 \Sel_{2^a}(E)\simeq
 (\Z/2^a\Z)^{\rank E(\Q)}\oplus(\Z/2\Z)^2\oplus\Sh(E)[2^a].
\end{equation}
No compatible canonical splitting of the towers is asserted.
\end{lemma}
\begin{proof}
Let $S=\Sel_{2^a}(E)$ and $K=\delta_{2^a}(E(\Q)/2^aE(\Q))$.
For $1\le j\le a$, the map induced by
$[2^{a-j}]:E[2^a]\to E[2^j]$ sends $\delta_{2^a}(P)$ to $\delta_{2^j}(P)$.
If $\delta_{2^a}(P)\in2^jS$, its image in $\Sel_{2^j}(E)$ is zero,
so $P\in2^jE(\Q)$. Therefore $K\cap2^jS=2^jK$ for every $j$.
Thus $K$ is a pure subgroup of the finite abelian $2$-group $S$ and is
a direct summand. The finite group fact follows, for example, by selecting
cyclic generators in decreasing order of their orders: purity ensures
that the generators chosen in $K$ extend at each order to generators of
$S$. The quotient is $\Sh(E)[2^a]$. The $2$-primary rational torsion of
$A_m$ is $(\Z/2\Z)^2$, proving \eqref{eq:split}.
\end{proof}

\begin{theorem}\label{thm:tower}
Let $E=A_m$, let $s=\dim\Sel_2^0(E)$, and suppose that the ordinary Cassels
matrix has rank $s-1$. Then $s$ is odd, and for every $a\ge1$,
\begin{equation}\label{eq:tower}
 \Sel_{2^a}(E)\simeq
 \Z/2^a\Z\oplus(\Z/2\Z)^{s+1}.
\end{equation}
For every $a\ge2$, the image in the pure $2$-Selmer group is the same
one dimensional ordinary radical. In particular no further $2$-power
descent shrinks this image.
If $r_E=\rank E(\Q)$, then $r_E\in\{0,1\}$ and
\begin{equation}\label{eq:shadichotomy}
 \Sh(E)[2^\infty]\simeq
 (\Q_2/\Z_2)^{1-r_E}\oplus(\Z/2\Z)^{s-1}.
\end{equation}
This does not unconditionally distinguish $r_E=1$ from a rank zero
curve with a nonzero divisible $2$-primary part of $\Sh$.
\end{theorem}
\begin{proof}
The $2$-primary group $\Sh(E)[2^\infty]$ is cofinitely generated because
its subgroup killed by $2$ is finite. Its maximal divisible subgroup
has the form $(\Q_2/\Z_2)^c$, and the finite quotient carries the perfect
alternating Cassels--Tate pairing. Hence, as an abstract group,
\[
 \Sh(E)[2^\infty]\simeq(\Q_2/\Z_2)^c
 \oplus\bigoplus_{j=1}^{t}(\Z/2^{e_j}\Z)^2,\qquad e_j\ge1.
\]
The divisibility criterion in \cite[Theorem 3.1]{Fisher} identifies the
kernel of the full pairing with the divisible subgroup; it does not
require that $c=0$.
A paired block with $e_j=1$ contributes rank two to the ordinary pairing
on $\Sh[2]$. A block with $e_j\ge2$ contributes two radical dimensions.
The divisible part and the Mordell--Weil group contribute only to the radical.
Consequently
\[
 s=r_E+c+2t,\qquad
 \dim\rad C=r_E+c+2\#\{j:e_j\ge2\}.
\]
The left side of the second equality is one. Thus $r_E+c=1$, every
$e_j=1$, and $2t=s-1$. This proves \eqref{eq:shadichotomy}.
Applying Lemma~\ref{lem:split}, the rank and divisible parts together
contribute exactly one cyclic factor of order $2^a$, proving
\eqref{eq:tower}. Under the natural map to $\Sel_2^0(E)$, the finite order two
$\Sh$ factors are killed for $a\ge2$, whereas the Mordell--Weil/divisible
contribution maps onto the ordinary radical. The assertion about the image follows.
\end{proof}

\begin{corollary}\label{cor:towerdensity}
Fix $p\equiv11,19\pmod{40}$ and let $q$ vary in $\mathcal P_p$.
Use the five characters $(\beta,\tau,\gamma,\upsilon,\rho)$ of \eqref{eq:fourcompletebits}.
Outside
\[
 \beta=0,\qquad\rho=\gamma=\upsilon,
\]
one has, simultaneously for every $a\ge1$,
\begin{equation}\label{eq:densitytower}
 \Sel_{2^a}(A_{-10pq})\simeq\Z/2^a\Z\oplus(\Z/2\Z)^4.
\end{equation}
This entire tower pattern has relative natural density $7/8$ in
$\mathcal P_p$ (absolute prime density $7/256$).
For each single $a\ge2$ the group in \eqref{eq:densitytower} also has
exactly that density. For each $a\ge2$, its image in the pure $2$-Selmer group is the line
\[
 \left\langle(\gamma+\upsilon)W+(\rho+\beta+\gamma)U+\beta V\right\rangle.
\]
\end{corollary}
\begin{proof}
The ordinary fourth matrix is zero exactly on the excluded locus and
otherwise has rank two on a three dimensional pure group.
Apply Theorem~\ref{thm:tower} with $s=3$ and the matrix/distribution results
of Theorem~\ref{thm:fourthcassels} and Corollary~\ref{cor:fourcompleteranks}.
Conversely, for $a\ge2$, the group in \eqref{eq:densitytower} has subgroup
killed by $4$ of order $64$. By \eqref{eq:split} this order is also the
order of $\Sel_4$; the zero ordinary matrix instead gives order $256$.
Thus the excluded locus cannot have this $2^a$-Selmer group.
Thus the assertion for all levels comes from a single ordinary rank condition,
rather than from an intersection of separate density statements.
\end{proof}

\begin{corollary}\label{cor:fourtowers}
On the relative density $1/16$ locus of Corollary~\ref{cor:threejoint}, where the first three
ordinary matrices are nondegenerate, simultaneously for all $a\ge1$,
\[
\begin{array}{c|c}
E&\Sel_{2^a}(E)\\\hline
A_{5pq},A_{-5pq}&(\Z/2\Z)^6\\
A_{10pq}&(\Z/2\Z)^4\\
A_{-10pq}&\Z/2^a\Z\oplus(\Z/2\Z)^4.
\end{array}
\]
The first three curves have rank zero with the same finite $2$-primary
$\Sh$ groups described in Corollary~\ref{cor:threejoint}. For the fourth curve the Selmer groups are determined as displayed, while the
rank/divisible-$\Sh$ alternative in \eqref{eq:shadichotomy} remains.
\end{corollary}
\begin{proof}
For a nondegenerate ordinary pairing, rank and divisible $\Sh$ vanish
and the finite $2$-primary group is killed by $2$. Use
Lemma~\ref{lem:split}. The nonzero $U,V$ entry of the third matrix forces
the fourth matrix to have rank two, so Theorem~\ref{thm:tower} applies.
\end{proof}

\subsection{Remaining structural problems}\label{subsec:remainingproblems}

We have determined the ordinary matrices throughout \eqref{eq:family}, the
entire next pairing on the two dimensional radical locus, and one full row on
the larger zero character radical. Three natural problems remain.

\begin{problem}
\label{prob:quaternary}
Construct a Qin type positive definite quaternary quadratic form model for
the representation congruences of $F$ and $G$, with explicit integral
lattices, congruence conditions, and a proved relation between the
resulting representation numbers and $\D(n)$. In particular, determine
whether such a model explains the normalizations $\D(p)/16$ and
$\D(pq)/32$ in Theorem~\ref{thm:reciprocity} and
Conjecture~\ref{conj:pqbridge}.
\end{problem}
The intersections of quadrics used for descent in
Proposition~\ref{prop:4cover} do not provide this positive definite
representation number construction. Proposition~\ref{prop:sparse}
reduces the correction term to a sparse factorization condition, but does
not turn it into a quaternary lattice count.

\begin{problem}
\label{prob:highergoverning}
Theorem~\ref{thm:reciprocity} gives
$\varrho(p)=(-1)^{(d-1)/2}$ for $5d^2-py^2=1$.
As primes $p\equiv11\pmod{40}$ vary, decide whether there is a finite
Galois extension $K/\Q$, independent of $p$, and a conjugacy invariant
function
$f:\operatorname{Gal}(K/\Q)\to\{\pm1\}$ such that
$\varrho(p)=f(\operatorname{Frob}_p)$ for every such unramified prime.
Determine whether the limits
\[
 \lim_{X\to\infty}
 \frac{\#\{p\le X:p\text{ prime},\ p\equiv11\pmod{40},\ b_p=e\}}
   {\#\{p\le X:p\text{ prime},\ p\equiv11\pmod{40}\}}
 \qquad(e=0,1)
\]
exist, and evaluate them.
\end{problem}
The fields of Section~\ref{sec:governing} govern ordinary Cassels matrices
with $p$ fixed and the auxiliary prime varying, whereas the class fields
of Appendix~\ref{app:classfields} vary with $p$. Neither construction
settles a fixed field or a density theorem for the higher character.

\begin{problem}
\label{prob:furtherdescent}
Complete the next pairing on ordinary radicals of dimension at least four
arising from Theorem~\ref{thm:allfour}. Theorems~\ref{thm:auxcrossU}
and~\ref{thm:auxcrossV} determine the complete $\Lambda'$ row, and
Theorems~\ref{thm:ranktwo42691} and~\ref{thm:ranktwo434531} determine two
full four dimensional examples, but the remaining first argument lifts are
not known uniformly. In particular, on the locus
$(e,c_*,d_*)=(0,0,1)$ of Remark~\ref{rem:fourdimnextmatrix}, compute
$\mathcal B(\Lambda,U)$; by Corollary~\ref{cor:onebit8selmer} this single
bit decides the abstract $8$-Selmer group. In the prime family, when
$\D(p)\ne0$ and $b_p=0$, determine the later pairing that fixes the
exponent in
$\Sh(A_{5p})[2^\infty]\simeq(\Z/2^{a_p}\Z)^2$, $a_p\ge3$.
Determine joint distributions when several auxiliary primes vary. Finally,
extend the local descent, ordinary pairing, and higher pairing calculations
to the remaining squareclasses of the $3/5$ family and to the $4/5$ family
$y^2=x(x-n)(x+9n)$.
\end{problem}
Theorem~\ref{thm:tower} completely determines all finite $2$-power
Selmer groups when the ordinary radical is one dimensional, but it does
not separate Mordell--Weil rank from a possible divisible part of $\Sh$.
Theorem~\ref{thm:compositehigher} treats one explicit two dimensional
radical locus; larger radicals still require genuinely new higher
pairing calculations.

The generalized theta series for the other angles in \cite{ImShin}
carry additional local weights. Extending the representation congruences
to those weighted coefficients is separate from the four base twists of
the $3/5$ curve treated here.

\appendix

\section{Class fields and Jacobi sum descriptions}
\label{app:classfields}\label{sec:jacobi}

\subsection{Imaginary norm characters and dihedral fields}\label{subsec:imaginarydihedral}

Throughout the appendix, $p\equiv11\pmod{40}$ is prime and we retain the
norm representation and conic notation of Section~\ref{sec:higher}. The
conic data naturally define several class fields. We describe these fields
first and then compare them with the quintic Jacobi sum construction. This
provides a class field interpretation of the characters appearing in the
main text, independent of the Pell and class number proof of
Theorem~\ref{thm:reciprocity}.

\begin{proposition}\label{arith:imaginarynorm}
Let $(u,v,c)$ be a primitive integral point of $\Gamma_p$. Put
\[
 s=u^2-5v^2,\qquad
 d=k\bigl((a+b)(u^2+5v^2)+2(a+5b)uv\bigr).
\]
Then $p\nmid c$ and $5d^2=c^4+4ps^2$. Set $e=0$ if $c$ is odd,
and $e=1$ if $c$ is even, and define
\[
 C=c/2^e,\qquad S=s/4^e,\qquad D=d/4^e.
\]
These are odd integers satisfying
\begin{equation}\label{arith:quarticnorm}
 C^4+4pS^2=5D^2,\qquad \gcd(C,S)=\gcd(C,5p)=1,
 \qquad \varrho(p)=k(-1)^e\leg C5.
\end{equation}
For $\beta=C^2+2S\sqrt{-p}\in\mathcal O_M$ there are a prime
$\mathfrak q\mid5$ and an integral ideal $\mathfrak a$ such that
\begin{equation}\label{arith:idealroot}
 (\beta)=\mathfrak q\mathfrak a^2,\qquad
 \Norm\mathfrak a=|D|,\qquad [\mathfrak a]^2=[\mathfrak q]^{-1}.
\end{equation}
\end{proposition}
\begin{proof}
With the notation of \eqref{arith:trace}, put
$\eta_0=k\alpha\phi\xi^2$. Expansion and the norm give
\[
 \eta_0=\frac{c^2+d\sqrt5}{2},\qquad
 \Norm_{K/\Q}\eta_0=-ps^2.
\]
Hence $5d^2=c^4+4ps^2$.

We have $\gcd(u,v)=1$: a common prime divisor would also divide $c$.
Completing the square in the conic gives
\[
 Ac^2=(Au+Bv)^2-20pv^2.
\]
Here $A=k(a+5b)$ and $B=5k(a+b)$, so $p\nmid A$.
If $p\mid c$, then $p\mid Au+Bv$ and $p\nmid v$,
so the right side has $p$-adic valuation one, a contradiction.
Next suppose an odd prime $q$ divides $c$ and $s$.
We already know $q\ne5,p$, and $v$ is a unit modulo $q$.
Writing $r=u/v$ modulo $q$, the two equations imply
\[
 r^2=5,\qquad (a+5b)+(a+b)r=0.
\]
They force
$(a+5b)^2-5(a+b)^2=-4p\equiv0\pmod q$, a contradiction.
Thus $c$ and $s$ have no common odd prime divisor.

If $c$ is odd, then $u,v$ have opposite parity, and $s,d$ are odd.
If $c$ is even, Lemma~\ref{arith:dyadic} gives $u,v$ odd and
$v_2(c)=1$. Hence $s\equiv4\pmod8$. In the identity
$5d^2=c^4+4ps^2$, the two terms on the right have valuations four
and six, respectively, so $v_2(d)=2$.
Dividing that identity by $16^e$ proves the assertions about
$C,S,D$ and their coprimality. Since $(2/5)=-1$,
$k(c/5)=k(-1)^e(C/5)$.

To obtain the ideal factorization, note that $\Norm_{M/\Q}\beta=5D^2$. The ideals $(\beta)$ and
$(\bar\beta)$ are coprime. Indeed their common prime divisors would
lie above divisors of both $2C^2$ and $4S\sqrt{-p}$; coprimality of
$C,S$, $p\nmid C$, and oddness of the norm exclude all of these.
The prime $5$ splits in $M$ because $(-p/5)=1$. Exactly one prime
$\mathfrak q\mid5$ divides $(\beta)$, with odd valuation; every
other valuation is even by the norm identity and coprimality with
the conjugate. This yields \eqref{arith:idealroot}, with the stated norm.
\end{proof}

Genus theory shows that $\Cl(M)$ has odd order, so squaring is an
automorphism. Hence \eqref{arith:idealroot} singles out the unique square
root of $[\mathfrak q]^{-1}$. The ideal class alone, however, does not
remember the full conic character: the square rational part $C^2$ and the
factor $k(-1)^e$ in \eqref{arith:quarticnorm} retain the additional
quadratic character $(C/5)$.

The same norm data define a natural class field, without using the reciprocity
theorem. Put
\[
 \begin{gathered}
 K=\Q(\sqrt5),\quad M=\Q(\sqrt{-p}),\quad
 M'=\Q(\sqrt{-5p}),\\
 \phi=\frac{1+\sqrt5}{2},\quad \alpha=a+b\sqrt5,
 \quad \eta=-k\alpha\phi,
 \end{gathered}
\]
and set $A=k(a+5b)$. This sign makes the resulting extension unramified at $2$.

\begin{proposition}\label{arith:dihedralfield}
Let $F=K(\sqrt\eta)$ and
\[
 H=K(\sqrt{-p},\sqrt\eta).
\]
The quartic field $F=\Q(\sqrt\eta)$ has discriminant $-25p$,
and $\sqrt\eta$ has minimal polynomial $X^4+AX^2-p$.
Its normal closure is $H$,
with $\operatorname{Gal}(H/\Q)\simeq D_4$ of order eight.
The extension $H/M'$ is cyclic of degree four and unramified at all
finite places. Using $h(-5p)\equiv4\pmod8$, it is therefore
the maximal unramified abelian $2$-extension of $M'$, and
\[
 \disc(H)=(5p)^4.
\]
Each prime of $M'$ above $2$ has Frobenius of order four in
$\operatorname{Gal}(H/M')$. Thus its class generates the
$2$-primary quotient of $\Cl(M')$; the order of the full ideal
class can have an additional odd factor.
\end{proposition}

\begin{proof}
Since $\Norm_{K/\Q}\eta=-p$ and
$\Tr_{K/\Q}\eta=-A$, the stated polynomial annihilates
$\sqrt\eta$. The ideal $(\alpha)$ is one of the two distinct
primes of $K$ above $p$. In particular the valuation pattern of
$\eta$ at these primes is $(1,0)$. Hence $\eta$ is not
a square in $K$, and that its squareclass is independent of the
squareclass of $-p$, whose valuation pattern is $(1,1)$.
Moreover $\eta\notin\Q$, since
\[
 \alpha\phi=\frac{a+5b+(a+b)\sqrt5}{2}
\]
and $a+b\ne0$. Hence $[F:\Q]=4$ and $[H:\Q]=8$.

Write $R=\sqrt5$, $T=\sqrt{-p}$, and $W=\sqrt\eta$.
The conjugate of $\eta$ in $K$ is $-p/\eta$, so all roots
$\pm W,\pm T/W$ of the quartic belong to $H$. Define
\[
 \tau(R)=-R,\quad\tau(T)=-T,\quad\tau(W)=T/W,
 \qquad
 r(R)=R,\quad r(T)=-T,\quad r(W)=W.
\]
These assignments respect the defining relations and give
$\tau^4=r^2=1$, $\tau^2(W)=-W$, and
$r\tau r=\tau^{-1}$. Thus the Galois group is $D_4$.
The order four subgroup generated by $\tau$ fixes $RT$,
and its fixed field is $M'$. Hence the extension is cyclic as claimed.

The ramification is as follows. For odd primes of $K$ away
from $(\alpha)$, the element $\eta$ is a unit and a quadratic
extension obtained by its square root is unramified or split.
At $(\alpha)$ its valuation is one, so $F/K$ is tamely ramified
with relative discriminant exponent one.

At $2$, the field $K_2/\Q_2$ is the unramified quadratic
extension. Since $kb\equiv k(a+b)\equiv1\pmod4$,
\[
 \eta=-2kb-k(a+b)\phi
    \equiv 2+3\phi=(1+\phi)^2\pmod{4\mathcal O_{K_2}}.
\]
Dividing $\eta$ by $(1+\phi)^2$ gives a unit congruent to
$1$ modulo $4$. Adjoining its square root is unramified or
split: an integral generator satisfies a polynomial
$X^2-X+(1-\eta/(1+\phi)^2)/4$, whose discriminant is a unit.
It cannot split, because a square $\eta$ in $K_2$ would have
square norm in $\Q_2$, whereas $\Norm\eta=-p\equiv5\pmod8$.
Thus $F/K$ is unramified at $2$, and
\[
 \mathfrak d_{F/K}=(\alpha),\qquad |\disc(F)|=5^2p.
\]
The two real conjugates of $\eta$ have opposite signs, so $F$
has signature $(2,1)$ and $\disc(F)=-25p$.

We finish the ramification calculation by passing from $F$ to its normal
closure over $M'$. Away from $2,5,p$, all the
quadratic generators above define unramified extensions locally.
At $p$, the field $K$ splits. In either embedding into $\Q_p$,
the completion of $H$ is generated by $\sqrt{-p}$ and the square
root of a unit: use $\eta$ when its valuation is zero, and
$-\eta/p$ when its valuation is one. Its ramification index over
$\Q_p$ is therefore two, exactly that of $M'/\Q$ at $p$.

At $5$, the element $-p$ is a square in $\Q_5$. The completion
of $M'$ is therefore the same ramified quadratic field as that
of $K$. The element $\eta$ is a unit there, so adjoining
$\sqrt\eta$ is unramified. More precisely, its residue is
$2ka\in\F_5^\times$, which is a nonsquare since $ka$ is a
square modulo $5$; thus this local extension has degree two.

Finally, $-5p\equiv1\pmod8$, so $2$ splits in $M'$.
The field $\Q_2(\sqrt{-p})$ is the same unramified quadratic
field as $K_2$. The calculation above shows that the completion
of $H$ at a prime above $2$ is the unramified extension of
$\Q_2$ of degree four. This proves both the absence of ramification and the
asserted Frobenius order. There are no real places of
$M'$ to consider. The discriminant and class field assertions
now follow from the tower discriminant formula and the fact
that the $2$-part of $h(-5p)$ is exactly four.
\end{proof}

\subsection{Ray class comparison}\label{subsec:rayclasscomparison}

\begin{proposition}\label{arith:rayclassbridge}
Let $C,S,D$ and $\beta=C^2+2S\sqrt{-p}$ be as in
Proposition~\ref{arith:imaginarynorm}, and let $\mathfrak q\mid5$
be the prime occurring with odd valuation in $(\beta)$.
Then
\[
 H=M(\sqrt5,\sqrt{-\beta}).
\]
The extension $M(\sqrt{-\beta})/M$ is quadratic with relative
discriminant $\mathfrak q$. It is the unique quadratic subextension
of the ray class field of $M$ of modulus $\mathfrak q$, and its
absolute discriminant is $5p^2$.
\end{proposition}

\begin{proof}
First use the unnormalized conic point and put
$\xi=u+v\sqrt5$, $s=\Norm\xi$, and
$\beta_0=c^2+2s\sqrt{-p}=4^e\beta$.
The expansion from Proposition~\ref{arith:imaginarynorm} reads
\[
 \eta\xi^2=\frac{-c^2-d\sqrt5}{2}.
\]
In $H$, choose $w=\sqrt\eta\,\xi$ and
$w'=\sqrt{-p}\,\bar\xi/\sqrt\eta$. Then
\[
 ww'=s\sqrt{-p},\qquad
 (w-w')^2=-\beta_0,\qquad
 (w-w')(w+w')=-d\sqrt5.
\]
The integer $d$ is nonzero by $5d^2=c^4+4ps^2$.
These identities show that $M(\sqrt5,\sqrt{-\beta})\subseteq H$
and recover $w+w'$, then $w$ and $\sqrt\eta$, from the smaller
field. Hence the two fields are equal.

The ideal factorization $(\beta)=\mathfrak q\mathfrak a^2$
shows that, at odd primes, $M(\sqrt{-\beta})/M$ ramifies exactly
at $\mathfrak q$, with discriminant exponent one. At $2$, put
$\omega=(1+\sqrt{-p})/2\in\mathcal O_M$. The oddness of
$C,S$ gives
\[
 \beta=C^2-2S+4S\omega\equiv3\pmod{4\mathcal O_M}.
\]
Consequently $-\beta\equiv1\pmod4$, which proves that the
quadratic extension is unramified at $2$. This establishes the
relative discriminant statement; it also proves that the conductor
is $\mathfrak q$. Its absolute discriminant is therefore
$\disc(M)^2\Norm\mathfrak q=5p^2$.

For uniqueness, $h(-p)$ is odd, $\mathcal O_M^\times=\{\pm1\}$,
and $\mathcal O_M/\mathfrak q\simeq\F_5$. The ray class exact
sequence gives
\[
 |\Cl_{\mathfrak q}(M)|
 =h(-p)\frac{|\F_5^\times|}{|\{\pm1\}|}=2h(-p).
\]
This abelian group has a unique quotient of order two. The
quadratic extension of conductor $\mathfrak q$ already constructed
is therefore the unique one.
\end{proof}

\begin{corollary}\label{arith:raycyclotomic}
Let $R$ be the maximal subextension of $2$-power degree of the
ray class field of $M$ of modulus $5\mathcal O_M$. Then
\[
 R=H\Q(\zeta_5),\qquad
 \operatorname{Gal}(R/M)\simeq\Z/4\Z\times\Z/2\Z.
\]
In particular $[R:M]=8$ and $[R:\Q]=16$.
\end{corollary}
\begin{proof}
Since $5=\mathfrak q\bar{\mathfrak q}$ in $M$, the ray class
exact sequence identifies the kernel of
$\Cl_{5\mathcal O_M}(M)\longrightarrow\Cl(M)$ with
\[
 \frac{\F_5^\times\times\F_5^\times}
   {\langle(-1,-1)\rangle}
 \simeq\Z/4\Z\times\Z/2\Z.
\]
This kernel has order eight, and the quotient has odd order.
It is therefore the entire $2$-primary subgroup of the ray class
group, which proves $[R:M]=8$ and the Galois group assertion.

Proposition~\ref{arith:rayclassbridge} and its conjugate show that
$H/M$ is biquadratic of conductor dividing $5\mathcal O_M$.
The extension $M(\zeta_5)/M$ is cyclic quartic, unramified away
from $5$, and tamely ramified at both primes above $5$; it too
has conductor dividing $5\mathcal O_M$. Thus both fields lie in
$R$. Their intersection is $M(\sqrt5)$: it contains this field,
and cannot be quartic because one of the two degree four Galois
groups is cyclic and the other is elementary abelian. Their
compositum consequently has degree eight over $M$, and must be
$R$.
\end{proof}

\begin{remark}
The propositions above identify the norm data of the conic with an explicit
class field construction. These fields depend on $p$, so they do not provide
a fixed governing field for the varying prime problem. Nor does the order
four Frobenius at $2$ by itself determine $\varrho(p)$. The reciprocity
theorem evaluates that character by a different route, through
$\Q(\sqrt{5p})$ and two genus character sums. The fixed field question of
Problem~\ref{prob:highergoverning} therefore remains open.
\end{remark}

There is also a quintic Jacobi sum description of the same character.
Let $N=(p-1)/5$, $L=\Q(\zeta_5)$, and $\mathfrak P\mid p$.
Define the order five character $\psi:\F_p^\times\to\mu_5$ by
$\psi(x)\equiv x^{-N}\pmod{\mathfrak P}$, and set $\psi(0)=0$.
The Jacobi sum used here is
\[
 J=\sum_{x\in\F_p}\psi(x)\psi(1-x).
\]

\subsection{Jacobi sums and descent}\label{subsec:jacobidescent}

\begin{proposition}\label{jac:Jsign}
With this convention,
\begin{equation}\label{jac:Jcorrect}
 J\equiv-\binom{2N}{N}\pmod{\mathfrak P},\qquad
 \leg{J\bmod\mathfrak P}{p}=-\leg{N!}{p}.
\end{equation}
\end{proposition}
\begin{proof}
The first congruence follows from a direct power sum calculation in $\F_p$:
\[
 J\equiv\sum_x x^{4N}(1-x)^{4N}
 =\sum_{j=0}^{4N}(-1)^j\binom{4N}{j}\sum_x x^{4N+j}.
\]
The only exponent divisible by $p-1=5N$ in this range is obtained
at $j=N$, and its inner sum is $-1$. Since $N$ is even,
\[
 J\equiv-\binom{4N}{N}.
\]
But $4N\equiv-N-1\pmod p$, so
$\binom{4N}{N}\equiv\binom{-N-1}{N}=(-1)^N\binom{2N}{N}$.
This establishes the first assertion, including the sign.

For the second assertion, let $\chi(x)=(\frac xp)$ and
$I_1=\sum_{N<j\le2N}\chi(j)$. The proof of
Theorem~\ref{thm:factorial} gives $h(-5p)=2I_1$, while
Lemma~\ref{lem:h5mod8} gives $h(-5p)\equiv4\pmod8$.
Thus $I_1\equiv2\pmod4$, while $N\equiv2\pmod8$.
The number of nonresidues in $N+1,\ldots,2N$ is
$(N-I_1)/2$, an even integer. Therefore
\[
 \leg{\prod_{j=N+1}^{2N}j}{p}=1,\qquad
 \leg{\binom{2N}{N}}p=\leg{N!}p.
\]
Since $(\frac{-1}{p})=-1$, the first assertion gives the second.
\end{proof}

\begin{remark}
For $p=11$, $N=2$ and
$J\bmod\mathfrak P=5$, whereas $\binom42\equiv6\pmod{11}$.
Their quadratic characters are $+1$ and $-1$. Thus the minus sign in \eqref{jac:Jcorrect} is essential for our convention
for $J$.
\end{remark}

\begin{proposition}\label{jac:nondescent}
Neither $J$ nor $-J$ represents a squareclass in $L^\times/L^{\times2}$
coming from $K^\times/K^{\times2}$, where $K=\Q(\sqrt5)$.
\end{proposition}
\begin{proof}
The standard Jacobi sum norm identity is $J\bar J=p$; it follows
from $J(\psi,\psi)=G(\psi)^2/G(\psi^2)$ and
$|G(\psi^i)|=\sqrt p$ for $1\le i\le4$.
Proposition~\ref{jac:Jsign} gives $v_{\mathfrak P}(J)=0$, since
$\binom{2N}{N}$ is nonzero modulo $p$. Since $p$ splits completely
in $L$, the norm identity implies
\[
 v_{\bar{\mathfrak P}}(J)=1.
\]
For every $\gamma\in K^\times$, complex conjugation fixes $\gamma$,
so its valuations at $\mathfrak P$ and $\bar{\mathfrak P}$ are
equal. Multiplication by a square changes both valuations by even
integers. The unequal parities $0$ and $1$ above therefore preclude
$J=\gamma z^2$ or $-J=\gamma z^2$ with $\gamma\in K^\times$,
$z\in L^\times$.
\end{proof}

\begin{remark}\label{arith:rayJacobiobstruction}
The common ray class field $R$ of Corollary~\ref{arith:raycyclotomic}
does not contain the quadratic extension defined by the quintic
Jacobi sum. With $\eta$ as in Proposition~\ref{arith:dihedralfield}
and $L=\Q(\zeta_5)$, that corollary gives
\[
 R=L(\sqrt{-p},\sqrt\eta).
\]
Every quadratic subextension of $R/L$ therefore has a Kummer
squareclass represented by one of $-p$, $\eta$, or $-p\eta$,
all of which lie in $K$. On the other hand, Proposition~\ref{jac:nondescent} shows that the valuation
parities of the Jacobi sum $J$ at the chosen prime
$\mathfrak P$ and its complex conjugate are zero and one,
respectively. Thus neither $J$ nor $-J$ has a squareclass
represented in $K$, and
\[
 L(\sqrt J)\not\subseteq R,\qquad
 L(\sqrt{-J})\not\subseteq R.
\]

\end{remark}

Although a single Jacobi sum does not descend in squareclass, a product of two
conjugates does. Let $\sigma(\zeta_5)=\zeta_5^2$, so that
$\sigma^2$ is complex conjugation, and write $J_1=\sigma J$.

\begin{proposition}\label{jac:productdescent}
Put $z=JJ_1$, $t=\Tr_{L/\Q}J$, and
\[
 g=z+\bar z+2p\in K.
\]
Then $g$ is a nonzero totally positive algebraic integer and
\begin{equation}\label{jac:productnorm}
 z=\frac{(z+p)^2}{g},\qquad
 \Norm_{K/\Q}g=pt^2.
\end{equation}
In particular the squareclass of $JJ_1$ descends to $K$.
Nevertheless, at the specified prime $\mathfrak P$,
\begin{equation}\label{jac:productcharacter}
 \leg{J_1\bmod\mathfrak P}{p}=\leg{N!}{p},\qquad
 \leg{g\bmod\mathfrak P}{p}=-1.
\end{equation}
More generally, fix integers $e_0,e_1,e_2,e_3$ and let
\[
 Z=\prod_{j=0}^3(\sigma^jJ)^{e_j}.
\]
If $Z$ is a unit at $\mathfrak P$ and its squareclass comes from
$K^\times/K^{\times2}$, then
\begin{equation}\label{jac:monomialcharacter}
 \leg{Z\bmod\mathfrak P}{p}=(-1)^{e_0-e_2}.
\end{equation}
Thus a fixed monomial with these properties has a constant quadratic
character, independent of $p$.
\end{proposition}
\begin{proof}
Since $\sigma\psi$ reduces to $x^{3N}$ at $\mathfrak P$, the same
power sum calculation as before gives
\[
 J_1\equiv-\binom{3N}{2N}\pmod{\mathfrak P}.
\]
Both $J$ and $J_1$ are units there. Wilson's theorem gives
$(3N)!(2N)!\equiv-1\pmod p$, since $2N$ is even. Taking quadratic
characters in the last binomial expression therefore yields
\[
 \leg{J_1\bmod\mathfrak P}{p}
 =\leg{-1}{p}
  \leg{(3N)!(2N)!N!}{p}
 =\leg{N!}{p}.
\]
The identities $J\bar J=J_1\bar J_1=p$ imply
$z\bar z=p^2$ and hence $zg=(z+p)^2$. Since $z$ is a unit at
$\mathfrak P$, it cannot equal $-p$, so $g\ne0$.
Under each embedding of $K$ into $\mathbf R$, $g$ has the form
$2p+2\operatorname{Re}(z')$ with $|z'|=p$; it is positive because
it is nonzero. Also $\sigma z=pJ_1/J$, whence
\[
 \sigma g=\frac{p(J+J_1)^2}{JJ_1}.
\]
Multiplying by $g=(JJ_1+p)^2/(JJ_1)$ gives
\[
 g\sigma g
 =p\left(\frac{(JJ_1+p)(J+J_1)}{JJ_1}\right)^2
 =p\left(J+J_1+\bar J+\bar J_1\right)^2=pt^2.
\]
At $\mathfrak P$, $\bar J$ and $\bar J_1$ each have valuation one,
so $g\equiv JJ_1$. Proposition~\ref{jac:Jsign} and the first
identity of \eqref{jac:productcharacter} give $(g\bmod\mathfrak P/p)=-1$.

For the last assertion, use the exact expression
\[
 Z=p^{e_2+e_3}J^{e_0-e_2}J_1^{e_1-e_3}.
\]
Being a unit at $\mathfrak P$ forces $e_2+e_3=0$.
At $\bar{\mathfrak P}$ its valuation is $e_0+e_1$.
A squareclass coming from $K$ has equal valuation parities at
$\mathfrak P$ and $\bar{\mathfrak P}$, so $e_0+e_1$ is even.
Writing $f=(N!/p)$ now gives
\[
 \leg{Z\bmod\mathfrak P}{p}
 =(-f)^{e_0-e_2}f^{e_1-e_3}=(-1)^{e_0-e_2},
\]
as claimed, including negative integer exponents.
\end{proof}

The right side of \eqref{jac:monomialcharacter} depends only on the fixed
exponents. Consequently such monomials cannot recover the varying factorial
character, and hence cannot by themselves recover $\varrho(p)$.
Proposition~\ref{jac:Jsign} and Theorem~\ref{thm:reciprocity}
give
\begin{equation}\label{jac:remaining}
 \leg{J\bmod\mathfrak P}{p}=-\varrho(p)
 =-(-1)^{(b-1)/2}\leg{c_0}{5}.
\end{equation}
The equality follows from the Pell and genus character argument in
the main text; the corresponding Kummer extensions are distinct
by Remark~\ref{arith:rayJacobiobstruction}.

As a final numerical check, the conic, factorial, and Jacobi sum identities,
together with the normalization in Proposition~\ref{arith:imaginarynorm},
agree for all $143$ primes $p<20000$ with $p\equiv11\pmod{40}$.
A few representative values are:

\begin{center}
\begin{tabular}{rrrrrr}
\toprule
$p$&$a$&$b$&$(u_0,v_0,c_0)$&$\varrho(p)$&$(J/\mathfrak P)$\\
\midrule
11&4&1&$(-1,1,2)$&$-1$&$1$\\
131&16&5&$(-1,1,6)$&$1$&$-1$\\
211&16&3&$(-1,1,2)$&$1$&$-1$\\
251&16&1&$(-3,7,42)$&$-1$&$1$\\
491&44&17&$(3,-2,9)$&$1$&$-1$\\
\bottomrule
\end{tabular}
\end{center}

\section*{Declarations}

\medskip
\noindent{\large\bfseries Funding}\par
\medskip
Supported by the National Natural Science Foundation of China
(Nos.\ 12231009 and 11971224).

\bigskip
\noindent{\large\bfseries AI assistance}\par
\medskip
Python programs using exact integer and rational arithmetic, together with
SymPy and NumPy where appropriate, were used for finite calculations and
consistency checks. These include the lattice enumeration in the genus
calculation, local Kummer images and descent matrices, representation counts,
Pell and class number checks, and symbolic verification of identities used in
the $4$-coverings and pushout calculations. Further finite checks concerned
the residue characters, local tangent symbols, higher pairing certificates,
and the sparse correction term. ChatGPT and Claude were used for exploratory
computation, editing, and reorganization; selected calculations and symbolic
identities were also checked independently with GPT-6 or Claude. All mathematical
statements, including the divisor class identifications, pairing formulas,
governing field degrees, and density results, are proved in the text.

\end{document}